\documentclass[production,11pt]{amsart} %

\usepackage[mathscr]{eucal}
\usepackage{amsmath}
\usepackage{amsthm}
\usepackage{amssymb}
\usepackage{amsfonts} % f"urs \mathbb Alphabet 
\usepackage{mathrsfs}
\usepackage{amscd}
\usepackage{enumerate} %um Numerierungsz"ahler hochz"ahlen zu k"onnen
\usepackage{wrapfig} % f"ur Bilder
\usepackage{layout}
\usepackage{float} % fuer Positionierung von figures
\usepackage[pdftex]{graphicx, color}

\newtheorem{Theorem}[equation]{Theorem}

\newtheorem{Proposition}[equation]{Proposition}
\newtheorem{Corollary}[equation]{Corollary}
\newtheorem{Lemma}[equation]{Lemma}
\newtheorem{Definition}[equation]{Definition}
\newtheorem{Remark}[equation]{Remark}

\newtheorem{Question}[equation]{Question}

\newtheorem{Definition and Proposition}[equation]{Definition and Proposition}
\newtheorem{Definition and Theorem}[equation]{Definition and Theorem}
\newtheorem{Definition and Remark}[equation]{Definition and Remark}
\newtheorem{Definition and Lemma}[equation]{Definition and Lemma}
\newtheorem{Definition and Example}[equation]{Definition and Example}
\newtheorem{Example}[equation]{Example}

\newcommand{\draft}[1]{} % blendet alles in {...} aus, wird aufgerufen via
\newcommand{\vsp}{\vspace{2mm}}

\def\epsilon{\varepsilon}
\def\phi{\varphi}

\newcommand{\al}{\alpha}
\newcommand{\be}{\beta}
\newcommand{\ga}{\gamma}
\newcommand{\de}{\delta}
\newcommand{\ep}{\epsilon}
\newcommand{\ze}{\zeta}

\newcommand{\lam}{\lambda}

\newcommand{\si}{\sigma}

\newcommand{\om}{\omega}

\newcommand{\Ga}{\Gamma}

\newcommand{\Lam}{\Lambda}

\newcommand{\gadot}{{\dot{\ga}}}

\newcommand{\gati}{{\ti{\ga}}}

\newcommand{\omti}{{\ti{\om}}}

\newcommand{\phihat}{{\hat{\phi}}}

\newcommand{\bti}{{\ti{b}}}

\newcommand{\pti}{{\ti{p}}}
\newcommand{\qti}{{\ti{q}}}
\newcommand{\rti}{{\ti{r}}}
\newcommand{\sti}{{\ti{s}}}
\newcommand{\tti}{{\ti{t}}}

\newcommand{\vti}{{\ti{v}}}

\newcommand{\xti}{{\ti{x}}}
\newcommand{\yti}{{\ti{y}}}
\newcommand{\zti}{{\ti{z}}}

\newcommand{\Bti}{{\ti{B}}}
\newcommand{\Cti}{{\ti{C}}}

\newcommand{\Hti}{{\ti{H}}}

\newcommand{\Mti}{{\ti{M}}}

\newcommand{\Wti}{{\ti{W}}}

\newcommand{\zdot}{{\dot{z}}}

\newcommand{\vhat}{{\hat{v}}}

\newcommand{\xhat}{{\hat{x}}}

\newcommand{\Hhat}{{\hat{H}}}

\def\C{{\mathbb C}}

\def\K{{\mathbb K}}
\def\N{{\mathbb N}}

\def\R{{\mathbb R}}

\def\Q{{\mathbb Q}}
\def\Z{{\mathbb Z}}

\newcommand{\mcA}{\mathcal A}

\newcommand{\mcD}{\mathcal D}

\newcommand{\mcF}{\mathcal F}

\newcommand{\mcH}{\mathcal H}

\newcommand{\mcM}{\mathcal M}
\newcommand{\mcN}{\mathcal N}

\newcommand{\mcP}{\mathcal P}

\newcommand{\mcZ}{\mathcal Z}

\newcommand{\mcHti}{\ti{\mcH}}

\newcommand{\mcMhat}{\widehat{\mathcal M}}
\newcommand{\mcNhat}{\widehat{\mathcal N}}

\newcommand{\msC}{\mathscr C}
\newcommand{\msD}{\mathscr D}

\newcommand{\msH}{\mathscr H}

\newcommand{\IFF}{\Leftrightarrow}

\newcommand{\ti}{\tilde}
\newcommand{\x}{\times}
\newcommand{\del}{\partial}

\newcommand{\lapla}{\triangle}

\newcommand{\inter}{\pitchfork}

\newcommand{\lkl}{{<}}
\newcommand{\rkl}{{>}}

\newcommand{\dd}{\mbox{$\mathfrak d$}}

\newcommand{\beq}{\begin{equation}}
\newcommand{\eeq}{\end{equation}}
\newcommand{\beqs}{\begin{equation*}}
\newcommand{\eeqs}{\end{equation*}}

\newcommand{\betrag}[1]{\lvert #1 \rvert}

\newcommand{\mcHpr}{{\mathcal H}_{pr}}
\newcommand{\mcHprti}{{\ti{\mathcal H}_{pr}}}

\DeclareMathOperator{\Diff}{Diff}

\DeclareMathOperator{\Ext}{Ext}

\DeclareMathOperator{\Fix}{Fix}

\DeclareMathOperator{\gap}{gap}

\DeclareMathOperator{\Ham}{Ham}

\DeclareMathOperator{\Id}{Id}
\DeclareMathOperator{\Img}{Im}
\DeclareMathOperator{\Ind}{Ind}

\DeclareMathOperator{\Int}{Int}

\DeclareMathOperator{\Per}{Per}

\DeclareMathOperator{\rk}{rk}

\DeclareMathOperator{\sign}{sign}
\DeclareMathOperator{\Span}{Span}
\DeclareMathOperator{\Spec}{Spec}

\DeclareMathOperator{\Symp}{Symp}
\DeclareMathOperator{\Vol}{vol}

\newcommand{\delbar}{\bar{\del}}

\newcommand{\delti}{\ti{\del}}

\newcommand{\cpt}{compact}

\newcommand{\cnst}{constant}

\newcommand{\diffeo}{diffeomorphism}

\newcommand{\fct}{function}

\newcommand{\ham}{Hamiltonian}

\newcommand{\hyp}{hyperbolic}

\newcommand{\lagr} {Lagrangian}

\newcommand{\mf}{manifold}

\newcommand{\refintroinvar}{Theorem \ref{introinvar}}
\newcommand{\refrank}{(\ref{rank})}

\newcommand{\refmuZcomp}{Remark \ref{mu Z-comp}}
\newcommand{\refinjaroundvertex}{Remark \ref{inj around vertex}}

\newcommand{\refgluing}{Theorem \ref{gluing}}
\newcommand{\refexistenceqq}{Theorem \ref{existence q, q'}}

\newcommand{\refpositionprim}{Remark \ref{position prim}}
\newcommand{\refindexprim}{Lemma \ref{index prim}}
\newcommand{\refprimarymax}{Remark \ref{primary max}}
\newcommand{\reffiniteprimary}{Remark \ref{finite primary}}
\newcommand{\refsignskewsym}{Lemma \ref{sign skew sym}}
\newcommand{\refLsigns}{(\ref{Lsigns})}

\newcommand{\refdelfraksquare}{Theorem \ref{delfraksquare}}

\newcommand{\refhomcohom}{Theorem \ref{hom cohom}}

\newcommand{\refpropclassifind}{Proposition \ref{prop classif ind 1}}
\newcommand{\refalwaysintersections}{Lemma \ref{always intersections}}

\newcommand{\reffiniteprimiterate}{Proposition \ref{finite prim iterate}}

\newcommand{\refpropclassifprimary}{Proposition \ref{prop classif primary}}
\newcommand{\refprimcutth}{Theorem \ref{prim cut th}}

\newcommand{\refkes}{(\ref{kes})}

\newcommand{\refchapterexamples}{Chapter \ref{chapterexamples}}

\newcommand{\refinvthcpt}{Theorem \ref{inv th cpt}}
\newcommand{\refstartend}{Proposition \ref{start end}}
\newcommand{\refflipviamixed}{Lemma \ref{flip via mixed}}
\newcommand{\refmovechar}{Corollary \ref{move char}}

\newcommand{\refhomoclinicloops}{Corollary \ref{homoclinicloops}}

\newcommand{\refsecunimport}{Proposition \ref{sec unimport}}

\newcommand{\reffishformula}{Proposition \ref{fish formula}}
\newcommand{\refcompared}{Lemma \ref{compare d}}
\newcommand{\refdkron}{(\ref{d kron})}

\newcommand{\refdeffg}{Lemma \ref{def f, g}}

\newcommand{\refmdel}{(\ref{m'(del)=0})}
\newcommand{\refrsprimaryisom}{Theorem \ref{rs primary isom}}
\newcommand{\refprimmoveinvar}{Theorem \ref{prim move invar}}
\newcommand{\refcomment}{Remark \ref{comment}}

\newcommand{\refsimplemixedinvar}{Proposition \ref{simple mixed invar}}
\newcommand{\refmixedinvar}{Theorem \ref{mixed invar}}

\newcommand{\refproofstartend}{Paragraph \ref{proof start end}}

\newcommand{\refprimstab}{Lemma \ref{prim stab} (\ref{second})}
\newcommand{\refsignequal}{Lemma \ref{sign equal}}

\newcommand{\refdimleq}{Proposition \ref{dimleq}}
\newcommand{\refinfinitehom}{Proposition \ref{infinitehom}}

\newcommand{\refsemiprimhom}{Theorem \ref{semiprimhom}}

\begin{document}

\title{Homoclinic points and Floer homology}

\author{Sonja Hohloch}

\address{Sonja Hohloch \\ Ecole Polytechnique F\'ed\'erale de Lausanne \\ SB MATHGEOM CAG \\ 1015 Lausanne, Switzerland}
\email{sonja.hohloch@epfl.ch}
\urladdr{http://sma.epfl.ch/$\sim$hohloch/}

\begin{abstract}
\noindent
A new relation between homoclinic points and Lagrangian Floer homology is
presented:
In dimension two, we construct a Floer homology generated by primary homoclinic
points. We compute two examples and prove an invariance theorem. Moreover, we
establish a link to the (absolute) flux and growth of symplectomorphisms.
\end{abstract}

\maketitle

\section{Introduction}

Homoclinic points are the intersection points of the stable and unstable
manifolds of a hyperbolic fixed point. Their discovery goes back to 1889 when
Poincar\'e \cite{poincare1}, \cite{poincare2} studied the n-body problem and
came across certain nonconvergent trigonometric series. First results about the
nature of homoclinic points are due to Birkhoff \cite{birkhoff} who discovered
an intricate amount of (mostly high) periodic points near homoclinic ones. This
phenomenon was formalized by Smale's \cite{smale1} \cite{smale2} horseshoe in
the 1960's. Melnikov's \cite{melnikov} perturbation method for producing and
detecting homoclinic points also dates to the 1960's. Since the 1970's,
genericity properties of homoclinic points were studied by several authors. But
in spite of these achievements, there are still many open questions.

\vsp

Floer theory is a much more recent development. Floer \cite{floer1},
\cite{floer2}, \cite{floer3} devised it in the late 1980's when he worked on the
Arnold conjecture. Arnold conjectured around 1960 that the number of fixed
points of a nondegenerate Hamiltonian diffeomorphism on a closed symplectic
manifold is greater or equal to the sum over the Betti numbers. 
Floer proved Arnold's conjecture first on closed symplectic manifolds with
$\pi_2(M)=0$ and then on so-called monotone manifolds. After his breakthrough,
the conjecture was established on more general closed symplectic manifolds by a
series of authors, cf.\ for references e.g.\ Salamon \cite{salamon}.

Floer theory is some kind of infinite dimensional Morse theory for the
symplectic action functional as a Morse function. 
It is vividly studied nowadays and has many applications not only in symplectic
geometry and dynamical systems. The first version of Floer homology was
Lagrangian Floer homology. Roughly speaking, its chain groups are generated by
the intersection points of two Lagrangian submanifolds. The grading is induced
by the Maslov index. The boundary operator counts flow lines of the negative
$L^2$-gradiant flow of the symplectic action functional between intersection
points with Maslov index difference one.

\vsp

The present paper is motivated by the fact that the stable and unstable manifold
of a hyperbolic fixed point of a symplectomorphism are Lagrangian submanifolds.
In such a case the homoclinic points can be considered as intersection points of
a Lagrangian intersection problem for which one might hope to define a
Lagrangian Floer homology. The main obstacle is the wild oscillation and
accumulation behaviour of the noncompact Lagrangians. Classical Lagrangian Floer
homology is defined for compact Lagrangian submanifolds and can be generalized
to `nice' noncompact ones. But those techniques fail in the present situation. 

In order to actually {\em count} connecting flow lines one needs compactness of
the associated 0-dimensional moduli spaces. If the dimension of the manifold is
greater than two this turns out to be a tricky analysis problem about Gromov
compactness of spaces of pseudo-holomorphic curves.

But on two-dimensional manifolds, the analysis can be replaced by combinatorics
and counting of certain orientation preserving immersions as shown by de Silva
\cite{de silva}, Fel'shtyn \cite{felshtyn} and Gautschi $\&$ Robin $\&$ Salamon
\cite{gautschi-robbin-salamon}. 
Since a symplectic form is a nondegenerate, closed 2-form the notions of `volume
preserving' and `symplectic' coincide and symplectomorphisms and volume preserving
diffeomorphisms are the same.

The wild behaviour and the noncompactness of the (un)stable manifolds prevent
the analysis ansatz. Therefore we will work in a two-dimensional setting.
Nevertheless, the set of homoclinic points is still too large to allow a
well-defined and meaningful Floer homology as analysed in Hohloch
\cite{hohloch}. Our generator set will be the set of so-called primary
homoclinic points which are defined by very rigid geometric properties.

\vsp

In the following, $(M, \om)$ stands for $\R^2$ or a closed surface with genus $g
\geq 1$ with their resp.\ volume forms. Let $\phi$ be a symplectomorphism on $M$
with hyperbolic fixed point $x$. Denote the associated stable resp.\ unstable
manifolds by $W^s:=W^s(x, \phi)$ resp.\ $W^u:=W^u(x, \phi)$ and set $\mcH:=W^s
\cap W^u$ to be the set of homoclinic points. Given $p$, $q \in \mcH$, we call
$[p,q]_s \subset W^s$ and $[p,q]_u \subset W^u$ the stable resp.\ unstable
segment between $p$ and $q$. 
We call $p$ contractible if the loop $[p,x]_s \cup [p,x]_u$ is contractible in
$M$. 
We denote by $\mathcal H_{[x]} \subset \mathcal H$ the set of contractible
homoclinic points and call
\begin{equation*}
\mathcal H_{pr}:=\{p \in \mathcal H_{[x]}\backslash\{x\} \mid\ ]p,x[_s\ \cap \
]p,x[_u\ \cap\ \mathcal H_{[x]} = \emptyset \}
\end{equation*}
the set of {\em primary} points. 
$\phi$ induces a $\Z$-action on $\mcH$ via $\Z  \x \mcH \to \mcH$, $(n,p)
\mapsto \phi^n(p)$. 
$\tilde{\mathcal H}_{pr}:=\mathcal H_{pr} \slash \mathbb Z$ is finite and we
denote the equivalence class or orbit of $p$ by $\langle p \rangle$. 
The Maslov index $\mu$ induces a grading $\mu:\tilde{\mathcal H}_{pr} \to
\mathbb Z$. We define $m(p,q)$ to be the number of certain orientation
preserving immersions with start point $p$ and end point $q$ and set $m(\langle
p \rangle, \langle q \rangle):= \sum_{n \in \Z} m(p,\phi^n(q))$.

\begin{Theorem}
The groups and operator
\begin{align*}
C_k(x, \phi):=\bigoplus_{\stackrel{\langle p \rangle \in \tilde{\mathcal
H}_{pr}}{\mu(\langle p \rangle)=k}} \mathbb Z \langle p \rangle \quad \mbox{and}
\quad
 \partial \langle p \rangle := \sum_{\stackrel{\langle q \rangle \in
\tilde{\mathcal H}_{pr}}{\mu(\langle q \rangle )= \mu(\langle p \rangle)-1}}
m(\langle p \rangle , \langle q \rangle) \langle q \rangle
\end{align*}
with $k \in \Z$ form a chain complex, i.e. $\del \circ \del =0$, and we call the
resulting homology $ H_*:=H_*(x, \phi):= \frac{\ker \partial_*}{\Img
\partial_{*+1}}$ {\em primary Floer homology}. $C_k(x, \phi)$ and thus $H_k(x,
\phi)$ vanish for $k \notin \{\pm 1, \pm 2, \pm 3\}$.
\end{Theorem}

The well-definedness of $\partial$ and the proof of $\partial \circ \partial =0$
are tricky combinations of dynamical and combinatorial arguments.

\vsp

$H_*$ is invariant under so called {\em contractibly strongly intersecting
(symplectic) isotopies} (defined later before \refinvthcpt):

\begin{Theorem}
\label{introinvar}
Let $(M,\om)$ be a closed symplectic two-dimensional \mf\ with genus $g \geq
1$. 
Let $\phi$ and $\psi$ be symplectomorphisms with \hyp\ fixed points $x \in
\Fix(\phi)$ and $y \in \Fix(\psi)$. Let $(x, \phi)$ and $(y, \psi)$ be csi and
let all primary points of $\phi$ and $\psi$ be transverse.
Assume there is a csi isotopy $\Phi$ from $(x, \phi)$ to $( y, \psi)$.
Then
\beqs
H_*(x,\phi) \simeq H_*(y, \psi).
\eeqs
\end{Theorem}

The proof has to combine analytical {\em and} combinatorial arguments since a
primary point $p \in \mathcal H_{pr}$ might vanish (analogously arise) in two
ways:

\quad (i) $p$ vanishes as intersection point {\em or}

\quad (ii) $p$ persists as intersection point, but is no longer primary.

\vsp 

The invariance implies an existence and bifurcation criterion for homoclinic
points and the fixed point.
Conjecturally Hamiltonian isotopies are naturally strongly intersecting.
Moreover, \refintroinvar\ allows to classify homoclinic tangles up to csi
isotopy.

\vsp

There are several versions of homoclinic Floer homology with quite different
flavours. Their well-definedness can be easily deduced from the construction and
well-definedness of primary Floer homology.

\vsp

One version is {\em chaotic Floer homology} $\Hhat_*$ which takes into account
the periodic points `near' a homoclinic tangle. More precisely, we have a whole
sequence $n \mapsto \Hhat_*(x,\phi^n)$ where $n$ is the number of iterates of
the symplectomorphism. For fixed $n$, the boundary operator associated to
$\Hhat_*(x, \phi^n)$ counts only those connecting immersions whose range does
not contain any fixed points of $\phi^n$, i.e. $n$-periodic points of $\phi$.
This leads to an interesting behaviour of the sequence $n \mapsto \Hhat_*(x,
\phi^n)$ and the definition of a symplectic zeta function
\beqs
\ze_{x,\phi}(z):= \exp\left(\sum_{n=1}^\infty \frac{\chi(\Hhat_*(x,\phi^n))}{n}
z^n\right )
\eeqs
where $\chi(H^{\Fix}_*(x,\phi^n))$ denotes the Euler characteristic of
$\Hhat_*(x,\phi^n)$. The study of this function is an ongoing project.

\vsp

An important question in symplectic dynamics is the growth behaviour of
symplectomorphisms under iteration. Among others, it has been studied
extensively by Polterovich \cite{polterovich3}, \cite{polterovich4}. By means of
the growth behaviour, group theoretic question about the group of Hamiltonian
diffeomorphisms can be answered. In \cite{polterovich4}, a Hamiltonian version
of the Zimmer program is devised and the proofs relie on the growth behaviour of
iterated Hamiltonian diffeomorphisms.

If we want to study the iteration behaviour of symplectomorphisms by means of
primary Floer homology we easily find 
\begin{equation}
\label{rank}
\rk H_*(x, \phi) \leq \rk H_*(x, \phi^n).
\end{equation}
For Hamiltonian diffeomorphisms, equality in \refrank\ turns out to be
equivalent to proving \refintroinvar\ for Hamiltonian diffeomorphisms. 
Thus the strong invariance in \refintroinvar\ seemingly opposes easy examples
with 
$\rk H_*(x, \phi) < \rk H_*(x, \phi^n)$.

Another version of homoclinic Floer homology, namely {\em semi-primary Floer
homology} $\Hti_*$, is more apt for detecting increasing rank. Instead of
primary points, its chain complex is generated by so-called semi-primary points.
$\Hti_*$ has weaker invariance properties than primary Floer homology. It is
easy to find examples with
\beqs
\Hti_*(x, \phi) < \Hti_*(x, \phi^n).
\eeqs
Increasing rank of semi-primary Floer homology actually means that parts of the
tangle wrap in a certain way around some genus of the surface. 

\vsp

This line of thoughts has been completed in Hohloch \cite{hohloch2}. In that work, primary Floer homology on $\R^2$ and so-called
{\em cylinder Floer homology} $\msH_*$ (a variant of semi-primary Floer
homology) on the infinite cylinder $\mcZ$ are studied using the filtration by
the symplectic action $\mcA$. The action interval of the filtered groups appears
as an upper index. The action spectrum is denoted by $\Spec(x, \phi)$ and the
minimal distance between two action levels by $\gap(x, \phi)$. The boundary
operator is modified in such a way that we keep track of the homotopy class on
$\mcZ$. 

\begin{Theorem}[\cite{hohloch2}]
Let $\phi \in \Symp(\R^2)$ resp.\ $\phi \in \Ham^c(\mcZ)$.
Let $b \in \Spec(\phi, x)$ and $0 < \ep \leq \frac{1}{2} \gap(\phi, x)$. 
Assume that there are $k$ primary classes with action $b$.
Then we obtain for the homoclinic Floer homology on $\R^2$ resp.\ $\mcZ$
\begin{gather*}
H^{]b - \ep, b+ \ep]}_*(\phi, x) \simeq \Z^k \quad \mbox{and} \quad
H^{]b - \ep, b+ \ep]}_*(\phi^n, x) \simeq (\Z^{k})^n, \\
\msH^{]b - \ep, b+ \ep]}_*(\phi, x)  \simeq \Z^k \quad \mbox{and} \quad
\msH^{]b - \ep, b+ \ep]}_*(\phi^n, x) \simeq (\Z^{k})^n.
\end{gather*}
Thus the rank grows linearly with the number of iterations.
\end{Theorem}

Homoclinic Floer homology is also linked to transport and (absolute) flux of a
dynamical system. In the symplectic plane $(\R^2, \om)$, the {\em absolute flux}
(briefly flux) of a symplectomorphism $\phi$ through a simply closed curve $c$
is defined as the volume of the set of points which are swept out of the
interior of the curve, i.e. 
\beqs
\mcF lux_\phi(c)=\Vol_\om(\phi(\Int(c)) \cap \Ext(c)).
\eeqs
It also can be defined for noncontractible curves on the cylinder.
Note that this notion is different from the flux homomorphism in symplectic
geometry (cf.\ \cite{mcduff-salamon-sympl}, \cite{polterovich2}) which, roughly
speaking, considers the difference between $\phi(\Int(c)) \cap \Ext(c)$ and
$\phi(\Ext(c)) \cap \Int(c)$.

The absolute flux has been studied by MacKay $\&$ Meiss $\&$ Percival
\cite{mackay-meiss-percival} in order to gain a better understanding of the
transport. In their setting, transport means the motion of points unter (many)
iterations. Invariant curves have zero flux and are therefore complete barriers
for the transport. MacKay $\&$ Meiss $\&$ Percival \cite{mackay-meiss-percival}
focus on partially invariant curves associated to cantori, homoclinic and
periodic points. In this case, transport is only possible through the
non-invariant part of the curve. The non-invariant part forms a so-called {\em
turnstile}. We generalize this notion in \cite{hohloch2} and distinguish between
{\em true, overtwisted} and {\em generalized turnstiles}. Overtwisted turnstiles
correspond to {\em mixed moves} and generalized turnstiles to {\em primary
moves} in the prove of \refintroinvar. 

\begin{Proposition}[\cite{hohloch2}]
True and overtwisted turnstiles are annihilated by the boundary operator.
\end{Proposition}

In our setting, the Maslov index and the (relative) action are invariant under
iteration of the symplectomorphism. The flux through a homoclinic orbit $\langle
p \rangle$ is defined as the flux through a curve parametrizing $[p,x]_s \cup
[p,x]_u$. It transforms
\beqs
\mcF lux_{\phi^n}(\langle p \rangle)= n \mcF lux_\phi(\langle p \rangle)
\eeqs
as the Maslov index and action in classical Floer theory, see Ginzburg $\&$
G\"urel \cite{ginzburg-gurel}.
MacKay $\&$ Meiss $\&$ Percival \cite{mackay-meiss-percival} identified the flux
(under certain assumptions) with Mather's \cite{mather1} {\em difference in
action $\lapla W$}. If the primary points $p$ and $q$ form a true turnstile and
$v \in \mcM(p,q) \neq \emptyset$ we extend their result to

\begin{Theorem}[\cite{hohloch2}]
Under certain assumptions holds
\beqs
\mcA(\langle p \rangle )-\mcA( \langle q \rangle )=\mcA(\langle p \rangle ,
\langle q \rangle)= \int_v \om =\mcF lux_\phi(\langle p \rangle)=\lapla W_{p,q}.
\eeqs
\end{Theorem}

Altogether, $H_* $ is the first invariant which takes the {\em algebraic}
interaction of homoclinic points into account and links them to dynamical
properties like the absolute flux and growth of symplectomorphisms. 

\subsection*{Acknowledgments}

The author wishes to thank G. Noetzel, M. Schwarz, Z. Xia and E. Zehnder for
helpful discussions, explanations and references.

\section{Immersions, cutting and gluing}

\subsection{Maslov index and homotopy class}

In this subsection, we recall the definition of the Maslov index for Lagrangian
subspaces in $R^{2n}$ as it can be found in McDuff $\&$ Salamon
\cite{mcduff-salamon-sympl}. Using suitable trivialitations, Floer \cite{floer1}
generalized it to symplectic manifolds. Finally, we introduce crucial notations
like (un)stable segments and homotopy classes for homoclinic points.

\vsp

Denote by $\mathcal L(n)$ the space of Lagrangian subspaces of $(\R^{2n},
\om_0)$ with $\om_0:= \sum_{i=1}^n dx_i \wedge dy_i$.
Represent $\Lam \in \mathcal L(n)$ by $\Lam= \left( \begin{smallmatrix} X \\ Y\end{smallmatrix} \right)$ with $U:= X +i Y \in
U(n)$ and define
$\rho : \mathcal L(n) \to S^1$, $ \rho(\Lam):= \det (U \circ U)$.
For a loop $\Lam : \R \slash \Z \to \mathcal L(n)$, define the {\em Maslov index
of loops of \lagr\ subspaces} by $\mu(\Lam):= \deg (\rho \circ \Lam)$
where $\deg$ denotes the mapping degree of $\rho \circ \Lam : \R \slash \Z \to
S^1$. If $\al : \R \to \R$ is a lift of $\rho \circ \Lam$, i.e. $\det
(X(t)+iY(t))= e^{i\pi \al(t)}$,
we obtain
$\mu(\Lam)= \al(1)-\al(0)$.

\vsp

Let $(M, \omega)$ be a $2n$-dimensional symplectic manifold and $\phi$ a
symplectomorphism with hyperbolic fixed point $x$. For symplectomorphisms, the
(un)stable manifolds $W^u:=W^u(x, \phi)$ and $W^s:=W^s(x, \phi)$ are Lagrangian
submanifolds and there are injective immersions $\ga_u: \R^n \to W^u$ and
$\ga_s: \R^n \to W^s$ with $\ga_u(0)=x=\ga_s(0)$. 
Provide $\mcP(W^u, W^s):=\{ \be: [0,1] \to M \mid \be(0) \in W^u, \be(1) \in
W^s\}$ with the smallest topology such that the following three maps are
continuous:
\begin{align*}
& \mcP(W^u, W^s) \to C([0,1];M),  && \be \mapsto \be, \\
& \mcP(W^u, W^s) \to \R, &&  \be \mapsto \ga_u^{-1}(\be(0)), \\
& \mcP(W^u, W^s) \to \R,  && \be \mapsto \ga_s^{-1}(\be(1)).
\end{align*}
Fix $\al \in \mcP(W^u, W^s)$ and denote its connected component by
$\mcP_\al(W^u, W^s)$. Identify $p$, $q \in \mcH$ as constant paths in $
\mcP_\al(W^u, W^s)$.
Let $v:[0,1] \to \mcP(W^u,W^s)$ with $v(0)\equiv p$ and $v(1) \equiv q$ and see
it as map $v: [0,1]^2 \to M$ via $v(s,t):=v(s)(t)$. 
The square $[0,1]^2$ is contractible and we can find a trivialization $\Phi:=
\Phi_v : v^*TM
\to [0,1]^2 \x \R^{2n}$ such that 
the symplectic
form on the fibers is mapped to the standard $\om_0$ on $\R^{2n} \simeq \C^n$,
$\Phi$ is \cnst\ on
$\{0\} \x [0,1]$ and on $\{1\} \x [0,1]$,
and
$\Phi(T_pW^s)=  i \Phi(T_pW^u)$ and
$\Phi(T_qW^s)= i \Phi(T_q W^u)$.

Denote by $\del [0,1]^2$ the boundary of $[0,1]^2$ and define the loop $\Lam_v:
\del [0,1]^2 \to \mathcal L(n)$ starting in $(0,0)$ and running through $(1,0)$,
$(1,1)$ and $(0,1)$ back to $(0,0)$ piecewise via
\begin{align*}
&(\xi,0)  \mapsto \Phi(T_{v(\xi,0)}W^u), && (\xi,1)  \mapsto
\Phi(T_{v(\xi,1)}W^s), \\ 
&(1,\eta)  \mapsto e^{\frac{i \pi \eta }{2}} \Phi(T_qW^u), &&
(0,\eta) \mapsto e^{ \frac{i \pi (\eta-1) }{2}} \Phi(T_pW^s).
\end{align*}
Under the above conventions, we define the {\em relative Maslov index for $p$,
$q \in \mcH$} via $\mu(p,q):=\mu(\Lam_v)$. 
If $\pi_2(M)=0$, then $c_1|_{\pi_2(M)}=0$ (where $c_1$ denotes the first Chern
class of $M$) and the construction is independent from the choosen path $v$ and
the trivialization $\Phi$. 
Concerning the two-dimensional situation, recall that the second homotopy class
of a closed surfaces with genus $g \geq 1$ always vanishes.

\vsp

From now on, $(M,\om)$ is either $(\R^2, dx \wedge dy)$ or a closed,
two-dimensional \mf\ with genus $g \geq 1$. 
For $i \in \{u,s\}$ the immersions $\ga_i: \R \to W^i$ induce an ordering $<_i$
resp.\ $\leq_i$ on $W^i$ via 
\beqs
\ga_i(t)<_i \ga_i(\tti) \ \IFF \ t<\tti \quad \mbox{resp.}\quad
\ga_i(t)\leq_i \ga_i(\tti) \ \IFF \ t\leq \tti.
\eeqs
By abuse of notation, we say that $p$, $q \in W^i$ induce an ordering on $W^i$
via setting $p<_i q$ resp.\ $p \leq_i q$.
For $i \in \{0,1 \}$ consider $p$, $q \in W^i$ and set $t_i^p=\ga_i^{-1}(p)$,
$t_i^q:=\ga_i^{-1}(q)$, $t_i^-:=\min\{t_i^p, t_i^q\}$ and $t_i^+:=\max\{t_i^p,
t_i^q\}$. We call
\beqs
[p,q]_u :=\ga_u([t_u^-, t_u^+]) \quad \mbox{ resp. } \quad [p,q]_s
:=\ga_s([t_s^-, t_s^+])
\eeqs
the {\em segments} in $W^u$ resp.\ $W^s$ between $p$ and $q$. The segments are
independent of the chosen immersion and a priori just sets of points, thus
$[p,q]_i=[q,p]_i$. Analogously, we define the open and half-open segments
$]p,q[_i$ and $[p,q[_i$.

\vsp

Now we assign to each $p \in \mcH$ a {\em homotopy class} in $\pi_0(\mcP(W^u,
W^s)) \simeq \pi_1(M,x)$: Denote by $c_p: [0,1] \to W^u \cup W^s$ a curve with
$c_p(0)=x=c_p(1)$ which runs through $[x,p]_u$ to $p$ and through $[p,x]_s$ back
to $x$.
Set $[p]:= [c_p] \in \pi_1(M, x)$ and $[-p]$ for the path with the inverse
parametrization. Then $\mcH_{[x]}:=\{p \in \mcH \mid [p] =[x] \}$ is the set of
{\em contractible} homoclinic points. $\mcH_{[x]}$ is invariant under the action
of $\phi$.
Moreover, if $\phi=\phi_1$ is the time-1 map of a flow and $\xi: S^1 \to M$,
$\xi(t):=\phi_t(x)$ and $\xi$ is contractible or $\pi_1(M,x)$ abelian then
$[p]=[\phi_1^n(p)]$ for all $p \in \mcH$ and $n \in \Z$.

\begin{Remark}
\label{mu Z-comp}
For contractible $p$, $\pti$, $q \in \mcH$, we observe: 
\begin{enumerate}[(1)]
\item
$\mu(q,p)=-\mu(p,q)$ and $\mu(p, q) + \mu(q, \pti)= \mu(p,\pti)$.
\item
$\mu(p, q)= \mu(\phi^n(p), \phi^n(q))$ for $n \in \Z$, i.e. the (relative)
Maslov index of $p$ and $q$ is invariant under the $\Z$-action of $\phi$ on
$\mcH$.  
\item
$\mu(p, \phi^n(p))=0$ for all $n \in \Z$.
\item
$\mu(p,q)=\mu(p,\phi^n(q))$ for $n \in \Z$.
\end{enumerate}
\end{Remark}

The (relative) Maslov index yields a {\em grading} $\mu: \mcH_{[x]} \to \Z $ via
$\mu(p):=\mu(p,x)$ such that for contractible $p$ and $q$ holds
$\mu(p,q)= \mu(p,x) +\mu(x,q)=\mu(p,x)- \mu(q,x)= \mu(p)-\mu(q)$.

\begin{Remark}
Let $\tau : (\Mti, \omti) \to (M, \om)$ be the universal cover with
$\tau^*\om=\omti$ and $p$, $q \in \mcH$ with $[p]=[q]$. Denote by $[\pti,
\qti]_i$ the lift of $[p,q]_i$ to the universal cover $(\Mti, \omti)$ starting
in $\pti \in\tau^{-1}(p)$. Then $\mu(p,q)=\mu(\pti, \qti)$.
\end{Remark}

%\begin{proof}
%Let $w$ and the trivialization $\Phi: w^*TM \to [0,1]^2 \x \R^{2n}$ be as
%needed for \refmupq\ and denote by $\wti$ the lift of $w$ starting in $\pti$ and
%by $\qti$ the accordingly lifted $q$. Then $\Phiti: \wti^*T\Mti \to [0,1]^2 \x
%\R^{2n}$, $\Phiti(z,v):=\Phi(z, D\tau(v)))$ is an analogous trivialization for
%$\wti$. Since $D\tau(T_{\wti}[\pti, \qti]_i)=T_{w}W^i$ for $i \in \{s,u\}$ the
%loops $\Lam_w$ and $\Lam_\wti$ coincide and thus
%$\mu(p,q)=\mu(\Lam_w)=\mu(\Lam_\wti)=\mu(\pti, \qti)$.
%\end{proof}

\subsection{Immersions, di-gons and hearts}

This subsection introduces certain di-gons, also known as 2-gons, lunes or
half-moons (Chekanov \cite{chekanov}, de Silva \cite{de silva}, Gautschi $\&$
Robbin $\&$ Salamon \cite{gautschi-robbin-salamon}, Robbin \cite{robbin}). They
will be crucial for the definition of the boundary operator of the Floer chain
complex.

\vsp 

A {\em di-gon} is the polygon $D\subset \R^2$ with two convex vertices at $(-1,
0)$ and $(1,0)$ sketched in Figure \ref{2-gon} (a). Denote its upper boundary by
$B_s$ and its lower boundary by $B_u$.

A {\em heart} is either the polygon $D_b $ of Figure \ref{2-gon} (b) or the
polygon $D_c$ of Figure \ref{2-gon} (c). A heart is characterised by two
vertices at $(-1,0)$ and $(1,0)$ where one is convex and one concave. Denote
their upper boundaries by $B_s$ and their lower boundaries by $B_u$.

\begin{figure}[H]
\begin{center}

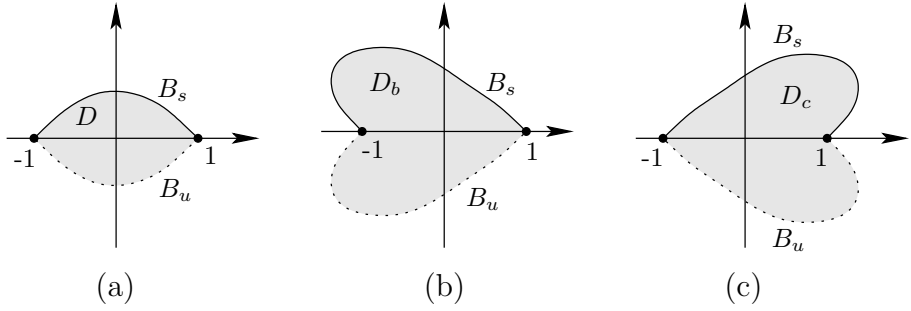

\caption{Di-gon and heart.}
\label{2-gon}

\end{center}
\end{figure}

\vsp

We require the immersions in the following definitions to be immersions also on
the boundaries and vertices.
Thus the image of a small neighbourhood of a convex (concave) vertex of a
polygon is a wedge-shaped region with angle smaller (larger) than $\pi$.

\begin{Definition}
Let $D$ be the di-gon and $p$, $q \in \mcH$ with $\mu(p,q)=1$. We define
$\mcM(p,q)$ to be the {\em space of smooth, immersed di-gons} $v : D \to M$
which are orientation preserving and satisfy
$v(B_u) \subset W^u$, $v(B_s) \subset W^s$, 
$v((-1,0))=p$ and $v((1,0))=q$.
Denote by $G(D)$ the group of orientation preserving diffeomorphisms of $D$
which preserve the vertices and call $\mcMhat(p,q):=\mcM(p,q) \slash G(D)$ the
space of unparametrized immersed di-gons.
\end{Definition}

Since there is exactly one segment $[p,q]_i$, $i \in \{s,u\}$, joining $p$, $q
\in \mcH$ and since $\pi_2(M)=0$ we deduce $\# \mcMhat(p,q) \in \{0,1\}$ for $p$
and $q$ with $\mu(p,q)=1$.

\begin{Definition}
Consider the hearts $D_b$ and $D_c$ and $p$, $q \in \mcH$ with $\mu(p,q)=2$. We
define $\mcN_b(p,q)$ resp.\ $\mcN_c(p,q)$ to be the {\em space of smooth immersed
hearts} $w: D_b \to M$ resp.\ $w: D_c \to M$ which are
orientation preserving and satisfy
$w(B_u) \subset W^u$, $w(B_s) \subset W^s$, 
$w(-1,0)=p$ and $w(1,0)=q$.
We set $\mcN(p,q):=\mcN_b(p,q) \ \dot{\cup}\ \mcN_c(p,q)$.
Denote by $G(D_b)$ resp.\ $G(D_c)$ the group of orientation preserving
diffeomorphisms of $D_b$ resp.\ $D_c$ which preserve the vertices and let
$\mcNhat_b(p,q):=\mcN_b(p,q) \slash G(D_b)$ resp.\ $\mcNhat_c(p,q):=\mcN_c(p,q)
\slash G(D_c)$ and $\mcNhat(p,q) :=\mcNhat_b(p,q) \ \dot{\cup}\ \mcNhat_c(p,q)$
be the spaces of unparametrized immersed hearts.
\end{Definition}

If we work with the spaces $\mcM(p,q)$ and $\mcN(p,r)$ we always implicitly
assume $p$, $q$, $r \in \mcH$ with $[p]=[q]$, $[p]=[r]$, $\mu(p,q)=1$ and
$\mu(p,r)=2$.

\subsection{Winding number}

In the following, we define the winding number for di-gons and hearts. It will
be used for analysing and classification purposes.

\vsp

Consider the universal cover $\tau : \Mti \to M$ with induced orientation as
topological \mf. For all $\zti \in \Mti$, the orientation induces an isomorphism
$H_2(\Mti, \Mti \backslash\{\zti\}) \simeq \Z$ and the contractibility of
$\Mti\simeq \R^2$ implies $H_1(\Mti \backslash\{\zti\}) \simeq H_2(\Mti, \Mti
\backslash\{\zti\})$. Denote the fundamental class of $S^1$ by $[S^1]$.
Now consider a continuous path $\gati :S^1 \to \Mti$ and $\zti \in \Mti
\backslash \Img (\gati)$. We define the {\em winding number of $\gati$ w.r.t.
$\zti$} by $\Ind_\gati(\zti):=\gati_*([S^1]) \in H_1(M\backslash\{\zti\}) \simeq
\Z$.

Identifying $\Mti$ with $\R^2$ by an orientation preserving \diffeo, the winding
number also can be seen as mapping degree of $S^1 \to S^1$, $t \mapsto
\frac{\gati(t) - \zti}{\betrag{\gati(t)-\zti}}$.

\begin{Definition}
%\label{winding number}
%\newcommand{\refwindingnumber}{Definition \ref{winding number}}
Let $A$ stand for $D$, $D_b$ or $ D_c$ and consider $v: A \to M$ and a lift
$\vti: A \to \Mti$ of $v$.
The winding number $\Ind_\vti(\zti)$ of $\vti$ w.r.t. $\zti \in \Mti \backslash
\vti(\del A)$ is defined as the winding number of the path $\vti |_{\del A}$
w.r.t. $\zti$ with $\del A$ parametrized counterclockwise.

The {\em winding number} of $v$ w.r.t. $z \in M \backslash v(\del A)$ is defined
as
\beqs
\Ind_v(z):=\sum_{\zti \in \tau^{-1}(z)} \Ind_\vti(\zti)
\eeqs
\end{Definition}

%$\Ind_v$ does not depend on the choice of the lift $\vti$ and the sum is finite
since $\Ind_\vti$ vanishes for all $\zti$ lying in the unbounded component of
$\Mti \backslash \vti(\del A)$.

%\vsp

There is another way to compute the winding number of $\vti: A \to \Mti$:
For $\zti \in \Mti\backslash \vti(\del A)$ let $\Bti$ be a small ball around
$\zti$ and similarly consider small balls $B_i$ around the $z_i \in
\tau^{-1}(\zti)$. Identify $\del \Bti \simeq S^1 \simeq \del B_i$ and set
$\hat{A}:=A \backslash (\bigcup_{z_i \in \tau^{-1}(\zti)}B_i)$.
Then using some kind of `local degree' (see Bredon \cite{bredon}) we obtain
\begin{align*}
\Ind_\vti(\zti) = \deg( \del A \to \del \Bti) 
 = \deg (\del \hat{A} \to \del \Bti) + \deg (\bigcup_{z_i \in \tau^{-1}(\zti)}
\del B_i \to \del \Bti).
\end{align*}
Now if $N$ and $P$ are \cpt\ orientable \mf s of dimension $n$ without boundary
and if a smooth $\al: N \to P$ can be extended smoothly to some
$(n+1)$-dimensional \mf\ $Q$ with $\del Q=N$ then $\deg (\al)=0$ (see for
example Milnor \cite{milnor}).
Recognizing $\hat{A}$ as $Q$ and $(\bigcup_{z_i \in \tau^{-1}(\zti)} \del B_i)
\cup \del A$ as $N$ we deduce $\deg (\del \hat{A} \to \del \Bti)=0$ whereas the
term $ \deg (\bigcup_{z_i \in \tau^{-1}(\zti)} \del B_i \to \del \Bti)$ yields
for orientation preserving immersions:

\begin{Remark}
%\label{pre-image v}
%\newcommand{\refpreimagew}{Remark \ref{pre-image v}}
For $v \in \mcM(p,q)$ and $z \in M \backslash v(\del D)$ holds
$\Ind_v(z)=\#v^{-1}(z)$
and therefore in particular $\Ind_v \geq 0$. The analogous result is true for $v
\in \mcN(p,r)$.
If there is a component of $M \backslash v(\del A)$ with $\Ind_v < 0$ then $v$
is no immersion.
\end{Remark}

The union of those components of $M\backslash v(\del D)$ with $\Ind_v>0$ is
called the {\em interior} $\Int(v)$ of $v$. The union of the others is called
the {\em exterior} $\Ext(v)$ of $v$ (their winding number vanishes).

\vsp

The following remark will be needed for the existence of the `cutting points' in
the cutting construction \refexistenceqq. There we will need to know that the
vertices of an immersion are not multiply covered. Now choose a metric on $M$.
Since the image of our immersions stays in a compact region the following does
not depend on the choice of the metric.

\begin{Remark}
\label{inj around vertex}
Let $v \in \mcM(p,q)$.
Then there is $\ep>0$ such that $U_p:=v^{-1}(B_\ep(p))$ is a connected
neighbourhood of $(-1,0)\in D$ with $v|_{U_p}$ injective. 
As a consequence $\Ind_v=1$ on $B_\ep(p) \cap v(D)$ and $v(U_p)$ is the
wedge-shaped piece of $B_\ep(p)$ bounded by $([p,q]_u \cup [p,q]_s) \cap
B_\ep(p)$ with angle $<\pi$. 
%In particular $B_\ep(p) \backslash v(U_p)=B_\ep(p) \backslash v(D)$ lies in the
exterior of $v$. 
An analogous statement is true for $q$. For $v \in \mcN(p,r)$ with vertices $p$
and $r$ the only change is $> \pi$ for the concave vertex.
\end{Remark}

Here the lack of self-intersections of $W^u$ and $W^s$ plays an important role
--- otherwise the statement is not true.

\subsection{Gluing and cutting}

Briefly, gluing of two immersed di-gons $v \in \mcMhat(p,q)$ and $\vhat \in
\mcMhat(q,r)$ with $\mu(p,q)=1=\mu(q,r)$ (and therefore $\mu(p,r)=2$) is the
construction which recognizes the tupel $(v,\vhat)$ as an element of
$\mcNhat(p,r)$. 
Cutting is the `inverse' construction which starts with $w \in \mcNhat(p,r)$ and
finds two significant points $q_u$, $q_s \in \mcH$ such that $w$ can be seen
either as tupel $(v,\vhat) \in \mcMhat(p,q_u) \x \mcMhat(q_u,r)$ or as tupel
$(v', \vhat') \in \mcMhat(p,q_s) \x \mcMhat(q_s,r)$.

\begin{Theorem}[Gluing]
\label{gluing}
Let $p$, $q$, $ r \in \mcH$ with $[p]=[q]=[r]$ and $\mu(p,q)=1=\mu(q,r)$. Let $v
\in \mcMhat(p,q)$ and $\vhat \in \mcMhat(q,r)$. Then the {\em gluing procedure
$\#$} for $v $ and $\vhat$ yields an immersed heart $w:=\vhat\#v \in
\mcNhat(p,r)$.
\end{Theorem}

\begin{proof}
The four possible geometric positions of the three involved points are described
in Figure \ref{(non)admissible shapes}. The $q$ which lies on that part of the
unstable manifold, which crossed the interior of the immersed heart after
passing the concave vertex, is called $q_u$ and analogously for $q_s$. 
The gluing construction $\#$ glues $v \in \mcM(p,q_u)$ and $\vhat\in
\mcM(q_u,r)$ along the common boundary segment $[p,q_u]_u$. For technical
details see Chekanov \cite{chekanov}. 
\end{proof}

\begin{figure}[h]
\begin{center}

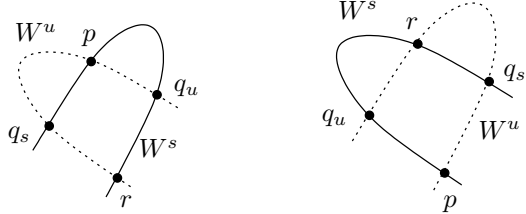

\caption{Immersions with $\mu(p)=\mu(q)+1=\mu(r)+2$.}
\label{(non)admissible shapes}
\end{center}
\end{figure}

We call the two connected components of $W^s\backslash\{x\}$ resp.\
$W^u\backslash \{x\}$ the {\em branches} of the (un)stable \mf s.
$W^u$ and $W^s$ are called {\em strongly intersecting} (w.r.t. $x$) if each
branch of $W^u$ intersects each branch of $W^s$, i.e. $W^i_+ \cap W_j^+ \neq
\emptyset \neq W^i_- \cap W_j^+ $ for $i, \ j \in \{0,1\}$ and $i \neq j$.

\vsp

To be strongly intersecting is generic in $C^1$-topology on closed
$n$-dimensional \mf s (Takens \cite{takens}, Xia \cite{xia}). For $C^r$-topology
with $1 \leq r \leq \infty$ there are results by Robinson, Pixton and Oliveira
on $S^2$ and $T^2$. If the action of the symplectomorphism on the first homology
group is irreducible then Oliveira \cite{oliveira} proved $C^r$-genericity for
closed surfaces with genus $g \geq 2$. This hypothesis is not fulfilled by
symplectomorphisms isotopic to the identity (for example \ham\ \diffeo s). For
the latter ones, Xia \cite{xia3} proved strongly intersecting to be
$C^r$-generic on closed surfaces.

\begin{Theorem}[Cutting]
\label{existence q, q'}
Let $W^u$ and $W^s$ be strongly intersecting and transverse.
Let $p$, $ r \in \mcH$ with $[p]=[r]$ and $\mu(p,r)=2$ and $w \in \mcN(p,r)$. 
Then there are distinct, unique $q_u$, $q_s \in \mcH$ with
$\mu(p,q_i)=1=\mu(q_i, r)$ and $v_i \in \mcM(p,q_i)$, $\vhat_i \in \mcM(q_i,r)$
such that $\vhat_i \#v_i=w$ for $i \in \{s,u\}$.
\end{Theorem}

Note that a symplectomorphism $\phi$ is either orientation preserving on the
stable {\it and} unstable \mf\ or orientation reversing on both.
In the first case we call $\phi$ {\em $W$-orientation preserving} and in the
latter one {\em $W$-orientation reversing}. 

Moreover recall Palis' $\lam$-lemma \cite{palis}: Given a small $\dim
W^u$-dimensional disk $D^u \subset W^u(x,\phi)$ centered around $x$, $p \in
W^s(x,\phi)$ and a $\dim W^u$-dimensional disk $D$ around $p$ intersecting
$W^s(x,\phi)$ transversely, then $\bigcup_{n\geq 0}\phi^n(D)$ contains an
$m_u$-disk arbitrarily $C^k$-close to $D^u$.

\vsp

The proof of the cutting construction differs considerably from the one for {\em
\cpt} Lagrangian sub\mf s which can be found in Chekanov \cite{chekanov}, de
Silva \cite{de silva}, Gautschi $\&$ Robbin $\&$ Salamon
\cite{gautschi-robbin-salamon} or Robbin \cite{robbin}.

\begin{proof}
Let us start with the {\it $W$-orientation preserving case} and assume $p$ to be
the concave vertex of $w \in \mcN(p,r)$, thus $w: D_b \to M$. Let w.l.o.g. $r<_i
p$ for $i \in \{s,u\}$.

Given a small disk neighbourhood $D(x)$ of $x$ in $W^u$ there is $n \in \N$ such
that $\phi^{-n}(p)$ and $\phi^{-n}(r)$ lie in $D(x)$. If we can prove the
existence of `cutting points' $q_u$ and $q_s$ for $\phi^{-n}(p)$, $\phi^{-n}(r)$
and $\phi^{-n} \circ w$ then $\phi^n(q_u)$ and $\phi^n(q_s)$ are cutting points
for $p$, $r$ and $w$.

Now choose $D(x)$ to be the `convergence disk' $D^u \subset W^u$ of the
$\lam$-lemma and assume from now on w.l.o.g. $p$, $r \in D^u$.

\vsp

In order to find the cutting point $q_u$, we follow the segment $[p, \infty[_u$.
For a certain time after $p$, it stays in the interior of $w$. We will prove
that at some point it passes $w(\del D_b)$ to the exterior of $w$ and that the
first such point will be our desired $q_u$.
We define 
\beqs
q_u:=\min \{q \in W^u \mid p<_uq, \ q \in\ ]r,p[_s , \ [q, q+\ep[_u \ \cap \
w(D_b)^c \neq \emptyset \mbox{ for } \ep>0 \}
\eeqs
where the last condition deals with the possible lack of global injectivity.
Now we prove that such a minimum always exists.
We use the notation for the branches $W^u_\pm$ and $W^s_\pm$ as sketched in
Figure \ref{qucutpoint} (a).

\begin{figure}[h]
\begin{center}

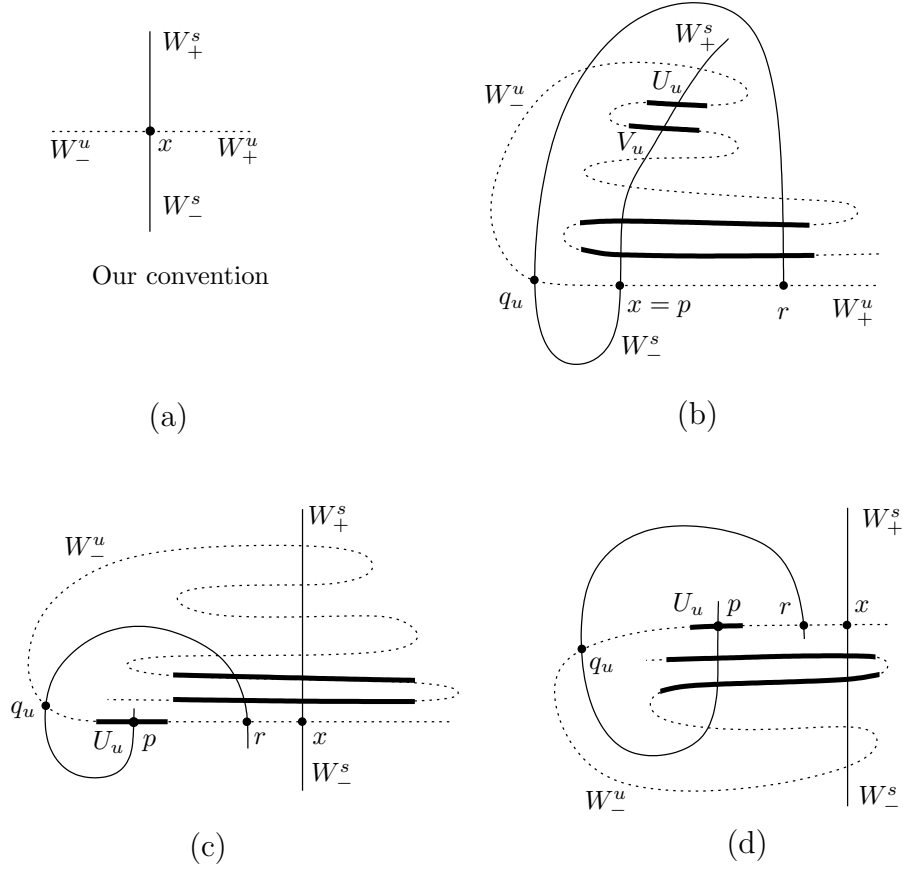

\caption{Constructions for $q_u$.}
\label{qucutpoint}
\end{center}
\end{figure}

\vsp

{\it Case $p=x\neq r$}: For the relative positions of $p$ and $r$ see Figure
\ref{qucutpoint}(b).
$W^u_+$ is the branch of $W^u$ containing $r$. $W^s_+$ is the branch of $W^s$
which starts in the local picture on the same side of $W^u$ as $[r,r + \ep]_s$
for small $\ep >0 $.
Since $W^u$ and $W^s$ are strongly intersecting and transverse there is $q \in
W^u_- \inter W^s_+$ with $p<_u q$.
Let $t_i^q:=\ga^{-1}_i(q)$ and consider a small neighbourhood of $q$ in $W^u_-$.
If $\sign(\det (\gadot_u(t_u^q), \gadot_s(t_s^q)))$ is negative we denote the
neighbourhood by $U_u$ and otherwise by $V_u$. These neighbourhoods meet $W^s$
transversely in $q$ such that the $\lam$-lemma implies the $C^k$-convergence of
disks $D_n \subset \phi^{n}(U_u)$ resp.\ $D_n \subset \phi^n(V_u)$ to $D^u$ for
$n \to \infty$. Recall that $r \in D^u$ and that $[r,p]_s$ intersects $W^u$ in
$r$ transversely. Thus for given $\ep >0$, there is $n_0$ large enough such that
$D_n$ and $[r,r + \ep]_s$ intersect for $n \geq n_0$, see the extra bold long
segments in Figure \ref{qucutpoint} (b). 
\refinjaroundvertex\ states that, for $\ep>0$ small enough, the ball $B_\ep(r)$
splits into two wedge-shaped regions $W_{int} \subset \Int(w)$ and $W_{ext}
\subset \Ext(w)$ with common boundary $([r,p]_u  \cup [r,p]_s) \cap B_\ep(r)$.

Thus for $n$ large enough, $D_n$ meets $W_{ext}$ before or after passing
$]r,r+\ep[_s$ depending on if $D_n$ lies in an iterate of $U_u$ or $V_u$.
Therefore, the segment $[p,\infty [_u$ leaves $ \Int(w)$ and meets $Ext(w)$ such
that points as claimed in the definition of $q_u$ exist and so does the minimum
$q_u$.

\vsp

{\it Case $p \neq x$}. Here we do not need the (un)stable \mf s to be strongly
intersecting. The sketches in Figure \ref{qucutpoint} are schematical and it is
unimportant if $x$ lies in the exterior of $w$ or not. 
There are {\it two subcases}, namely if $p \in W^u_- \inter W^s_+$ as in (c) or
if $p \in W^u_- \inter W^s_-$ as in (d).

We start with $p \in W^u_- \inter W^s_+$ and consider a small neighbourhood
around $p$ in $W^u_-$. 
If $\sign(\det (\gadot_u(t_u^p), \gadot_s(t_s^p)))$ is negative we denote the
small neighbourhood by $U_u$ and otherwise by $V_u$. Since $p \in W^s_+$, the
disks $D_n \subset\phi^n(U_u)$ resp.\ $D_n \subset \phi^n(V_u)$ approach the
`convergence disk' $D^u$ centered at $x$ from the $W^s_+$-side for $n \to
\infty$, see the extra bold long segments in Figure \ref{qucutpoint} (c). As in
the proof of {\it case $p=x$}, we consider the special neighbourhood sectors
$W_{int}\subset \Int(w)$ and $W_{ext} \subset \Ext(w)$ of $r$ provided by
\refinjaroundvertex. We conclude that $]p, \infty[_u$ passes somewhere through
$W_{ext}$ and therefore meets the exterior of $w$. Thus points as claimed in the
definition of $q_u$ exist and so does the minimum $q_u$.

\vsp

{\it Case $p \in W^u_- \inter W^s_-$} (sketched in Figure \ref{qucutpoint} (d))
proceeds analogously to {\it case $p \in W^u_- \inter W^s_+$} except from the
following fact: Now the disks $D_n \subset\phi^n(U_u)$ resp.\ $D_n \subset
\phi^n(V_u)$ approach the `convergence disk' $D^u$  from the $W^s_-$-side for $n
\to \infty$, see the extra bold long segments in Figure \ref{qucutpoint} (d).
Therefore we have to use the sectors $W_{int}\subset \Int(w)$ and $W_{ext}
\subset \Ext(w)$ of $p$ instead of those of $r$ and then proceed as above.

\vsp

Since we only need the oscillation behaviour predicted by the $\lam$-lemma and
those special neighbourhood sectors around the vertices the proof carries over
to all possible relative positions of $x$, $p$ and $r$ within $W^u$ and $W^s$ in
case $p \neq x$.

Exchanging the roles of $W^u$ and $W^s$, the constructions for $q_s$ are similar
to those for $q_u$.
If $r$ is the concave vertex the proof proceeds similarily.

\vsp

Now we will describe the cutting procedure from $p$ to $q_u$. 
Recall from \refinjaroundvertex\ that $w$ is injective on a small neighbourhood
of $p$. If we consider $w^{-1}([p,q_u]_u)$ then there is a unique segment in
$D_b$ denoted by $I$ whose start point is $w^{-1}(p)=-1$. By definition of
$q_u$, the segment $[p, q_u+\ep]_u$ leaves $w(D_b)$ through $q_u$ for $\ep >0$.
Thus there is $\qti \in w^{-1}(q_u)$ which has to be the endpoint of $I$. In
fact, since $q_u$ lies per definitionem on a boundary segment parting the
interior from the exterior $w$ is injective in a neighbourhood of $q_u$ such
that $\{\qti\} =w^{-1}(q_u)$ is even unique.

We now cut $ D_b$ along $I$ into $ D_b^v$ and $ D_b^{\vhat}$ 
The boundary conditions of $D_b^v$ are $B_u^v= I$ and $B_s^v$ is the segment
from $-1 $ to $\qti$ in $B_s$. And for $D_b^{\vhat}$ we have $B_u^{\vhat}=I\cup
B_u$ and $B_s^{\vhat}$ is the segment from $\qti$ to $1$ in $B_s$.
Identify $D_b^v$ and $D_b^{\vhat}$ with the di-gon $D$ via $h^v: D_b^v \to D$
with $h^v(B_i^v)=B_i^D$ and $h^{\vhat}: D_b^{\vhat} \to D$ with
$h^{\vhat}(b_i^{\vhat})=B_i^D$ for $i \in \{s,u\}$ and define 
\begin{align*}
& v: D \to M, \quad  v(z):=w((h^v)^{-1}(z)), \\
& \vhat: D \to M, \quad  \vhat(z):=w((h^{\vhat})^{-1}(z)).
\end{align*}
Since our techniques considered the branches of the (un)stable \mf s separately
the $W$-orientation reversing case is reduced to the $W$-orientation preserving
case by considering the $W$-orientation preserving $\phi^2$ instead of $\phi$.
\end{proof}

\section{Primary Floer homology}

Set $p^n:= \phi^n(p)$ for $p \in \mcH$ and $n \in \Z$. Keep in mind that in this
notation $p=p^0$.

\subsection{Primary homoclinic points}

This subsection introduces a very special kind of homoclinic points, so-called
(semi-)primary points. These points are characterized by a very rigid geometric
condition. Semi-primary points are usually the first type of homoclinic points
investigated by physicists and they play a crucial role in the Melnikov method
(see for example Rom-Kedar \cite{rom-kedar1, rom-kedar2}).

\vsp

$p \in \mcH\backslash\{x\}$ is called {\em semi-primary} if $]x,p[_u \ \cap\
]x,p[_s\ = \emptyset$.
$p \in \mcH_{[x]}\backslash \{x\}$ is {\em primary} if $]x,p[_u \ \cap \
]x,p[_s\ \cap\ \mcH_{[x]}  =\emptyset $.
Nonprimary points are called {\em secondary}.
Clearly,
iterates of a (semi-)primary point are again (semi-)primary. If $W^u \cap W^s
\neq \emptyset$ then semi-primary points always exist.
We require $[p]=[x]$ in the definition of primary points, since this condition
was already necessary for the invariance of the Maslov
index and the homotopy classes under the $\Z$-action of $\phi$. The condition
`\dots $\cap\ \mcH_{[x]}$' will be necessary in the invariance discussion.

\begin{Remark}
\label{position prim}
%\newcommand{\refpositionprim}{Remark \ref{position prim}}
%\label{no comp frame}
%\newcommand{\refnocompframe}{Remark \ref{no comp frame}}
\begin{enumerate}[(1)]
\item
Let $\phi$ be $W$-orientation preserving, $p \in \mcH$ (semi-)primary and denote
the branches containing $p$ by $W^u_p$ and $W^s_p$. 
Then for every (semi-)primary $q \in( W^u_p \cap W^s_p) \backslash \{p^n \mid n
\in \Z\}$ there is a unique $n \in \Z$ such that $q^n \in \ ]p,p^1[_u \ \cap \
]p, p^1[_s$.

If $\phi$ is $W$-orientation reversing then $p^1$ has to be replaced by $p^2$
and $n$ by $2n$.
\item
Let $p$ be semi-primary and $q$ primary within the same pair of branches. If $q
\in\ ]p,p^1[_u$ then $q \notin\ ]x,p^1[_s$. If $q \in \ ]p,p^1[_s$ then $q
\notin\ ]x, p[_u$. Moreover there is $k\in \N_0$, $n \in \Z$ such that $q \in \
]p^n,p^{n+1}[_u\ \cap\ ]p^{n+k},p^{n+k+1}[_s$. There are schematic tangles with
$k >0$.
\end{enumerate}
\end{Remark}

Now consider the universal covering $\tau: (\Mti, \omti) \to (M, \om)$ with
$\omti=\tau^*\om$.
For $\xti \in \tau^{-1}(x)$ and $i \in \{s,u\}$, denote by $\Wti^i(\xti)$ the
lift of $W^i$ passing through $\xti$.
We denote the lift of the segment $[p,q]_i$ starting in $\pti \in \tau^{-1}(p)$
and ending in $\qti \in \tau^{-1}(q)$ by $[\pti, \qti]_i$. Given $\xti_u$,
$\xti_s \in \tau^{-1}(x)$, we call $\pti \in \Wti^u(\xti_u) \cap \Wti^s(\xti_s)$
{\em homoclinic} if $\xti_u=\xti_s$ and otherwise {\em heteroclinic}.
{\em Lifting the tangle (to $\xti \in \tau^{-1}(x)$)} means that we consider the
tangle generated by $\Wti^i:=\Wti^i(\xti)$ for $i \in \{s,u\}$ on $\Mti$. With a
contractible $p \in W^u \cap W^s$, we associate $\pti \in \tau^{-1}(p)$ such
that the lift of $[p,x]_i$ starting in $\pti$ ends in $\xti$. With a
noncontractible $p$, we associate $\pti$ such that the lift of $[p,x]_u$
starting in $\pti$ ends in $\xti=\xti_u$. 
Therefore contractible homoclinic points lift to homoclinic points and
noncontractible ones to heteroclinic ones.

$\pti \in \Wti^u(\xti) \cap \Wti^s(\xti)$ is called {\em primary} if $]\pti,
\xti[_u \ \cap\ ]\pti, \xti[_s\ = \emptyset$.
For $\pti$, $\qti$, $\rti \in \Wti^u \cap \Wti^s$, we define $\mcM(\pti, \qti)$,
$\mcMhat(\pti, \qti)$, $\mcN(\pti, \rti)$ and $\mcNhat(\pti, \rti)$ in the
obvious way.

\vsp 

Now consider the tangle lifted to $\xti \in \tau^{-1}(x)$. 
$p \in W^u \cap W^s$ is primary if and only if $\pti \in \Wti^u \cap \Wti^s$ is
primary. Moreover, \refpositionprim\ holds also for the primary points in
$\Wti^u \cap \Wti^s$. The following statements are independent of the chosen
reference point $x$.

\begin{Lemma}
\label{index prim}
Let $\pti \in \Wti^u \cap \Wti^s$ be primary. Then $\mu(\pti):=\mu(\pti,\xti)
\in \{\pm 1, \pm 2, \pm 3\}$. There is either an embedded di-gon or an embedded
heart or an embedded 2-gons with two concave vertices from $\pti$ to $\xti$
(resp.\ from $\xti$ to $\pti$ depending on the sign of the index). For the
primary $p:=\tau(\pti)$ follows $\mu(p):=\mu(p,x) \in \{\pm 1, \pm 2, \pm 3\}$.
\end{Lemma}

\begin{proof}
Since $[\pti]=[\xti]$, the two points can be connected by a path in
$\mcP(W^u,W^s)$. Since $]\pti,x[_u \ \cap\ ]\pti,x[_s \ = \emptyset$ the region
enclosed by $[\pti, \xti]_u$ and $[\pti,\xti]_s$ is an embedded polygon with two
vertices. Assume the intersections in $\pti$ and $\xti$ to be orthogonal and
parametrize the segments $[\pti, \xti]_u$ from $\pti$ to $\xti$ and $[\pti,
\xti]_s$ from $\xti$ to $\pti$. Then the Maslov index is twice the winding
number of the tangent vector of the segments. Thus only $\mu(\pti,\xti) \in
\{\pm 1,\pm 2, \pm 3\}$ can be realized without violating the boundary condition
$]\pti,\xti[_u \ \cap\ ]\pti,\xti[_s \ = \emptyset$.    
Therefore $\mu(p)=\mu(\pti) \in \{\pm 1, \pm 2, \pm 3\}$ for $p= \tau(\pti)$.
\end{proof}

Note that the immersion between a primary $p$ and the fixed point $x$ does not
need to be globally injective since there might be noncontractible points in
$]x,p[_u \ \cap \ ]x,p[_s$.

\begin{Remark}
\label{primary max}
%\newcommand{\refprimarymax}{Remark \ref{primary max}}
%\label{prop adj}
%\newcommand{\refpropadj}{Corollary \ref{prop adj}}
\label{finite primary}
\begin{enumerate}[(1)]
\item
Let $i \in \{s,u\}$, $\pti \in \Wti^u \cap \Wti^s $ be primary and
$p:=\tau(\pti)$. 
Lift $\ga_i$ to $\gati_i: \R \to \Wti^i$ and obtain the ordering $<_i$ on
$\Wti^i$.
W.l.o.g. assume $p <_i\phi(p)$; for $W$-orientation reversing $\phi$ use
$p<_i\phi^2(p)$. 
Then
\begin{align*}
 \pti _+ & :=\max \{\qti \in \Wti^s \mid \qti<_s\pti  , \ \qti \in\ ]\xti, \pti
[_u\}, \\
 \pti _- & :=\min\{\qti \in \Wti^u \mid  \pti <_u \qti, \ \qti \in \ ]\xti,\pti [_s
\}
\end{align*}
are primary and $\pti _\pm$ is called {\em adjacent} to $\pti $.
\item
Let $\pti $ be primary and $\qti =\pti _\pm$. Then $]\pti ,\qti[_u \ \cap\ ]\pti
,\qti[_s\ = \emptyset$.
If moreover $\pti $ and $\qti$ are transverse then $\mu(\pti ,\qti) \in \{ 1,
-1\}$ and there is an embedded di-gon between them.
\item
%If $\pti $ is a transverse primary point then there is an odd number of
transverse primary points in $]\pti ,\pti ^1[_u \ \cap\ ]\pti ,\pti ^1[_s$.
Let $\pti $ be primary and order the primary points in $[\pti, \pti^{-1}]_u \cap
[\pti, \pti^{-1}]_s$ via $\pti $, $\pti _+$, $(\pti _+)_+, \dots, \pti ^{-1}$
and assume them transverse. Then their relative Maslov index alternates between
$+1$ and $-1$.
\item
Let all primary points $p \in W^u \cap W^s$ be transverse. Then there are modulo
$\Z$-action only finitely many primary points. The same is true for the primary
points in $\Wti^u \cap \Wti^s$.
\end{enumerate}
\end{Remark}

\subsection{Signs and coherent orientations}

Now we define the signs needed for the definition of the boundary operator of
the Floer chain complex. The signs have to satisfy a certain compatibility with
the cutting and gluing procedure which is known in classical Floer theory as
`coherent orientations'.

We will define two kinds of signs with slightly different properties depending
on the type of homoclinic points (primary or secondary) and the
symplectomorphism ($W$-orientation preserving or reversing).

\vsp

For $i \in \{s,u\}$, associate to each branch $W^i_+$ and $W^i_-$ its `jump
direction' as orientation and denote it by $o(W^i_+)$ resp.\ $o(W^i_-)$. 
Let $p$, $q$ be primary with $\mu(p,q)=1$ and $v \in \mcM(p,q)$.
Associate to $v(B_i)=[p,q]_i$ the orientation induced by the parametrization
from $p$ to $q$ and call it $o_{pq}$.
In (the proof of) \refpropclassifind, we will show that $x \notin\ ]p,q[_u\
\cap\ ]p,q[_s$. Thus, there is a branch $W_{pq} \in \{W^u_+, W^u_-, W^s_+,
W^s_-\}$ containing both $p$ and $q$.
We set
\beqs
m(p,q):= \left\{
\begin{aligned}
 1 &&& \mbox{if } \mu(p,q)=1, \ \mcM(p,q) \neq \emptyset, \ o(W_{pq})=o_{pq}, \\
 -1 &&& \mbox{if } \mu(p,q)=1, \ \mcM(p,q) \neq \emptyset, \ o(W_{pq}) \neq
o_{pq}, \\
0 &&& \mbox{otherwise}.
\end{aligned}
\right.
\eeqs
If there are two branches $W^u_{pq}$ and $W^s_{pq}$ containing $p$ and $q$ then
$p$ and $q$ are adjacent and $o(W^u_{pq})=o_{pq}=o(W^s_{pq})$, compare Figure
\ref{all immersions}. Thus $m(p,q)$ is well-defined. We do not need to
distinguish the cases $W$-orientation preserving and reversing since
$m(p,q)=m(p^l, q^l)$ for all $l \in \Z$. 
The definition does not generalize to arbitrary homoclinic points.

\begin{Lemma}
\label{sign skew sym}
Let $p$ and $r$ be primary with $\mu(p,r)=2$ and $w \in \mcNhat(p,r)$. For $i
\in \{s,u\}$ assume the existence of $q_i$ with $\mu(p,q_i)=1=\mu(q_i,r)$ and
$v_i \in \mcMhat(p,q_i)$ and $\vhat_i \in \mcMhat(q_i,r)$ such that $\vhat_i
\#v_i=w$. Then
\beqs
m(p,q_u) \cdot m(q_u,r)=-m(p,q_s) \cdot m(q_s,r).
\eeqs
\end{Lemma}

\begin{proof}
Have a look at \refpropclassifprimary\ and check in Figure \ref{primary cutting}
the eight possible $w=\vhat_i \#v_i \in \mcNhat(p,r)$ sketched in the left and
right column. This yields the claim.
\end{proof}

Whereas $m(p,q)$ is well-defined only for primary points, there is another way
to define signs for arbitrary homoclinic points: Fix an orientation $o_u$ on
$W^u$. 
Now let $p$, $q \in \mcH$ with $\mu(p,q)=1$ and provide $[p,q]_u$ with the
orientation $o_{pq}$ induced by the parametrization from $p$ to $q$. For
$W$-orientation preserving $\phi$, we define
\beqs
n(p,q):=
\left\{
\begin{aligned}
+1 & \quad \mbox{if } \mcMhat(p,q) \neq \emptyset \mbox{ and } o_{pq}=o_u,  \\
 -1 & \quad \mbox{if } \mcMhat(p,q) \neq \emptyset \mbox{ and } o_{pq} \neq o_u,
 \\
0 & \quad \mbox{if } \mcMhat(p,q)=\emptyset.
\end{aligned}
\right.
\eeqs
$n(p,q)$ clearly also could be defined using an orientation on $W^s$. For
$W$-orientation reversing $\phi$ we have to set $n_2(p,q):= n(p,q) \ mod \ 2$.
The signs depend as follows from the chosen data: Set $\mcH_{[x]}^l=\{p \in \mcH
\mid \mu(p,x)=l, [p]=[x] \}$ and provide $W^i$ with the orientation induced by
$\gadot_i$. Let $\si_{01}:=\sign(\det (\gadot_u(0), \gadot_s(0)))$ and denote
the signs defined via the orientation on $W^i$ by $n(p,q,W^i)$. Then
\begin{equation}
\label{Lsigns}
\begin{aligned}
& n(p,q,W^u)= \si_{01} n(p,q,W^s) && \mbox{for } p \in \mcH_{[x]}^{2l}, \\
& n(p,q,W^u)= - \si_{01} n(p,q,W^s) && \mbox{for } p \in \mcH_{[x]}^{2l+1}
\end{aligned}
\end{equation}
for all $ q \in \mcH_{[x]} $ and $ l \in \Z$.

\begin{Lemma}
Let $p$, $r \in \mcH$ with $\mu(p,r)=2$ and $w \in \mcN(p,r)$. For $i \in
\{s,u\}$ consider $q_i \in \mcH$ with $\mu(p,q_i)=1=\mu(q_i,r)$ and $\vhat_i \in
\mcM(p,q_i)$ and $v_i \in \mcM(q_i,r)$ such that $\vhat_i \#v_i=w$.
Then
\beqs
n(p,q_u)\cdot n(q_u,r)=- n(p,q_s) \cdot n(q_s,r)
\eeqs
and this relation also is true for $n_2$.
\end{Lemma}

\begin{proof}
Consider Figure \ref{(non)admissible shapes}, choose an orientation on $W^u$ and
check that the claim is true. 
If we choose the other orientation on $W^u$ all signs swap and the relation
remains true.
\end{proof}

\subsection{Primary Floer homology}

Now we are ready to define the Floer chain complex. We assume from now on (if
not stated otherwise) all homoclinic points to be primary and transverse.

We define on $ \mcHpr:=\{p \in \mcH \mid p \mbox{ primary}\}$ an equivalence
relation via $ p \sim q \IFF \exists\ n \in \Z$ with $ q^n=p$. 
We set $ \mcHprti:= \mcHpr \slash _ \sim$ and denote by $\langle p \rangle$ the
equivalence class of $p$. Note that $\# \mcHprti <\infty$ according to
\refprimarymax.
Due to \refmuZcomp, we can establish a well-defined homotopy class and a Maslov
index via $[\langle p \rangle ] := [p]$, $\mu(\langle p \rangle, \langle q
\rangle):=\mu(p,q)$ and $\mu(\langle p \rangle):= \mu(p,x)$.
We define
%\label{delfrak}
%\newcommand{\refdelfrak}{Definition \ref{delfrak}}
\begin{gather*}
 \mathfrak C_m:= \mathfrak C_m(x, \phi; \Z) := \bigoplus_{\stackrel{p \in
\mcHpr}{\mu(p)=m}} \Z p, \\
 \mathfrak d_m : \mathfrak C_m \to \mathfrak C_{m-1}, \qquad
 \mathfrak d (p) = \sum_{\stackrel{q\in \mcHpr}{\mu(q)=\mu(p)-1}}m(p,q)q 
\end{gather*}
on a generator $p$ and extend $\mathfrak d$ by linearity. $\phi$ induces
$\phi_*: \mathfrak C_* \to \mathfrak C_*$ satisfying $\phi_* \circ \dd = \dd
\circ \phi_*$.
The sum is finite since $\#\mcHprti <\infty$ and, as we will see later in
\reffiniteprimiterate, $\#\{n \in \Z \mid \mcM(p,q^n) \neq \emptyset\}<\infty$.

\vsp

$\mu(p)=\mu(p^n)$ for $n \in \Z$ implies that the chain groups
have infinite rank over $\Z$.
But since $\mu(p):=\mu(p,x) \in \{\pm 1, \pm2, \pm 3\}$ for $p \in \mcHpr$ there
are at most six nonvanishing chain groups.

\begin{Theorem}
\label{delfraksquare}
$\mathfrak d \circ \mathfrak d=0$, i.e. $(\mathfrak C_*, \mathfrak d_*)$ is a
chain complex.
\end{Theorem}

The proof of \refdelfraksquare\ is postponed to the following subsections. The
homology of $(\mathfrak C_*, \mathfrak d)$ is
\beqs
\mathfrak H_m:= \mathfrak H_m(x,\phi; \Z):= \frac{\ker \mathfrak d_m} {\Img
\mathfrak d_{m+1}}.
\eeqs

Since the chain groups have infinite rank over $\Z$ this might also be the case
for
the homology groups. In order to enforce finite rank, we will divide by the
$\Z$-action:
%\label{del}
%\newcommand{\refdel}{Definition \ref{del}}
For $\langle p
\rangle $, $\langle q \rangle \in \mcHprti$ set
%\beqs
$m(\langle p \rangle, \langle q \rangle):=\sum_{n \in \Z} m(p,q^n)$ 
%\eeqs
and define
\begin{gather*}
 C_m:=C_m(x,\phi; \Z):= \bigoplus_{\stackrel{\langle p \rangle \in
\mcHprti}{\mu(\langle p
 \rangle)=m}} \Z \langle p \rangle,\\
 \del_m : C_m \to C_{m-1},
\qquad \del \langle p \rangle := \sum_{\stackrel{\langle q \rangle \in
    \mcHprti}{\mu(\langle q \rangle)=\mu(\langle p \rangle) -1}} m(\langle p
\rangle , \langle q \rangle) \langle q \rangle
\end{gather*}
on a generator $\langle p \rangle$ and extend $\del$ by linearity.
The compatibility of the $\Z$-action with the Maslov index and the homotopy
classes yields the well-definedness of $\del$.

\vsp

We have $\rk_\Z(C_m)=\#\{\langle p \rangle \in \mcHprti \mid \mu(\langle p
\rangle)=m\}< \infty$. 
And due to \refindexprim, at most $C_{\pm 1}$, $C_{\pm 2}$ and $C_{\pm 3}$ are
nonzero. Moreover, \reffiniteprimary\ implies $\rk_\Z C_{\pm 2}= \rk_\Z C_{\pm
1} + \rk_\Z C_{\pm 3}$.

If we generalize the notion of equivalence classes to finite sums via $\langle
p + q \rangle= \langle p \rangle + \langle q \rangle$ the differential can also
be written as 
\beqs
\del \langle p \rangle = \langle \mathfrak d p \rangle =
\sum_{\stackrel{q\in \mcHpr}{\mu(q)=\mu(p)-1}}m(p,q)\langle q
\rangle.
\eeqs
Therefore $\mathfrak d ^2=0$ immediately implies

\begin{Theorem}
$\del \circ \del=0$, i.e. $(C_*, \del_*)$ is a chain complex.
\end{Theorem}

We define the {\em primary Floer homology of $\phi$ in $x$} as
\begin{equation}
\label{phfh}
H_m:= H_m(x, \phi; \Z):= \frac{\ker \del_m}{\Img \del_{m+1}}.
\end{equation}

Since already the $C_m$ have finite rank over $\Z$ so has $H_m$. All chain
groups $C_m$ and homology groups $H_{m}$ with $m \neq \pm 1, \pm 2, \pm 3$
vanish.

\vsp

\subsubsection*{Homology and Cohomology}
The question about cohomology instead of homology leads in our situation to the
choice between $\phi$ and $\phi^{-1}$ as underlying symplectomorphism. More
precisely, $H_*(x,\phi)$ is related to $H_*(x,\phi^{-1})$ in the following way.
Consider
\beqs
C^m(x,\phi;\Z):=\bigoplus_{\stackrel{ \langle p \rangle \in
\mcHprti}{\mu(\langle p \rangle )=m}}\Z \langle p \rangle
\eeqs
with differential $\de : C^m(x,\phi; \Z) \to C^{m+1}(x, \phi; \Z)$ defined on
the generators by
\beqs
\de (\langle p \rangle) := \sum_{\stackrel{  q \in \mcHpr}{\mu( q 
)=m+1}}m(q,p)\langle q \rangle.
\eeqs
Then $\de \circ \de=0$ is proven analogously to $\del \circ \del=0$ and 
\beqs
H^*(x,\phi; \Z):= \frac{\ker \de}{\Img \de}
\eeqs 
is called {\em primary Floer cohomology of $\phi$ in $x$}.
Changing from $\phi$ to $\phi^{-1}$ transforms $W^u$ into $W^s$ and vice versa,
but apart from this the homoclinic tangle remains untouched. Therefore the sign
of the Maslov index of a homoclinic point $p=p_\phi$ in the tangle generated by
$\phi$ changes, when considered as homoclinic point $p=p_{\phi^{-1}}$ in the
tangle corresponding to $\phi^{-1}$, i.e. $\mu(p_\phi)=-\mu(p_{\phi^{-1}})$.
This implies

\begin{Theorem}
\label{hom cohom}
$H^*(x,\phi)=H_{-*}(x,\phi^{-1})$.
\end{Theorem}

\subsubsection*{The signs $n(p,q)$}

The above chain complexes and homologies can be defined analogously with
$n(p,q)$ resp.\ $n_2(p,q)$ (and $\Z \slash 2\Z$-coefficients in the latter case).
We will see that $\mathfrak H_m$ and $H_m$ do not depend on the chosen data:
Let $\phi$ be $W$-orientation preserving. Changing the
orientation of $W^u$ changes the sign of the $n(p,q)$. Thus $\mathfrak d$
transforms into $- \mathfrak d$ which has the same kernel and image as
$\mathfrak d$. 
\refLsigns\ implies that the differential obtained by using an orientation on
$W^s$ instead of $W^u$ equals for fixed Maslov index $\pm 1$ times the
$W^u$-induced differential. Thus $\ker \dd^{W^u}_k =\ker \dd^{W^s}_k$ and $\Img
\dd^{W^u}_k = \Img \dd^{W^s}_k$ for all $k$ such that the homologies coincide.
$H_m$ does not depend on the choice of the orientation for the same reasons as
$\mathfrak H_m$.

If $\phi$ is $W$-orientation reversing we have to use $\Z \slash 2
\Z$-coefficients $n_2(p,q)$ if we want to be able to divide by the $\Z$-action.

\vsp

If one computes the examples in \refchapterexamples\ with $n(p,q)$ instead with
$m(p,q)$ one obtains isomorphic homologies, but the generators of the homology
groups differ.

\subsection{Well-definedness, gluing and cutting}

%\section{Immersions between primary points}

In this subsection, we will prove \refdelfraksquare. The proof is mainly based
on classifications of immersions of relative Maslov index $1$ and $2$.

\vsp

Lift the homoclinic tangle w.r.t. $\xti \in \tau^{-1}(x)$. Given primary $p$,
$q\in W^u \cap W^s$ with associated primary $\pti$, $\qti \in \Wti^u \cap
\Wti^s$, the immersions in $\mcM(p,q)$ resp.\ $\mcN(p,q)$ lift exactly to the
immersions in $\mcM(\pti, \qti)$ resp.\ $\mcN(\pti, \qti)$. 
Primary Floer (co)homology is well-defined for $(\phi, x)$ on $M$ if and only if
it is well-defined for the lifted homoclinic tangle generated by $\Wti^u$ and
$\Wti^s$ on $\Mti$. 
Thus it is enough to prove the primary cutting and gluing procedure for the
lifted tangle $\Wti^u \cap \Wti^s$ on $\Mti$.

\begin{Proposition}[Classification for index difference 1]
\label{prop classif ind 1}
Let $p$, $q \in \mcH$ be primary with $\mu(p,q)=1$ and let $\pti$ and $\qti$ the
associated primary points in $\Wti^u \cap \Wti^s$. Then either
$\mcM(\pti,\qti)=\emptyset$ or $v \in \mcM(\pti,\qti)$ is, in fact, an
embedding.
\end{Proposition}

The elements of $\mcM(p,q)$ do not need to be embeddings. Nor is it true for
noncontractible semi-primary points.

\begin{proof}
In the following, we work with the lifted tangle on $\Mti$. We drop the tilde
associated to symbols on $\Mti$. Thus identify $p=\pti$ and $q=\qti$ etc.

The proof is tedious, but simple. $[p]=[q]=[x]$ allows us to write
$1=\mu(p,q)=\mu(p,x)+\mu(x,q)$ and \refindexprim\ provides the four cases
$(\mu(p,x), \mu(x,q)) \in \{(-2,3), (-1, 2), (2,-1), (3,-2)\}$. Since there are
always two possibilities to place the concave vertex of a standard heart the
number of cases multiplies by two. Moreover, we have to distinguish $]x,p[_i \
\cap \ ]x,q[_i \ = \emptyset$ or $\neq \emptyset$ for $i \in \{s,u\}$. 
Since $W^i$ is self-intersection free and one-dimensional we conclude in case
$]x,p[_i\ \cap\ ]x,q[_i \ \neq \emptyset$ either $[x,p]_i \subset [x,q]_i$ or
$[x,q]_i \subset [x,p]_i$.
This yields a lot of cases, but fortunately some of them are symmetric.
We recall from \refindexprim\ that there is modulo parametrization exactly one
embedding between $p$ and $x$ and $q$ and $x$. Since embeddings are by
definition bijective there is --- together with the boundary conditions ---
almost no degree of freedom in sketching them. Figure \ref{all immersions} lists
all possibly arising immersions.
\begin{figure}
\begin{center}
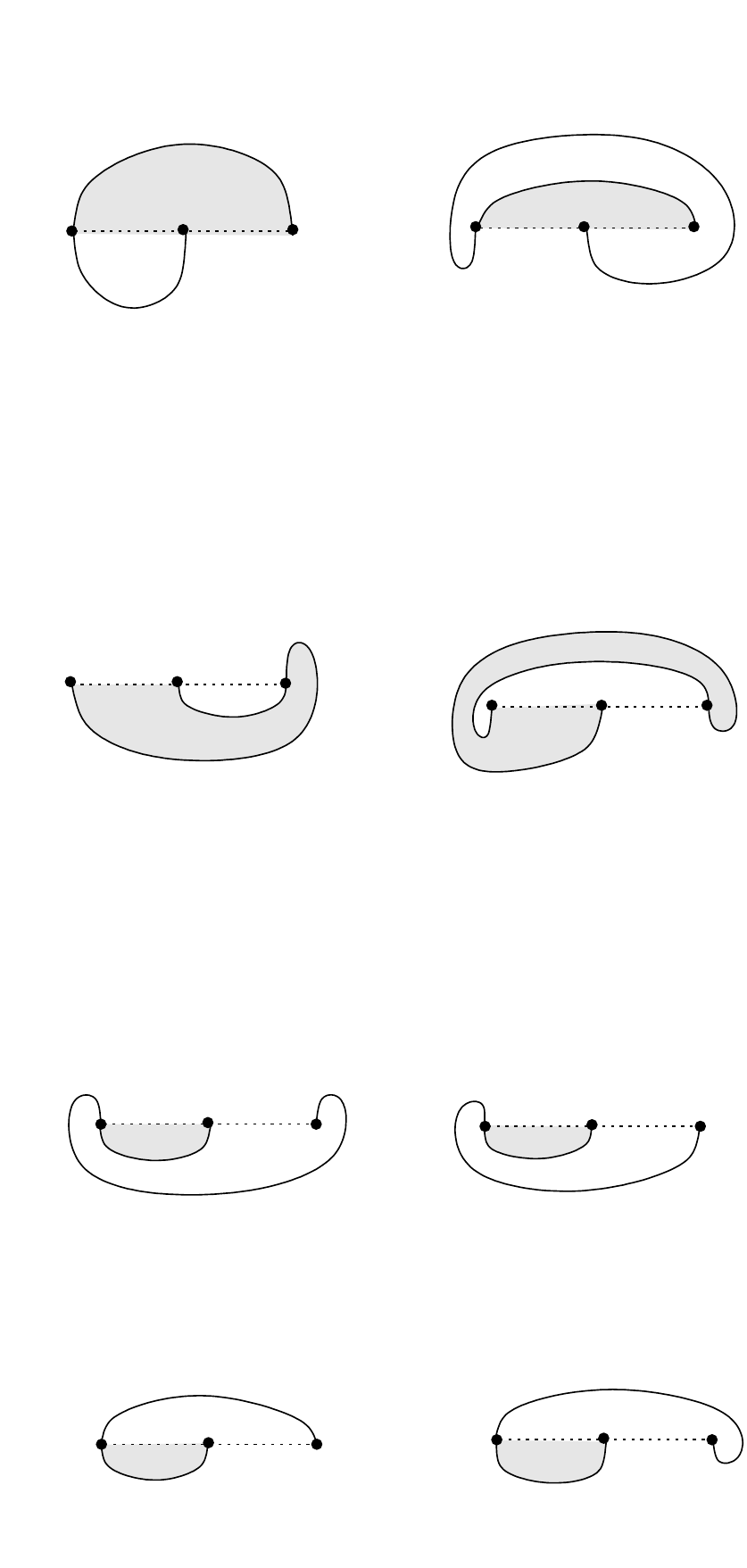
\caption{Immersions of relative index 1 up to obvious symmetries.}
\label{all immersions}
\end{center}
\end{figure}
\end{proof}

Since, according to \refpropclassifind, immersions between primary homoclinic
points $\pti$ and $\qti$ of $\Wti^u \cap \Wti^s$ are in fact embeddings it is
enough to show $]\pti,\qti[_u \ \cap \ ]\pti,\qti[_s \ \neq \emptyset$ to
prevent their existence:

\begin{Lemma}
\label{always intersections}
Let $p$, $q \in \mcH_{[x]}\backslash\{x\}$ and $p^n:=\phi^n(p)$ etc. for $n \in
\Z$. Let $\pti$, $\qti$ and $\pti^n$ etc. be the associated points in $\Wti^u
\cap \Wti^s$. Then there is $N \in \N_0$ such that for $n \in \Z$ with
$\betrag{n} \geq N$ we have $]\pti,\qti^n[_u \ \cap \ ]\pti,\qti^n[_s \ \neq
\emptyset$.
\end{Lemma}

\begin{proof}
Let $\pti$ etc. be the point associated to $p$ in the lifted tangle on $\Mti$.
Let $\phi$ be {\em $W$-orientation preserving}.

Consider the case $x \notin\ ]p,q[_u$ and $x \notin\ ]p,q[_s$. Then there is $N
\in \N_0$ such that $\pti^1 \in\ ]\pti, \qti^n[_u \ \cap\ ]\pti, \qti^n[_s$ for
all $n \geq N$ and $\pti^{-1} \in\ ]\pti, \qti^n[_u \ \cap\ ]\pti, \qti^n[_s$
for all $n \leq -N$.

If $x \in\ ]p,q[_u \ \cap\ ]p,q[_s$ then $\xti\in \ ]\pti, \qti^n[_u\ \cap  \
]\pti, \qti^n[_s$ for all $n \in \Z$.

Consider the case $x \in \ ]p,q[_u$ and $x \notin \ ]p,q[_s$. Then there is $N
\in \N_0$ such that $\qti^{N-1} \in\ ]\pti, \qti^n[_u \ \cap\ ]\pti, \qti^n[_s$
for all $n \geq N$ and $\pti^{-1} \in \ ]\pti, \qti^n[_u \ \cap\ ]\pti,
\qti^n[_s$ for all $n \leq -N$.
In the case $x \notin \ ]p,q[_u$ and $x \in \ ]p,q[_s$ conclude analogously.

Now consider {\em $W$-orientation reversing $\phi$}. Here we have to distinguish
between even and odd $n \in \Z$. Since $\phi^2$ is orientation preserving the
above proof carries over for even $n$ if we replace $p^1$ by $p^2$ etc. Thus we
only have to prove the claim for {\em odd} $n$.

If $x \notin\ ]p,q[_u$, $]p,q[_s$ then $\xti\in \ ]\pti, \qti^n[_u\ \cap  \
]\pti, \qti^n[_s$ for all odd $n$.

If $x \in\ ]p,q[_u\ \cap\ ]p,q[_s$ there is $N \in \N_0$ such that $\pti^2 \in\
]\pti, \qti^n[_u \ \cap\ ]\pti, \qti^n[_s$ for all odd $n \geq N$ and $\pti^{-2}
\in\ ]\pti, \qti^n[_u \ \cap\ ]\pti, \qti^n[_s$ for all odd $n \leq -N$.

If $x \in \ ]p,q[_u$ and $x \notin\ ]p,q[_s$ then there is an odd $N \in \N_0$
such that $\pti^2 \in\ ]\pti, \qti^n[_u\ \cap\ ]\pti, \qti^n[_s$ for odd $n \geq
N$ and $\qti^{N+2} \in \ ]\pti, \qti^n[_u\ \cap\ ]\pti, \qti^n[_s$ for odd $n
\leq - N$.

If $x \notin \ ]p,q[_u$ and $x \in\ ]p,q[_s$ conclude analogously.
\end{proof}

Now we prove that for $p \in \mcHpr$ the differential $\mathfrak d$
does not contain infinitely many iterates $m(p,q^n)q^n$ of some primary $q$.
This implies the well-definedness of $\mathfrak d$ and $\del$.

\begin{Proposition}
\label{finite prim iterate}
Let $p$, $q\in \mcHpr$ and $\mcM(p,q) \neq \emptyset$ and set $q^n:=\phi^n(q)$
for $n \in \Z$.
Then
\beqs
\#\{n \in \Z \mid \mcM(p,q^n) \neq \emptyset\} < \infty.
\eeqs
\end{Proposition}

\begin{proof}
Denote by $\pti$, $\qti$, $\qti^n$ etc. the associated points in $\Wti^u \cap
\Wti^s$ and recall that $v \in \mcM(p,q)$ exists if and only its lift $\vti \in
\mcM(\pti, \qti)$ exists.
\refalwaysintersections\ yields the existence of some $N>0$ such that $]\pti,
\qti^n[_u \ \cap \  ]\pti,\qti^n[_s \ \neq \emptyset$ for all $n \in \Z$ with
$\betrag{n} \geq N$. 

Assume $\vti_n \in \mcM(\pti,\qti^n) \neq \emptyset$ for some $n$ with
$\betrag{n} \geq N$. Since $[\pti,\qti^n]_u=\vti_n(B_u)$ and $
[\pti,\qti^n]_s=\vti_n(B_s)$ there is $z_u \in B_u$ and $z_s \in B_s$ such that
$\vti_n(z_u)=\vti_n(z_s)$. Since $\Wti^u$ and $\Wti^s$ do not have
self-intersections it follows $z_u$, $z_s \notin \{(-1,0), (1,0)\}$. Therefore
$\vti_n$ is not globally injective and thus no embedding. The claim now follows
from \refpropclassifind.
\end{proof}

The gluing theorem for primary points is clearly a special case of \refgluing. 
But it is a priori not clear, that the cutting procedure yields two primary
`cutting points' $q_u$ and $q_s$.

\begin{Proposition}[Classification for index difference 2]
\label{prop classif primary}
Let $p$, $r\in \mcHpr$ with $\mu(p,r)=2$ and let $\pti$ and $\rti$ be the
associated points in $\Wti^u \cap \Wti^s$. The possibly arising immersed hearts
$w \in \mcNhat(\pti,\rti)$ appear shadowed in Figure \ref{primary cutting}. $w$
is an embedding apart from the case $(\mu(\pti,\xti),\mu(\xti,\rti))=(1,1)$
where it is not globally injective.
\end{Proposition}

\begin{proof}
In the following, we work with the lifted tangle on $\Mti$. We drop the tilde
associated to symbols on $\Mti$, i.e. we identify $p=\pti$ and $r=\rti$ etc.

Since $[p]=[r]=[x]$ we can write $\mu(p,r)=\mu(p,x)+\mu(x,r)=2$. Now we proceed
as in the proof of \refpropclassifind\ and check the possible combinations for
$(\mu(p,x), \mu(x,r))$. \refindexprim\ restricts the possibilities to
$(\mu(p,x), \mu(x,r)) \in \{(3,-1), (1,1), (-1, 3)\}$
and we recall that the immersions of index $\mu(p,x)$ and $\mu(x,r)$ between $p$
and $x$ and $x$ and $r$ are embeddings. As before, we will consider the cases
$]x,p[_i \ \cap \ ]x,r[_i\ = \emptyset$ or $\neq \emptyset$. If $]x,p[_i \ \cap
\ ]x,r[_i\ \neq \emptyset$ this implies $[x,p]_i \subset [x,r]_i$ or $[x,r]_i
\subset [x,p]_i$ since $W^i$ is free of self-intersections and $\dim W^i =1$.
All possibly arising immersions are listed in Figure \ref{primary cutting}.
\end{proof}

\begin{figure}
\begin{center}

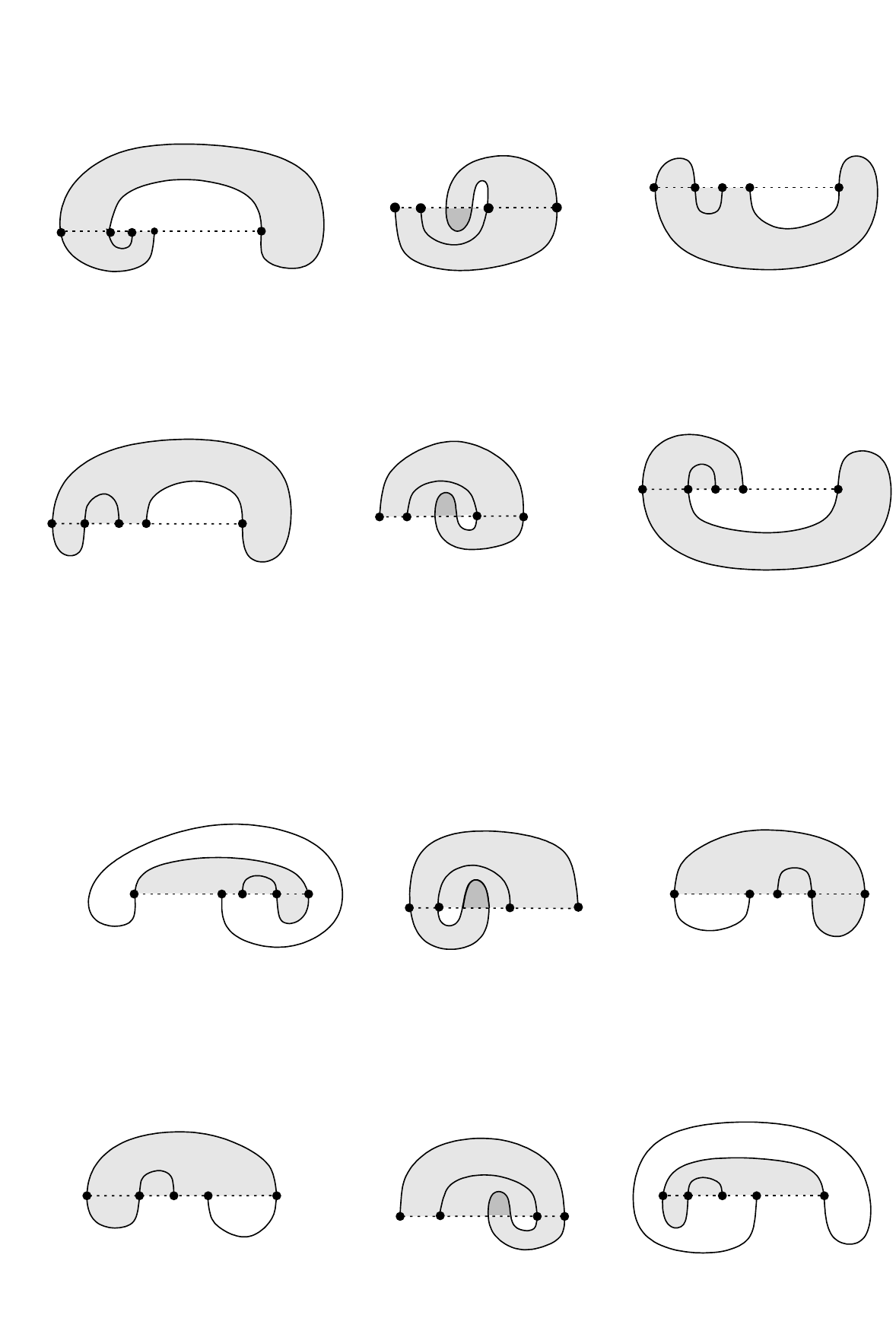

\caption{Cutting for primary points.}
\label{primary cutting}

\end{center}
\end{figure}

It will turn out that for primary $p$ and $r$ with $\mu(p,r)=2$ either {\em
both} cutting points $q_u$ and $q_s$ are primary {\em or none of them}.
In the proof of \refexistenceqq, {\it strongly intersecting} is only needed if
the concave vertex of the heart is the fixed point. Since $x \notin \mcHpr$ we
can drop this assumption in the following statement. Moreover, in
\refexistenceqq, the $\lam$-lemma was applied to the intersection at the concave
vertex of the immersion. Thus it is enough to require only the primary points to
be transverse.

\begin{Theorem}[Cutting for primary points] 
\label{prim cut th}
Let all primary points be transverse and $p$, $ r \in \mcHpr$ with $\mu(p,r)=2$
and $w \in \mcN(p,r)$.
Then there are unique points $q_u$ and $q_s$ such that either both $q_i$ are
primary admitting $v_i \in \mcM(p,q_i)$ and $\vhat_i \in \mcM(q_i, r)$ with
$\vhat_i \#v_i=w$ for $i \in \{s,u\}$ or none of them is primary.
\end{Theorem}

\begin{proof}
It is sufficient to show the claim for the lifted tangle generated by $\Wti^u$
and $\Wti^s$ on $\Mti$. We drop the tilde associated to symbols on $\Mti$ and
identify $\pti=p$ etc.

Let $p$ and $r$ be primary with $\mu(p,r)=2$. 
$p$ and $r$ are transverse intersection points such that the existence (and
uniqueness) of the cutting points $q_u$ and $q_s$ follows from the proof of
\refexistenceqq. But $q_u$ and $q_s$ might be nontransverse. We will prove that
$q_u$ and $q_s$ are either both primary or both nonprimary. If both are primary
then they are, by assumption, transverse and the claim follows from
\refexistenceqq.

\refpropclassifprimary\ together with Figure \ref{primary cutting} describes all
possible immersions of index difference 2 and sketches $q_u$ and $q_s$ and the
cuts to $q_u$ and $q_s$. For simplicity, the $q_i$ are sketched transverse.

Checking the shapes in Figure \ref{primary cutting}, we find that for all cases
$(\mu(p,x), \mu(x,r)) \in \{(3,-1), (-1,3)\}$ the immersion $w$ is an embedding
and that $q_u$ and $q_s$ are both primary.
In the case $]x,p[_u \ \cap \ ]x,r[_u\ \neq \emptyset = ]x,p[_s \ \cap \
]x,r[_s$, we only sketched the case $q_u \in [x,p]_s$, but also $q_u \in
[x,r]_s$ would be primary.
In the case $]x,p[_u \ \cap \ ]x,r[_u\ = \emptyset \neq ]x,p[_s \ \cap \
]x,r[_s$, we have to distinguish $q_s \in [x,p]_u$ or $q_s \in [x,r]_u$, but in
both cases $q_s$ is primary.

Now consider the case $(\mu(p,x), \mu(x,r))=(1,1)$. First we note that $w$ is
not necessarily an embedding. One of the cutting points is the fixed point which
is per definitionem not primary. Furthermore, those segments which join the
other cutting point to $x$ overcross in $p$ or $r$ such that this cutting point
also is nonprimary.
\end{proof}

\begin{proof}[Proof of \refdelfraksquare]
In the following, we work with the lifted tangle on $\Mti$ and drop the tilde
associated to symbols on $\Mti$ and identify $p=\pti$ etc.
We compute for a generator $p \in \mcHpr$
\begin{align*}
\mathfrak d_{m-1}(\mathfrak d_m (p))& 
= \mathfrak d_{m-1}\left(\sum_{\stackrel{q\in \mcHpr}{\mu(q)=\mu(p)-1}}m(p,q)q
\right) \\
& = \sum_{\stackrel{r\in \mcHpr}{\mu(r)=\mu(p)-2}} \ \sum_{\stackrel{q\in
\mcHpr}{\mu(q)=\mu(p)-1}}m(p,q)\cdot m(q,r) r \\
& = \sum_{\stackrel{r\in \mcHpr}{\mu(r)=\mu(p)-2}} \ \left( \sum_{\stackrel{q\in
\mcHpr}{\mu(q)=\mu(p)-1}}m(p,q)\cdot m(q,r)\right) r.
\end{align*}
Thus it is enough to show for fixed $r$
\beqs
\sum_{\stackrel{q\in \mcHpr}{\mu(q)=\mu(p)-1}}m(p,q)\cdot m(q,r)=0.
\eeqs
If all sign products vanish we are done. If $m(p,q)\cdot m(q,r) \neq 0$ both
signs $m(p,q)$ and $m(q,r)$ must be nonzero. In that case $\mcMhat(p,q)$ and
$\mcMhat(q,r)$ are nonempty and by the gluing construction $\mcNhat(p,r)$ is
nonempty. \refprimcutth\ tells us that for fixed $p$ and $r$ there are either
exactly two primary cutting points $q_u$ and $q_s$ or none. We are in the first
case since our $q$ is one of them. Since $m(p,q) \cdot m(q,r)=0$ for all $q \neq
q_u$, $q_s$ the sum simplifies to 
\beqs
m(p,q_u) \cdot m(q_u,r) + m(p,q_s) \cdot m(q_s,r)
\eeqs
which vanishes since $m(p,q_u) \cdot m(q_u,r) = - m(p,q_s) \cdot m(q_s,r)$ by
\refsignskewsym.
\end{proof}

\section{Examples}
\label{chapterexamples}

\subsection{Computation of examples}

In this subsection, we discuss the aptitude and accessibility of primary Floer
homology for explicit computations. We calculate two examples which arise from a
slight perturbation of the integrable systems sketched in Figure
\ref{figureeight}. For simplicity assume to be in $\R^{2n}$, i.e. the sets of
semi-primary and primary points coincide.

\begin{figure}[h]
\begin{center}

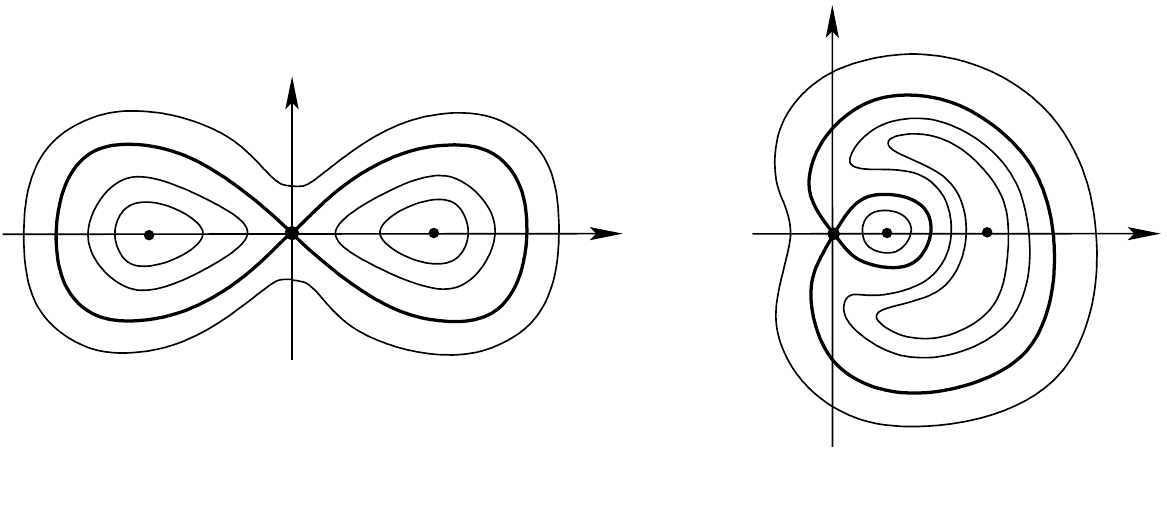

\caption{`Figure eight' and `tilted figure eight' homoclinic loops.}
\label{figureeight}

\end{center}
\end{figure}

\vsp

If we want to compute the primary Floer homology of an explicit tangle we have
to locate the primary points. 
Given a pair of intersecting branches of $W^s$ and $W^u$, we start at the fixed
point $x$ and run simultanously along both branches until they intersect for the
first time. This intersection point $p$ is primary. By \refpositionprim, all
other primary points arising from this pair of branches have exactly one
representant in $]p,p^1[_s\ \cap\ ]p,p^1[_u$. Since all primary points are
transverse there is only a finite number of primary equivalence classes and we
locate their representants in $]p,p^1[_s\ \cap\ ]p,p^1[_u$ applying successively
\refprimarymax. If we proceed in this way for all pairs of intersecting branches
we obtain representatives for all primary equivalence classes.

\vsp

To discover for a given $p$ all $q$ with $\mcM(p,q) \neq \emptyset$ is a little
bit more tedious. \reffiniteprimiterate\ and \refalwaysintersections\ assure
that there is only a finite number of canditates and that they are `not to far
away' from $p$. Therefore it remains to check those candidates.

\vsp

Thus, primary Floer homology is entirely determined by sufficiently large,
fixed, {\em compact} segments of $W^s$ and $W^u$ centered around $x$. Therefore
primary Floer homology can always be computed --- one only needs to plot a
finite part of the tangle with sufficient accuracy. Altogether, primary Floer
homology provides finite information of an infinite chaotic tangle.

\vsp

Using the computation of the two examples below, we will assign homology groups
also to the homoclinic loops displayed in Figure \ref{figureeight}. This will be
done in \refhomoclinicloops\ using the invariance property of primary Floer
homology.

\subsection{Figure-eight example}

We compute the primary Floer homology of the schematic tangle in Figure \ref{hom
ex 1}. Such a tangle might arise from a figure-eight homoclinic loop of an
integrable system (Figure \ref{figureeight}) by means of the Melnikov method.
The hyperbolic fixed point $x$ and the elliptic fixed points $y$ and $\yti$ are
printed extra bold and the Maslov grading of the primary points is given. There
are eight equivalence classes $\langle p \rangle$, $\langle \pti \rangle$,
$\langle q \rangle$, $\langle b \rangle$, $\langle \qti \rangle$, $\langle \bti
\rangle$, $\langle r \rangle$ and $\langle \rti \rangle$ with $\mu(\langle p
\rangle)= \mu(\langle \pti \rangle)=-1$, $\mu(\langle q \rangle)= \mu(\langle b
\rangle)= \mu(\langle \qti \rangle)= \mu(\langle \bti \rangle)=-2$ and
$\mu(\langle r \rangle)= \mu(\langle \rti \rangle)=-3$. 
Using the $m(p,q)$-signs we obtain

\begin{figure}
\begin{center}

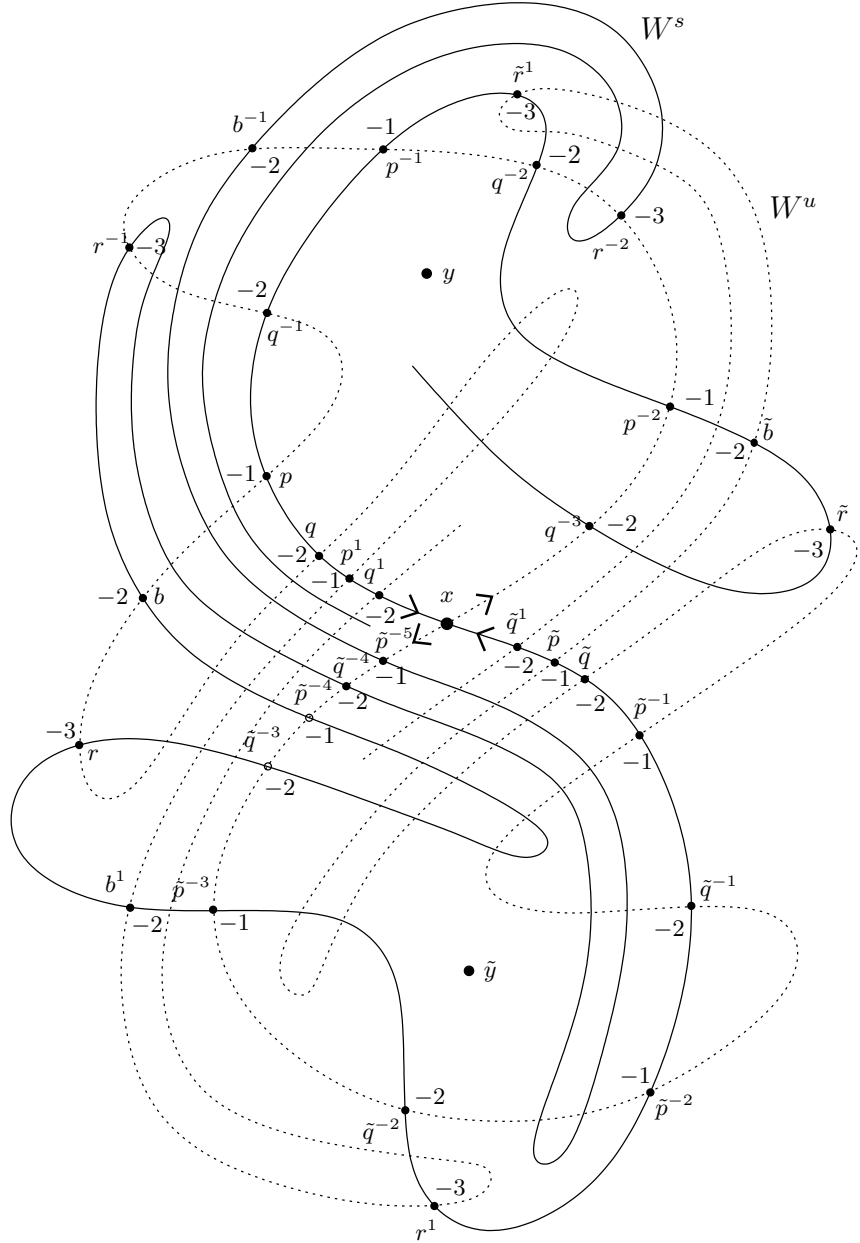

\caption{A `figure eight' homoclinic tangle.}
\label{hom ex 1}

\end{center}
\end{figure}

\begin{align*}
& \del \langle p \rangle = \langle q \rangle - \langle q^{-1} \rangle + \langle
b \rangle - \langle \bti^2 \rangle = \langle b \rangle - \langle \bti \rangle,
&& \del \langle \qti \rangle = \langle r^3 \rangle - \langle \rti \rangle \\
& && = \langle r \rangle - \langle \rti \rangle ,
 \\
& \del \langle \pti \rangle = -\langle \qti \rangle + \langle \qti^{-1} \rangle
+ \langle \bti \rangle - \langle b^4 \rangle = - \langle b \rangle + \langle
\bti \rangle, 
&& \del \langle \bti \rangle = -\langle \rti \rangle + \langle \rti^1 \rangle
=0,
\\
& \del \langle q \rangle= - \langle r \rangle + \langle \rti^3 \rangle = -
\langle r \rangle + \langle \rti \rangle , 
&& \del \langle r \rangle =0,
\\
& \del \langle b \rangle = \langle r \rangle - \langle r^{-1} \rangle =0, 
&& \del \langle \rti \rangle =0.
\end{align*}
and using the $n(p,q)$-signs with the orientation on $W^u$ induced by setting
$x<_u p$ we obtain
\begin{align*}
& \del \langle p \rangle = \langle q \rangle - \langle q^{-1} \rangle + \langle
b \rangle - \langle \bti^2 \rangle = \langle b \rangle - \langle \bti \rangle, 
&& \del \langle \qti \rangle = \langle r^3 \rangle + \langle \rti \rangle \\
& && = \langle r \rangle + \langle \rti \rangle , \\
& \del \langle \pti \rangle = \langle \qti \rangle - \langle \qti^{-1} \rangle -
\langle \bti \rangle + \langle b^4 \rangle = \langle b \rangle - \langle \bti
\rangle, 
&& \del \langle \bti \rangle = \langle \rti \rangle - \langle \rti^1 \rangle =0,
\\
& \del \langle q \rangle= - \langle r \rangle - \langle \rti^3 \rangle = -
\langle r \rangle - \langle \rti \rangle = - (\langle r \rangle + \langle \rti
\rangle), 
&& \del \langle r \rangle =0,
\\
& \del \langle b \rangle = \langle r \rangle - \langle r^{-1} \rangle =0, 
&& \del \langle \rti \rangle =0.
\end{align*}

The different signs lead to different boundary operators and different
generators of the homologies. Nevertheless, they turn out to be isomorphic:

\begin{align*}
& H_{-1}(x,\phi, m\mbox{-signs})= \Z(\langle p \rangle + \langle \pti \rangle)
\simeq \Z (\langle p \rangle - \langle \pti \rangle)=H_{-1}(x,\phi,
n\mbox{-signs}) , \\
& H_{-2}(x,\phi, m\mbox{-signs})=\frac{\Z\langle b \rangle \oplus \Z \langle
\bti \rangle\oplus \Z (\langle q \rangle + \langle \qti \rangle )}{\Z (\langle b
\rangle - \langle \bti \rangle)} = H_{-2}(x,\phi, n\mbox{-signs}),\\
& H_{-3}(x,\phi, m\mbox{-signs})= \frac{\Z \langle r \rangle \oplus \Z \langle
\rti \rangle}{\Z(\langle r \rangle - \langle \rti \rangle)} \simeq \frac{\Z
\langle r \rangle \oplus \Z \langle \rti \rangle}{\Z(\langle r \rangle + \langle
\rti \rangle)}= H_{-3}(x,\phi, n\mbox{-signs}).
\end{align*}

\subsection{Tilted figure-eight example}

By perturbing a tilted figure-eight homoclinic loop as in Figure
\ref{figureeight}, the homoclinic tangle of Figure \ref{hom ex 2} might arise. 

\begin{figure}
\begin{center}

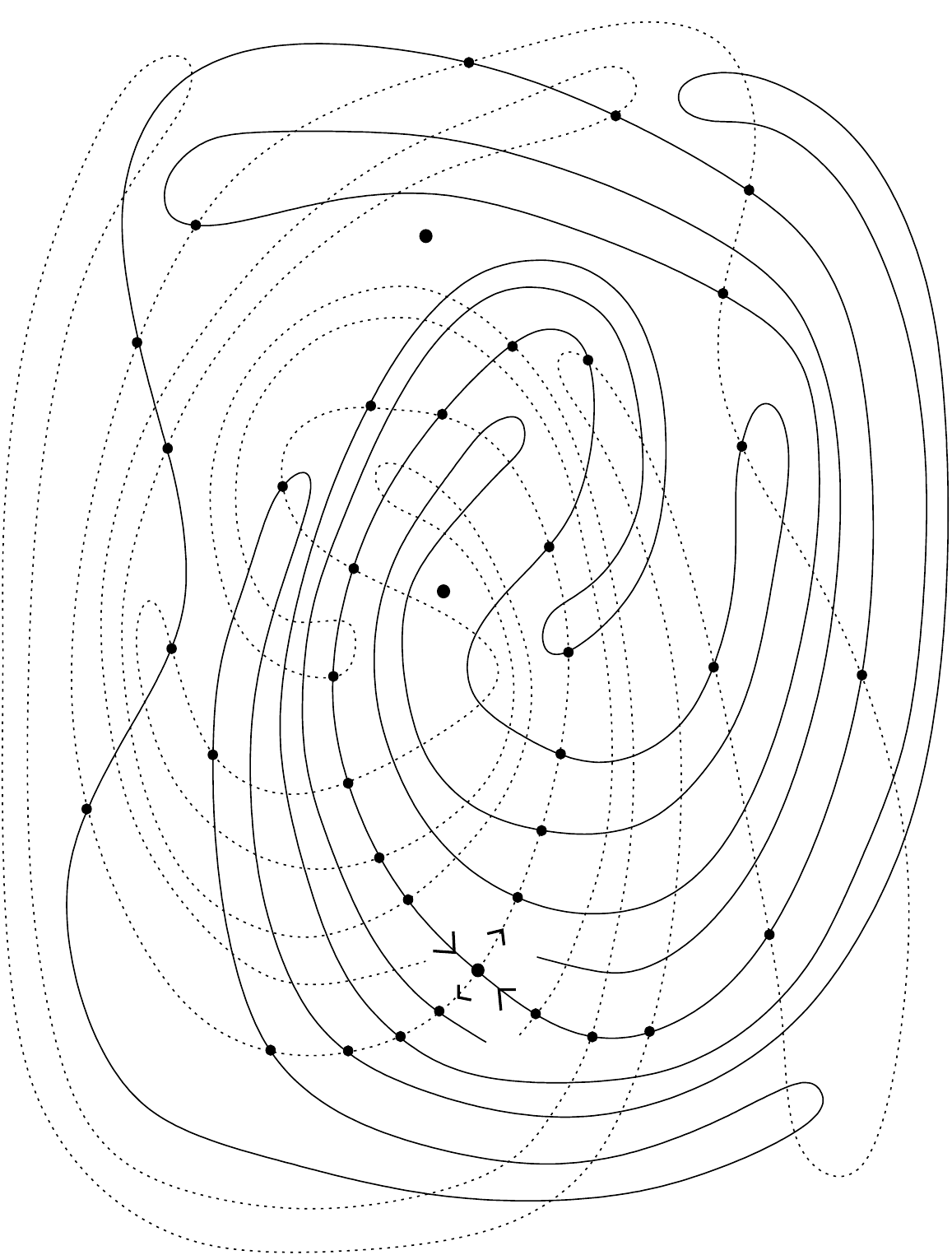

\caption{A `tilted figure eight' homoclinic tangle.}
\label{hom ex 2}

\end{center}
\end{figure}

\vsp

There are the eight equivalence classes $\langle \pti \rangle$, $\langle \qti
\rangle$, $\langle s \rangle$, $\langle \sti \rangle$, $ \langle r \rangle$,
$\langle \rti \rangle$, $\langle p \rangle$ and $\langle q \rangle$ with Maslov
index
$\mu(\langle \pti \rangle)=3$, $\mu(\langle \qti \rangle)= \mu(\langle s
\rangle)= \mu(\langle \sti \rangle)=2$, $\mu( \langle r \rangle)= \mu(\langle
\rti \rangle)=1$, $\mu(\langle p \rangle)=-1$ and $\mu(\langle q \rangle)=-2$.
Using the $m(p,q)$-signs, we obtain as boundary operator
\begin{align*}
& \del \langle \pti \rangle =  -\langle \qti \rangle + \langle \qti^1 \rangle -
\langle s \rangle + \langle \sti^{-3} \rangle= -\langle s \rangle + \langle \sti
\rangle , 
&& \del \langle r \rangle =0, \\
& \del \langle \qti \rangle = \langle r^{-1} \rangle - \langle \rti^{-3} \rangle
= \langle r \rangle - \langle \rti \rangle,
&& \del \langle \rti \rangle =0, \\
& \del \langle s \rangle = \langle r \rangle - \langle r^{-1} \rangle = 0,
&& \del \langle p \rangle = \langle q \rangle - \langle q^{-1} \rangle =0, \\
& \del \langle \sti \rangle = - \langle \rti \rangle + \langle \rti^1 \rangle
=0,
&& \del \langle q \rangle =0.
\end{align*}
For the $n(p,q)$-signs, we fix an orientation of $W^u$ via chosing a
parametrization in direction from $x$ to $p$. This leads to
\begin{align*}
& \del \langle \pti \rangle =  \langle \qti \rangle - \langle \qti^1 \rangle +
\langle s \rangle - \langle \sti^{-3} \rangle= \langle s \rangle - \langle \sti
\rangle , 
&& \del \langle r \rangle =0, \\
& \del \langle \qti \rangle = \langle r^{-1} \rangle + \langle \rti^{-3} \rangle
= \langle r \rangle + \langle \rti \rangle,
&& \del \langle \rti \rangle =0, \\
& \del \langle s \rangle = \langle r \rangle - \langle r^{-1} \rangle = 0,
&& \del \langle p \rangle = \langle q \rangle - \langle q^{-1} \rangle =0, \\
& \del \langle \sti \rangle = \langle \rti \rangle - \langle \rti^1 \rangle =0,
&& \del \langle q \rangle =0.
\end{align*}
As homology, we obtain $H_l=0$ for all $l \notin \{\pm 1, \pm2\}$ regardless of
the chosen signs. But for the remaining four groups we compute
\begin{gather*}
H_{2}(x,\phi, m\mbox{-signs})= \frac{\Z \langle s \rangle \oplus \Z \langle \sti
\rangle}{\Z (-\langle s \rangle + \langle \sti \rangle)} \simeq \frac{\Z \langle
s \rangle \oplus \Z \langle \sti \rangle}{\Z (\langle s \rangle - \langle \sti
\rangle)}   =H_{2}(x,\phi, n\mbox{-signs}) , \\
H_{1}(x,\phi, m\mbox{-signs})= \frac{\Z \langle r \rangle \oplus \Z \langle \rti
\rangle}{\Z(\langle r \rangle - \langle \rti \rangle)} \simeq \frac{\Z \langle r
\rangle \oplus \Z \langle \rti \rangle}{\Z(\langle r \rangle + \langle \rti
\rangle)}  = H_{1}(x,\phi, n\mbox{-signs}),\\
H_{-1}(x,\phi, m\mbox{-signs})= \Z \langle p \rangle  = H_{-1}(x,\phi,
n\mbox{-signs}) \\
H_{-2}(x,\phi, m\mbox{-signs})= \Z \langle q \rangle = H_{-2}(x,\phi,
n\mbox{-signs}).
\end{gather*}

\section{Invariance}

In classical Lagrangian Floer theory, invariance of the homology under
Hamiltonian perturbations of the underlying Lagrangians is an important feature.
Thus one can choose a particular nice Lagrangian within the Hamiltonian isotopy
class for the computation of the homology.

\vsp

Our situation differs strongly from the classical one. Whereas in the classical
situation a Hamiltonian diffeomorphism $f$ is applied {\em directly} to a
Lagrangian $L$ changing it to $f(L)$, the change here occurs {\em indirectly}.
We are going to perturb the underlying symplectomorphism $\phi$ which results in
changing both the stable and unstable manifolds. Our invariance proof is
inspired by Floer's original proof in \cite{floer3} which uses explicit chain
homotopies. The more modern ansatz via homotopy of homotopies is unfortunately
not applicable since it is not compatible with the bifurcation nature of primary
points. We will use the invariance to assign homology groups to the homoclinic
loops in Figure \ref{figureeight}.

\vsp

Primary points are printed extra bold in figures. In order to obtain smaller
sketches, we sometimes draw the \hyp\ fixed point $x$ `splitted' into two
copies.
In this section, $(M,\om)$ is a closed symplectic two-dimensional \mf\ with
genus $g \geq 1$.
The group of smooth diffeomorphisms $\Diff(M)$ is endowed with the Whitney
topology (which coincides on \cpt\ \mf s with the $C^r$-topology) and
$\Diff_\om(M) \subset \Diff(M)$ is the group of symplectomorphisms.

\subsection{Main results}

Let $\phi \in \Diff^k(M)$ with $k \geq 1$ and $x \in \Fix(\phi)$ hyperbolic and
$\psi \in \Diff^k(M)$ sufficiently $C^k$-near to $\phi$. Then it is wellknown
that $\psi$ has a \hyp\ fixed point $y$ near $x$. $W^i(y, \psi)$ is $C^k$-near
$W^i(x, \phi)$ for $ i\in \{u,s\}$, at least on \cpt\ neighbourhoods of $y$ and
$x$ in $W^i(y, \psi)$ and $W^i(x, \phi)$.
$y$ is called the {\em continuation} of $x$ and the signs of the corresponding
eigenvalues coincide. 

Let $\phi$, $\psi \in \Diff_\om(M)$ and $x\in \Fix(\phi)$ and $y \in \Fix(\psi)$
both \hyp.
An {\em isotopy (between $(x,\phi)$ and $(y, \psi)$)} is a smooth path $\Phi:
[0,1] \to \Diff_\om(M)$, $ \tau \mapsto \Phi(\tau)=:\Phi_\tau$ with
$\Phi_0=\phi$, $\Phi_1=\psi$, $x_0=x$ and $x_1=y$ and $x_\tau \in
\Fix(\Phi_\tau)$ as continuation between $x$ and $y$ for all $\tau \in [0,1]$.
Attaching $\tau$ to a symbol associates it to $(x_\tau,\Phi_\tau)$, i.e.
$\mcHpr^\tau$ denotes the set of primary points of $(x_\tau,\Phi_\tau)$ etc.
$(x, \phi)$ is called {\em contractibly strongly intersecting (csi)} if $W^u$
and $W^s$ are strongly intersecting and if each pair of branches has
contractible homoclinic points. An isotopy $\Phi$ is csi if $(x_\tau,\Phi_\tau)$
is csi for all $\tau \in [0,1]$.

\begin{Theorem}[Invariance]
\label{inv th cpt}
Let $(M,\om)$ be a closed symplectic two-dimensional \mf\ with genus $g \geq
1$. 
Let $\phi$, $\psi \in \Diff_\om(M)$ with \hyp\ fixed points $x \in \Fix(\phi)$
and $y \in \Fix(\psi)$. Let $(x, \phi)$ and $(y, \psi)$ be csi and let all
primary points of $\phi$ and $\psi$ be transverse.
Assume there is a csi isotopy $\Phi$ from $(x, \phi)$ to $( y, \psi)$.
Then
\beqs
H_*(x,\phi) \simeq H_*(y, \psi).
\eeqs
\end{Theorem}

We will prove \refinvthcpt\ in the following subsections. The proof carries over
to \cpt ly supported symplectomorphisms on $\R^2$.
`Csi' and `compactly supported' are crucial:

\begin{Remark}
\label{branch apart}
There are $\phi$, $\psi \in \Diff_{dx \wedge dy}(\R^2)$ with \hyp\ fixed points
$x \in \Fix(\phi)$ and $y \in \Fix(\psi)$ which can be joint by a symplectic
isotopy and which have
\begin{enumerate}[(1)]
\item
different number of pairs of intersecting branches,
\item
$H_*(x, \phi) \neq H_*(y, \psi)$.
\end{enumerate}
\end{Remark}

\begin{proof}
For small $\ep >0$, consider the path $  \Phi^\ep : [0,1] \to \Diff_{dx\wedge
dy}(\R^2)$ given by
\beqs
\Phi_\tau^\ep(x,y):=(x+y+ \ep f_\tau(x), y+ \ep f_\tau(x))
\eeqs
with $f_\tau(x):=- \tau x^3 - (1-\tau) x^2 +x$. We have
$\Phi_\tau^\ep(0,0)=(0,0)$ with 
$D\Phi_\tau^\ep(0,0)
= \left(
\begin{smallmatrix}
1+ \ep & 1 \\
\ep & 1
\end{smallmatrix}
 \right)$ as hyperbolic fixed point. Now set $\phi:=\Phi_0^\ep$ and
$\psi:=\Phi_1^\ep$. $\phi$ is the {\em volume preserving H\'enon map} and its
homoclinic tangle is sketched in Figure \ref{quadr cubic} (a): $\phi$ has one
pair of intersecting branches. The tangle of $\psi$ is sketched in Figure
\ref{quadr cubic} (b) and has four pairs of intersecting branches.
We compute
$H_2((0,0), \phi) \simeq \Z$, $H_1((0,0),  \phi) \simeq \Z$ and $ H_n((0,0), 
\phi)=0$ otherwise. But $\psi$ has $H_3 ((0,0), \psi) \neq 0$, thus $H_*((0,0),
\phi) \neq H_*((0,0), \psi)$.
\end{proof}

\begin{figure}[h]
\begin{center}

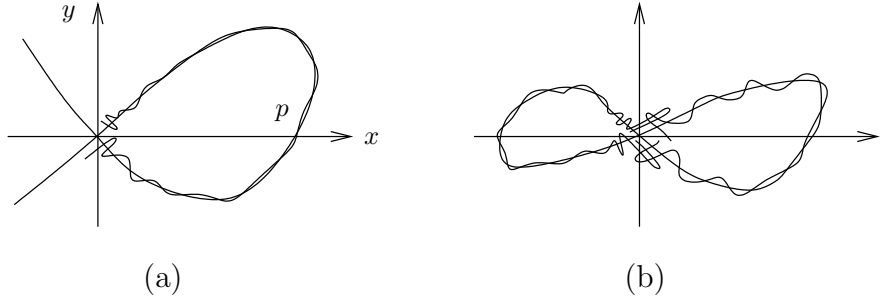

\caption{Homoclinic tangle of the quadratic map (a) and the cubic map (b).}
\label{quadr cubic}

\end{center}
\end{figure}

As an application, we obtain the following statement.

\begin{Corollary}[existence and bifurcation criterion]
%\label{inv corol}
%\newcommand{\refinvcorol}{Corollary \ref{inv corol}}
Assume the conditions of \refinvthcpt\ for $(M, \om)$, $(x,\phi)$ and
$(y,\psi)$, but $H_*(x,\phi) \neq H_*(y, \psi)$. Then $(x, \phi)$ and $(y,\psi)$
cannot be joint by a csi isotopy.
Thus, if there is a path $(\Phi_\tau)_{\tau \in [0,1]} \in \Diff_{\om}(M)$
between $\phi$ and $\psi$ then
\begin{enumerate}[(1)]
\item
{\em either} $\Phi$ is no isotopy, i.e. there is $\tau_0 \in [0,1]$ where
$x_{\tau_0}$ vanishes or undergoes a bifurcation,
\item
{\em or}, if $\Phi$ is an isotopy, there has to be a pair of branches and some
$\tau_0 \in [0,1]$ where all contractible homoclinic points vanish, i.e. there
are homoclinic bifurcations.
\end{enumerate}
\end{Corollary}

Now apply the invariance property to assign homology groups to homoclinic loops
of autonomous Hamiltonian systems. Let $H: \R^2 \to \R$ be a compactly
supported, autonomous smooth Hamiltonian function and $X$ its Hamiltonian vector
field. Assume $X$ to have locally a phase portrait like Figure \ref{figureeight}
(a) or (b). Let $\ep >0$ be small and let $Y$ be a smooth nonautonomous vector
field with support in the support of $H$. 

By Melnikov's method \cite{melnikov}, \cite{guckenheimer-holmes}, the time-one
map $\phi_\ep$ of the nonautonomous system $\zdot(t)=X(z(t)) + \ep Y(t, z(t))$
now has a homoclinic tangle instead of a homoclinic loop. The Melnikov function
measures the existence and `width' of the arising homoclinic tangle (for higher
iterates see Rom-Kedar \cite{rom-kedar1, rom-kedar2}). Therefore one knows quite
well how the tangle behaves for $\ep \to 0$. 

In case of Figure \ref{figureeight} (a), the tangle looks roughly like the one
of Figure \ref{hom ex 1}. In case of Figure \ref{figureeight} (b), the tangle
resembles somewhat Figure \ref{hom ex 2}. The isotopy $\ep \mapsto \phi_\ep$
satisfies the requirements of \refinvthcpt\ for small $\ep >0$. The natural
parametrization of the homoclinic loops induces the signs of the Maslov indices:
If the loop winds (counter)clockwise, the arising branches of the unstable and
stable manifold have primary points with positive (negative) Maslov index. We
summarize this as follows.

\begin{Corollary}
\label{homoclinicloops}
Figure-eight homoclinic loops are characterized by
\beqs
H_{\si \cdot 1}= \Z, \qquad H_{\si \cdot  2}= \Z\oplus \Z, \qquad H_{\si \cdot 
3}=\Z
\eeqs
where $\si=1$ for clockwise and $\si= -1$ for counterclockwise parametrization.
In case of the tilted figure-eight homoclinic loop, both cases lead, due to
symmetry, to
\beqs
H_{-2}= H_{-1}=H_{1}=H_2=\Z , \qquad H_{\pm 3}=0.
\eeqs
\end{Corollary}

In \refinvthcpt, we impose the transversality condition only on the primary
points of $(x,\phi)$ and $(y,\psi)$. This is convenient for applications since
it can be checked easily using \refpositionprim\ and \refprimarymax. But our
proof strategy requires perturbations of $\Phi$.
Thus we have to show that slight perturbations of the start and endpoint
preserve the associated primary Floer homologies.

\begin{Proposition}
\label{start end}
Let $(M,\om)$ be a closed, two-dimensional symplectic \mf\ with genus $g\geq 1$.
Let $\phi \in \Diff_{\om}(M)$ and $x \in \Fix(\phi)$ \hyp. Let $(x, \phi)$ be
csi and all primary points transverse. Then, for all $\phihat\in \Diff_\om(M)$
sufficiently close to $\phi$, holds
\beqs
H_*(x, \phi)=H_*(\xhat, \phihat).
\eeqs
where $\xhat \in \Fix(\phihat)$ is the continuation of $x$.
\end{Proposition}

The proof of \refstartend\ is postponed to \refproofstartend.

\begin{Remark}
\label{comment}
\begin{enumerate}[(1)]
\item
Whereas, in the two-dimensional setting, primary Floer homology can be defined
also for nonsymplectic \diffeo s, invariance only is natural within the class of
symplectomorphisms. 
\item
We conjecture that Hamiltonian diffeomorphisms are naturally csi.
\item
In contrast to classical Lagrangian Floer theory, invariance of primary Floer
homology relies on the nontrivial result of (generical) existence of
intersection points of the Lagrangians.
\end{enumerate}
\end{Remark}

\subsection{Outline of the proof of \refinvthcpt}

In order to prove \refinvthcpt, we have to deal with bifurcations of homoclinic
points. We denote the bifurcation parameter by $\tau$ with tangency at $\tau_0$.
Generically, a homoclinic tangency is simple and the picture looks (after a
suitable symplectic coordinate change) locally like the graph of
$f+C(\tau-\tau_0)$ intersecting the $x$-axis where $f$ is a quadratic,
homogenous and nondegenerate function and $C>0$.

Passing from $\tau<\tau_0$ to $\tau>\tau_0$ (or vice versa), we briefly call a
{\em move} and omit the bifurcation parameter $\tau$. By abuse of notation, we
speak of an $(r,s)$-move if the arising resp.\ vanishing two points are called
$r$ and $s$. In fact, we always have a family of $(r^n, s^n)_{n \in \Z}$-moves.
Given an $(r,s)$-move, there is always an embedded di-gon between $r$ and $s$
since $]r,s[_u\ \cap\ ]r,s[_s\ =\emptyset$. If $r$ and $s$ are primary then they
are adjacent to each other. Moreover, $x \notin [r,s]_u \cup [r,s]_s$ and
therefore $r$ and $s$ always lie on the same branches.

\begin{proof}[Proof of \refinvthcpt] 
First, we perturb the isotopy $\Phi$ in order to obtain the above mentioned
generic bifurcations. This is possible since csi is an open property in
$\Diff_\om(M)$. Since the set of all $\tau \in [0,1]$, for which all primary
points are nontransverse, is discrete we can perturb the isotopy once again
slightly, in order to obtain for all $\tau \in [0,1]$ a transverse primary point
within each pair of intersection branches.

Recall the properties of primary points from \refpositionprim\ and
\refprimarymax. Given a primary $p$, we call $[p,p^1]_u \cup [p,p^1]_s$ together
with the positions of $[p,p^1]_u \cap [p,p^1]_s$ and the immersions (embeddings
on $\Mti$) between adjacent points the {\em frame induced by $p$}. Every primary
equivalence class associated to that pair of branches and different from $p$ has
exactly one representative in the frame induced by $p$. As long as $p$ persist
as primary point its frame allows to observe all primary bifurcations.
Now we cover $[0,1]$ by a finite number of overlapping intervals associated to
persisting primary points.

\vsp

Primary points can only arise resp.\ vanish at certain distinguished parts of the
compact frame segments $[p,p^1]_u$ and $[p,p^1]_s$, see \refprimarymax\ and
Figure \ref{mixed}. Thus there are only finitely many $\tau \in [0,1]$, where
primary points can arise resp.\ vanish. As \refsecunimport\ and \refsignequal\
will show, bifurcations of secondary points either do not affect the chain
complex or are coupled to certain primary bifurcations. Since primary Floer
theory lives within compact segments centered around the fixed point, we can
model the relevant bifurcations of the isotopy as a sequence of moves as in knot
theory.

\vsp

Now let us discuss the moves more detailed.
W.l.o.g. we will assume from now on that, in case of a bifurcation in $p$ at
time $\tau_0$, the tangency $p$ unfolds into two points for $\tau > \tau_0$ and
vanishes for $\tau < \tau_0$. This we briefly call {\em after} resp.\ {\em
before} the bifurcation or move. We call a point {\em involved in a move} if it
is either the homoclinic tangency at time $\tau_0$ or one of the arising
transverse homoclinic points. Persistent transverse primary points $p$ and $q$
are called {\em combinatorically affected by a move} if the value of $m(p,q)$ is
changed by the move. By abuse of notation, we call in this case also the
elements of $\mcM(p,q)$ affected by the move.

There are two possibilities to generate (analogously destroy) a primary point
$p$ by a move:
\begin{enumerate}[a)]
\item
$p$ arises as intersection point.
\item
$p$ was secondary and becomes primary. This phenomenon we call a {\em
primary-secondary flip}, briefly a {\em flip}.
\end{enumerate}
In the latter case, the point does not necessarily need to be involved in the
move itself, cf.\ Figure \ref{mixed}. 
Primary points cannot switch to nontrivial homotopy classes or vice versa due to
`$\dots \cap \mcH_{[x]}$' in the definition of `primary'.

Since there are always two points involved in a bifurcation the following types
of moves are possible:
\begin{enumerate}[a)]
\item
If both arising points are primary the move is called {\em primary}.
\item
If one of the arising points is primary and the other one secondary the move is
called {\em mixed}.
\item
If both arising points are secondary the move is called {\em secondary}.
\end{enumerate}

We note

\begin{Lemma}
\label{flip via mixed}
Let $p$ be not involved itself in a given move, but let $p$ undergo a
primary-secondary flip. Then the move is a mixed one.
\end{Lemma}

\begin{proof}[Proof of \refflipviamixed]
Consider the lifted tangle.
The situation is sketched in Figure \ref{prim_second}. In (i), $p$ is primary
before the move. In (ii -- iv), the possible types of moves are listed which
turn $p$ secondary --- all of them are mixed.
\end{proof}

\begin{figure}[h]
\begin{center}

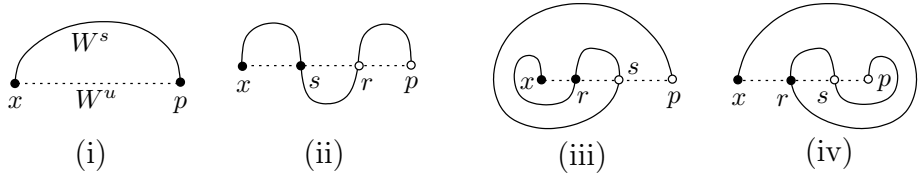

\caption{Causes for a primary-secondary flip.}
\label{prim_second}

\end{center}
\end{figure}

Having \refflipviamixed\ and Figure \ref{prim_second} in mind, we conclude the
following changes of the set of primary points under the different types of
moves.

\begin{Corollary}
\label{move char}
\begin{enumerate}[(1)]
\item
A primary move generates two primary points and does not flip any.
\item
A mixed move generates one primary point, but flips a certain number of primary
points secondary.
\item
A secondary move neither generates primary points nor can flip some of them
secondary, i.e. the set of primary points stays untouched.
\end{enumerate}
\end{Corollary}

\refmovechar\ characterizes how the different types of moves affect the
generator set of the chain groups. We will inquire about the potential changes
of the boundary operator in the next subsections. It will be proven in
\refsecunimport, \refrsprimaryisom, \refprimmoveinvar\ and \refmixedinvar\ that
all three kinds of moves leave the homology invariant which proves \refinvthcpt.
\end{proof}

\subsection{Invariance under secondary moves}

For simplicity, we work with the lifted homoclinic tangle on the universal
cover. According to \refmovechar, the generator set of the chain complex stays
unchanged under secondary moves and we will show now that this is also true for
the boundary operator.

\begin{figure}[h]
\begin{center}

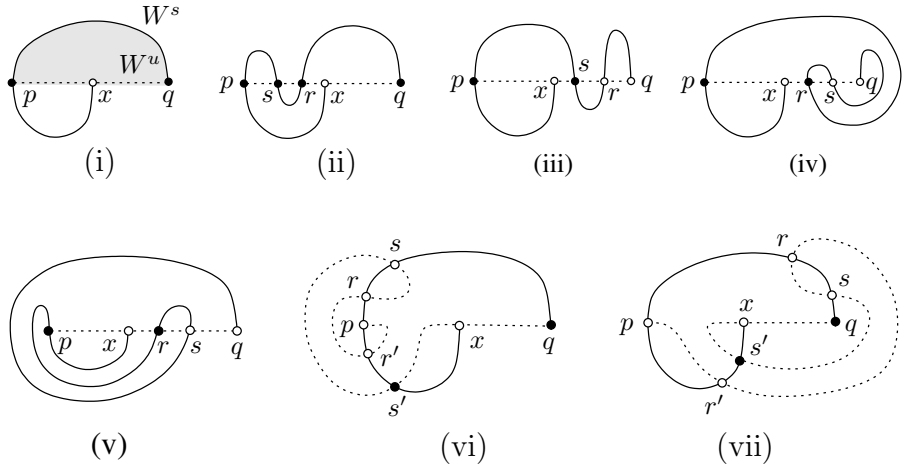

\caption{The effect of moves on an embedding between primary points $p$ and $q$.}
\label{rs perturb fish}

\end{center}
\end{figure}

\begin{Proposition}
\label{sec unimport}
Secondary moves do not affect embeddings between primary points.
\end{Proposition}

\begin{proof}
We argue by contradiction: 
Let $u$ be an embedding between primary points $p$ and $q$ with $\mu(p,q)=1$.
Consider an $(r,s)$-move such that $\{r,s\}=\ ]p,q[_u \ \cap \ ]p,q[_s$.  We
show: If $r$ and $s$ are secondary then the $(r,s)$-move already flipped either
$p$ or $q$ secondary before $r$ and $s$ could arise.
The proof is tedious, but elementary. We just have to check for the embeddings
of Figure \ref{all immersions} all combinatorial possibilities of $(r,s)$-moves
affecting the boundary $[p,q]_u \cup [p,q]_s$ such that $\{r,s\}=\ ]p,q[_u \
\cap \ ]p,q[_s$.

We only prove the assertion exemplarily in the case of Figure \ref{all
immersions} (i) with $\mu(p,q)=(-1,2)$ which is resketched in Figure \ref{rs
perturb fish} (i). The strategy and result for the other cases in Figure
\ref{all immersions} are the same.

Consider Figure \ref{rs perturb fish} (i) and the boundary $[p,q]_u \cup
[p,q]_s$ of the embedding between $p$ and $q$. $]p,q[_u\backslash \{x\}$
consists of the two connected components $]p,x[_u$ and $]q,x[_u$. Since $r$ and
$s$ always lie in the same branch we have to distinguish the cases $r$, $s \in\
]p,x[_u$ (see Figure \ref{rs perturb fish} (ii), (vi), (vii)) and $r$, $s \in\
]q,x[_u$ (see Figure \ref{rs perturb fish} (iii), (iv), (v)). Moreover, we have
to distinguish if $p$ is connected within $W^s$ first to $s$ (see Figure \ref{rs
perturb fish} (ii), (iii)) or to $r$ (see Figure \ref{rs perturb fish} (iv) --
(vii)). The cases (iv) and (v) on the one hand and (vi) and (vii) on the other
hand are basically the same. 

We deduce that (ii) is a primary move and that (iii), (iv) and (v) are mixed
ones. In (vi) and (vii) the points $r$ and $s$ are both secondary. But before
the move, starting in the situation of (i), generates the intersection points
$r$ and $s$ in (vi) and (vii) it has to pass through $]p,x[_s$ generating the
intersection points $r'$ and $s'$ which yields a mixed $(r',s')$-move flipping
$p$ secondary. 
\end{proof}

%\begin{Corollary}
%\label{secondary invar}
%\newcommand{\refsecondaryinvar}{Corollary \ref{secondary invar}}
\refsecunimport\ and \refmovechar\ imply the invariance of primary Floer
homology under secondary moves.
%\end{Corollary}
Moreover, we note that, according to the proof of \refsecunimport, a mixed move
affecting an embedding between two primary points always flips one of them
secondary.

\subsection{Invariance under primary moves}

Now we prove the invariance of primary Floer homology under primary moves. The
proof generalizes ideas of Floer \cite{floer3} and de Silva \cite{de silva}.

\vsp

In Figure \ref{rs primary}, the two possibilities for primary moves are sketched
which are deduced from Figure \ref{all immersions} (up to symmetries).

\begin{figure}[h]
\begin{center}

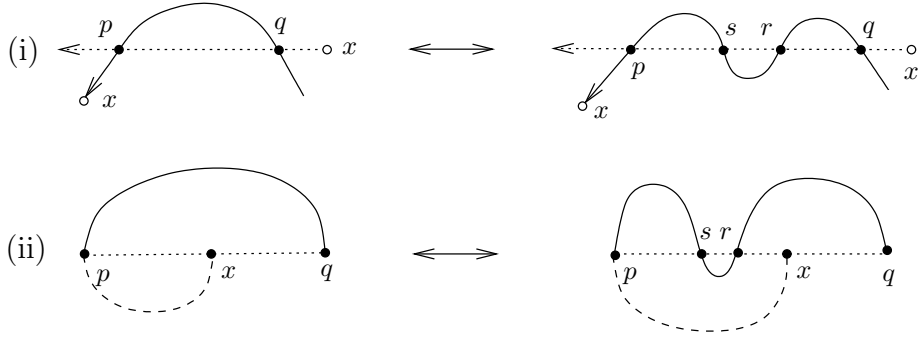

\caption{The primary $(r,s)$-move.}
\label{rs primary}

\end{center}
\end{figure}

Given $u \in \mcM(p,q)$, consider a primary $(r,s)$-move such that for some $m
\in \Z$ the familiy element $(r^m,  s^m)$ affects $u$, i.e. $[r^m,s^m]_i \subset
\ ]p,q[_i$ for $i \in \{u,s\}$. Then there is no $n \in \Z^{\neq m}$ such that
$(r^n, s^n)$ affects $u$: Let the symplectomorphism be $W$-orientation
preserving and w.l.o.g. $m=0$.
Since $x \notin\ [p,q]_u \cap\ [p,q]_s$ at least one of the points $p$, $q$ lies
in the same branch as $r$ and $s$ and w.l.o.g. let it be $p$. If there would be
an $n \neq 0$ with $[r^n, s^n]_i \subset \ ]p,q[_i$ then there is an iterate
$p^k$ with $p^k \in [r, r^n]_u \cap\ [r, r^n]_s \subset \ ]p,q[_u \ \cap\
]p,q[_s$. But then already $p^k\in\ ]p,q[_u \ \cap\ ]p,q[_s$ before the primary
$(r,s)$-move took place. Thus $u$ is no embedding.
For a $W$-orientation reversing symplectomorphism consider its square.

\vsp

We denote by $\lkl\cdot, \cdot \rkl\ : \ \mcHpr \x \mcHpr \to \{0,1\}$
the {\em Kronecker symbol} and extend it to the chain complex by linearity.

For an isotopy $\Phi$ which has a primary tangency at $\tau_0$ and which
displays a primary $(r,s)$-move for $\tau \in [\tau_0 -\ep, \tau_0 +\ep]$ we
abbreviate 
$\mcHpr:=\mcHpr(\Phi_{\tau_0 - \ep}, x_{\tau_0-\ep})$ and identify
$\mcHpr':=\mcHpr(\Phi_{\tau_0 + \ep}, x_{\tau_0+\ep})= \mcHpr \cup \{r^n, s^n
\mid n \in \Z\}$. Moreover set
$({\mathfrak C}_*, \dd):=({\mathfrak C}_*(x_{\tau_0-\ep}, \Phi_{\tau_0-\ep}),
\dd_{x_{\tau_0-\ep}, \Phi_{\tau_0-\ep}})$ and 
$({\mathfrak C}_*', \dd'):= ({\mathfrak C}_*(x_{\tau_0+\ep}, \Phi_{\tau_0+\ep}),
\dd_{x_{\tau_0+\ep}, \Phi_{\tau_0+\ep}})$. Mark signs {\em after} the move by a
prime, i.e. $m'(\cdot, \cdot)$. 
Given a primary $(r,s)$-move, we define the {\em projection}
\beqs
\pi: \mathfrak C_*' \to \mathfrak C_*, \quad
\pi(p)= p - \sum_{n \in \Z} \lkl p,r^n\rkl r^n -\lkl p,s^n\rkl s^n.
\eeqs
The inclusion $\mcHpr \hookrightarrow \mcHpr'$ induces the homomorphism $i:
{\mathfrak C}_* \to {\mathfrak C}_*'$. $\pi$ and $i$ commute with the
$\Z$-action on the chain complexes.
W.l.o.g. assume for the remaining subsection that for a primary $(r,s)$-move
holds $\mu(r,s)=1$ as sketched in Figure \ref{rs primary}.

\begin{Proposition}
\label{fish formula}
For all primary $p$, $q \in \mcHpr$ and all primary $(r,s)$-moves holds
\beqs
m(p,q)=m'(i(p),i(q))- \sum_{n \in \Z}m'(i(p),s^n)m'(r^n,s^n) m'(r^n,i(q)).
\eeqs
\end{Proposition}

\begin{proof}
We know that the primary $(r,s)$-move changes $\mcHpr$ to $\mcHpr'=\mcHpr \cup
\{r^n, s^n \mid n \in \Z\}$. Figure \ref{rs primary} sketches the possible
geometric positions of $p$, $q$, $r$, $s$. Recall that for primary points $p$
and $q$ an embedding $u \in \mcM(p,q)$ is combinatorically affected by the
primary $(r,s)$-move if and only if there is exactly one $n \in \Z $ such that
$]p,q[_u\ \cap \ ]p,q[_s = \{r^n,s^n\}$ after the move.

If the embedding is combinatorically affected by $r^n$ and $s^n$ then it
corresponds under the move to three embeddings between $r^n$ and $q$, $r^n$ and
$s^n$ and $p$ and $s^n$. Using some gluing construction within a small
neighbourhood $U$ containing the move, we obtain $\mcMhat(p,q) \simeq
\mcMhat(r^n,q) \x \mcMhat(r^n,s) \x \mcMhat(p,s^n)$. Counting with orientation,
we find $ m(p,q)=m'(i(p),s^n)=m'(r^n,i(q))=-m'(r^n,s^n)$ and thus
$m(p,q)=-m'(i(p),s^n)m'(r^n,s^n) m'(r^n,i(q))$. For $k \in \Z^{\neq n}$ the
embedding $u \in \mcM(p,q)$ stays unchanged and $m'(i(p),s^k)m'(r^k,s^k)
m'(r^k,i(q))=0$.

If $u$ is not combinatorically affected by the move then either $\mcMhat(p,s^l)=
\emptyset$ or $\mcMhat(r^l,q)=\emptyset$ for all $l \in \Z$. In this case we
have $-m'(i(p),s^l)m'(r^l,s^l) m'(r^l,i(q))=0$ and $m(p,q)=m'(i(p),i(q))$, 
thus altogether $m(p,q)=m'(i(p),i(q))-\sum_{n \in \Z} m'(i(p),s^n)m'(r^n,s^n)
m'(r^n,i(q))$.
\end{proof}

Now we express the boundary operator $\mathfrak d$ in terms of $\dd'$.

\begin{Lemma}
\label{compare d}
\begin{align*}
& \dd p = \pi(\dd'i(p) - \sum_{n \in \Z} m'(i(p),s^n)m'(r^n,s^n) \dd'r^n) \quad
\mbox{for }\mu(i(p),r)=0, \\
& \dd p = \pi(\dd'i(p)) \quad \mbox{otherwise}.
\end{align*}
\end{Lemma}

\begin{proof}
We compute formally
\begin{align*}
& \dd' i(p) = \sum_{\stackrel{\mu(i(p), \qti)=1}{\qti \notin \{r^n, s^n\mid n
\in \Z \}}} m'(i(p), \qti) \qti + \sum_{n \in \Z}m'(i(p), r^n)r^n + \sum_{n \in
\Z}m'(i(p), s^n)s^n , \\
& \dd' r^m = \sum_{\stackrel{\mu(r^m, \qti)=1}{\qti \notin \{s^n \mid n \in \Z
\}}} m'(i(p), \qti) \qti + \sum_{n \in \Z}m'(r^m, s^n)s^n
\end{align*}
and, making use of the Kronecker symbol via $\lkl \del p,q\rkl =m(p,q)$ etc., we
rewrite \reffishformula\ as
\begin{equation}
\label{d kron}
\lkl \dd  p,q\rkl  = \lkl \dd 'i(p)  - \sum_{n \in \Z}  m'(i(p),s^n)m'(r^n,s^n)
\dd 'r^n, i(q)\rkl .
\end{equation}
Applying $\pi$ to $\dd' i(p)$ and $\dd' r^m$ kills all $r^n$- and $s^n$-terms.
We end up exactly with those terms which occur (maybe multiplied by
$m'(i(p),s^n)m'(r^n,s^n)$) in \refdkron. So we obtain $ \dd p = \pi(\dd'i(p) -
\sum_{n \in \Z} m'(i(p),s^n)m'(r^n,s^n) \dd'r^n)$ for $\mu(i(p), r)=0$ and $ \dd
p = \pi(\dd'i(p))$ otherwise.
\end{proof}

Now note the following technical statement:

\begin{Lemma}
\label{m kron}
Consider a primary $(r,s)$-move. Then for $k$, $l \in \Z$ holds $m'(r^k,s^l)=0$
for $k \neq l$.
\end{Lemma}

\begin{proof}
For fixed $m$, the points $r^m$ and $s^m$ are adjacent, but not $r^m$ and
$s^{m-1}$ and $s^m$ and $r^{m+1}$: otherwise $\langle r \rangle$ and $\langle l
\rangle$ would be the only primary equivalence classes of their pair of
intersecting branches implying nonintersecting branches before the move in
contradiction to the assumption on the isotopy. From \refprimarymax\ and
\reffiniteprimary\ we deduce $]r^m, s^n[_u\ \cap\ ]r^m, s^n[_s\ \neq \emptyset$
for $\betrag{m-n}\geq 1$ and thus $\mcM(r^m, s^n)= \emptyset$ and $m'(r^m,
s^n)=0$.
\end{proof}

For the following proofs, keep in mind that $m(p,q)m(p,q) \in \{0,1\}$. 

\begin{Lemma}
\label{def f, g}
We define on the generators
\begin{align*}
& f: ({\mathfrak C}_*', \dd ') \to ({\mathfrak C}_*, \dd ), && f(p):=\pi(p -
\sum_{n \in \Z}m'(r^n,s^n)\lkl p,s^n \rkl  \dd 'r^n), \\
& g: ({\mathfrak C}_*, \dd ) \to ({\mathfrak C}_*', \dd '), && g(p):=i(p)-
\sum_{n \in \Z} m'(r^n,s^n)m'(i(p), s^n)r^n
\end{align*}
and extend them by linearity. Then $f$ and $g$ are $\Z$-equivariant chain maps.
\end{Lemma}

\begin{proof}
For $m \in \Z$ we compute
\beqs
 f(r^m)=0 ,\quad
 f(s^m)= -m'(r^m,s^m) \pi \dd 'r^m, \quad
 f(p)= \pi(p) \quad \mbox{for } p \neq r^m,\ s^m.
\eeqs
Recall $\mu(r^m,s^m)=1$ and $\mcMhat(r^m,s^m) \neq \emptyset$ such that
$m'(r^m,s^m) = \pm 1$ and keep the equations
\begin{align*}
& \dd  p = \pi(\dd 'i(p) - \sum_{n \in \Z} m'(p,s^n)m'(r^n,s^n) \dd 'r^n) \quad
\mbox{for }\mu(i(p),r)=0, \\
& \dd  p = \pi(\dd 'i(p)) \quad \mbox{otherwise}
\end{align*}
from \refcompared\ in mind. For $f$, we obtain
\begin{align*}
f(\dd 'r^m) & =f(i \pi\dd 'r^m +\sum_{n \in \Z} m'(r^m,s^n)s^n)\stackrel{\ref{m
kron}}{=}
f(i \pi\dd 'r^m +m'(r^m,s^m)s^m) \\
&= \pi i \pi \dd'r^m - 0 + 0 - m'(r^m, s^m)m'(r^m, s^m) \pi(\dd'r^m) \\
& =  \pi \dd 'r ^m -\pi \dd 'r^m =0= \dd  0 \\
& = \dd  f(r^m), \\
\dd  f(s^m) & = \dd (-m'(r^m,s^m) \pi \dd 'r^m)\stackrel{\ref{compare d}}{=}\pi
\dd '(-m'(r^m,s^m) i \pi \dd 'r^m) \\
& = - m'(r^m,s^m) \pi \dd ' i \pi \dd ' r^m \\
& = - m'(r^m,s^m) \pi \dd '(\dd 'r^m - \sum_{n \in \Z} m'(r^m,s^n)s^n) \\
& \stackrel{\ref{m kron}}{=} -m'(r^m,s^m)(\pi \dd ' \dd ' r^m -m'(r^m,s^m) \pi
\dd 's^m) \\
& = \pi \dd ' s^m \\
& = f(\dd ' s^m).
\end{align*}
For $p \neq r^m$, $ s^m$ and $m \in \Z$, we obtain
\begin{align*}
f(\dd 'p) & =f(i \pi \dd 'p + \sum_{n \in \Z} m'(p,r^n)r^n + m'(p,s^n)s^n) \\
& =   \pi i \pi \dd' p - \pi \left( \sum_{l \in \Z}m'(r^l, s^l) \lkl i \pi \dd'
p, s^l \rkl \dd'r^l \right) \\
& \quad + \pi \left( \sum_{n \in \Z}m'(p,r^n)r^n\right) \\
& \quad - \pi \left( \sum_{l \in \Z}m'(r^l, s^l) \lkl \sum_{n \in \Z}m'(p,r^n)r^n, s^l
\rkl \dd' r^l \right) \\
& \quad + \pi \left( \sum_{n \in \Z}m'(p,s^n)s^n\right) \\
& \quad - \pi \left( \sum_{l \in \Z}m'(r^l, s^l) \lkl \sum_{n \in \Z}m'(p,s^n)s^n, s^l
\rkl \dd'r^l \right) \\
& = \pi i \pi \dd 'p -0 + 0 -0 +0 - \sum_{l \in \Z} m'(p,s^l) m'(r^l,s^l) \pi
\dd 'r^l\\
& = \pi \dd 'p - \sum_{l \in \Z} m'(p,s^l) m'(r^l,s^l) \pi \dd 'r^l \\ 
& = \pi \dd 'i \pi p - \sum_{l \in \Z}m'(p,s^l) m'(r^l,s^l) \pi \dd 'r^l \\
& \stackrel{\ref{compare d}}{=} \dd  \pi p \\
& = \dd  f(p).
\end{align*}
Now we extend the definition of $m'(p,q)$ etc. by linearity from primary points
to elements of ${\mathfrak C}_*'$, i.e. $m'(\sum_j p_j,q):= \sum_j m'(p_j,q)$,
and consider $g$:

{\em Case $\mu(i(p),  r)=0$}: 
We first show 
\begin{equation}
\label{no sr}
i\pi \dd ' g(p)=\dd ' g(p)
\end{equation}
 which follows from $\lkl \dd' g(p), r^m\rkl =0$ due to $\mu(i(p),  r^m)=0$ for
$m \in \Z$ and 
\begin{align*}
\lkl \dd 'g(p), s^m \rkl  & = m'(g(p), s^m)= m'(i(p) - \sum_{n \in \Z}
m'(r^n,s^n) m'(i(p),s^n)r^n, s^m) \\
& = m'(i(p), s^m)- \sum_{n \in \Z} m'(r^n,s^n) m'(i(p), s^n)m'(r^n,s^m) \\
& \stackrel{\ref{m kron}}{=}  m'(i(p), s^m)- m'(r^m, s^m) m'(i(p), s^m)m'(r^m,
s^m) \\
&= m'(i(p), s^m)-m'(i(p), s^m) \\
& =0.
\end{align*}
Now we obtain
\begin{align*}
g(\dd  p)& \stackrel{\mu(i(p), r)=0}{=} i(\dd  p)
\stackrel{\ref{compare d}}{=} i \pi \dd '(i(p) - \sum_{n \in \Z}
m'(i(p),s^n)m'(r^n,s^n)r^n) \\
& = i \pi \dd ' g(p)\stackrel{\ref{no sr}}{=} \dd ' g(p).
\end{align*}
{\em Case $\mu(i(p), r) \neq 0$}:
First note
\begin{equation}
\label{m'(del)=0}
m'(\dd'i(p), s^m)= \lkl \dd' (\dd'i(p)), s^m \rkl =0
\end{equation}
and then compute
\begin{align*}
g(\dd  p) & \stackrel{\ref{compare d}}{=} g(\pi \dd 'i(p)) \\
& = i \pi \dd 'i(p) - \sum_{n \in \Z} m'(r^n,s^n) m'(i \pi \dd 'i(p), s^n)r^n \\
& = i \pi \dd ' i(p) - \sum_{n \in \Z} m'(r^n,s^n)m'(\dd ' i(p) \\ 
& \quad - \sum_{l \in
\Z} m'(i(p), r^l)r^l - m'(i(p), s^l)s^l, s^n) r^n \\
&  = i \pi \dd ' i(p) -\sum_{n \in \Z}  m'(r^n,s^n) (m'(\dd 'i(p), s^n) \\ 
& \quad - \sum_{l
\in \Z}m'(i(p), r^l)m'(r^l,s^n)- 0)r^n \\
& \stackrel{\refmdel}{=} i \pi \dd ' i(p) +\sum_{n \in \Z} m'(r^n,s^n) m'(i(p),
r^n)m'(r^n,s^n)r^n \\
& = i \pi \dd ' i(p) + \sum_{n \in \Z} m'(i(p), r^n)r^n  \\
& \stackrel{\mu(i(p), r)\neq0}{=} \dd 'i(p) \\
&   \stackrel{\mu(i(p), r)\neq0}{=}  \dd '(g(p)).
\end{align*}
Since $\pi$ and $i$ are $\Z$-equivariant so are $f$ and $g$.
\end{proof}

Now we show that $f$ and $g$ induce isomorphisms between the homologies of
$({\mathfrak C}_*', \dd ')$ and $({\mathfrak C}_*, \dd )$.

\begin{Theorem}
\label{rs primary isom}
The homologies of $(\mathfrak C_*', \dd')$ and $(\mathfrak C_*, \dd)$ are
isomorphic.
\end{Theorem}

\begin{proof}
For $f$ and $g$ from \refdeffg, we show that $f_*: H({\mathfrak C}_*', \dd ')
\to H({\mathfrak C}_*, \dd )$ and $g_* : H({\mathfrak C}_*, \dd ) \to
H({\mathfrak C}_*', \dd ')$ are inverse to each other. It is enough to show $f
\circ g \simeq \Id_{{\mathfrak C}_*}$ and $g \circ f \simeq \Id_{{\mathfrak
C}_*'}$ where $\simeq$ stands for homotopic by a chain homotopy.
$f \circ g: ({\mathfrak C}_*, \dd ) \to ({\mathfrak C}_*, \dd )$ is even the
identity:
\begin{align*}
f(g(p)) & = f(i(p))- \sum_{n \in \Z} m'(r^n,s^n)m'(i(p),s^n) f(r^n)= f(i(p)) \\
& = \pi i(p) - \pi \left( \sum_{n \in \Z} m'(r^n,s^n)\lkl i(p),s^n\rkl  \dd
'r^n\right) = \pi i(p) = \Id_{{\mathfrak C}_*}(p).
\end{align*}
Unfortunately, this is not true for $g \circ f$. But we can find a chain
homotopy $h: ({\mathfrak C}_*', \dd ') \to ({\mathfrak C}_{*+1}', \dd ')$
satisfying $g \circ f -\Id_{{\mathfrak C}_*'}= h \circ \dd ' + \dd ' \circ h$.
Choose 
\beqs
h(p):= - \sum_{n \in \Z} \lkl s^n,p\rkl m'(r^n,s^n)r^n
\eeqs
and compute for $m \in \Z$
\begin{align*}
& (h \circ \dd ' + \dd ' \circ h)(r^m) \\
& \qquad = - \sum_{n \in \Z} \lkl s^n, \dd 'r^m
\rkl  m'(r^n,s^n)r^n  
- \dd '\left( \sum _{n \in \Z}\lkl s^n,r^m \rkl m'(r^n,s^n)r^n\right) \\
& \qquad = - \sum_{n \in \Z} m'(r^m ,s^n )m'(r^n,s^n)r^n \\
& \qquad \stackrel{\ref{m kron}}{=} -m'(r^m, s^m) m'(r^m, s^m)r^m \\
 & \qquad =-r^m ,
\end{align*}
and 
\begin{align*}
& (h \circ \dd ' + \dd ' \circ h)(s^m) \\
& \qquad  = -\sum_{n \in \Z} \lkl s^n,\dd 's^m \rkl
m'(r^n,s^n)r^n - \sum_{n \in \Z} \lkl s^n,s^m\rkl m'(r^n,s^n) \dd 'r^n \\
& \qquad = -m'(r^m,s^m)\dd 'r^m
\end{align*}
and for $p \neq r^m$, $s^m$ for $m \in \Z$
\begin{align*}
& (h \circ \dd ' + \dd ' \circ h)(p) \\
& \qquad  = -\sum_{n \in \Z}\lkl s^n, \dd 'p\rkl 
m'(r^n,s^n) r^n   - \dd'\left( \sum_{n \in \Z} \lkl s^n,p\rkl  m'(r^n,s^n) r^n
\right) \\
& \qquad  = - \sum_{n \in \Z} m'(p,s^n) m'(r^n,s^n) r^n.
\end{align*}
On the other hand, we obtain
\begin{equation*}
(g \circ f - \Id_{{\mathfrak C}_*'})(r^m) = g(f(r^m))-r^m=-r^m 
\end{equation*}
and
\begin{align*}
& (g \circ f - \Id_{{\mathfrak C}_*'})(s^m) \\
& \qquad = g(\pi(s^m)-\pi \left( \sum_{n \in
\Z} m'(r^n,s^n)\lkl s^m,s^n\rkl \dd 'r^n \right) )-s^m \\
& \qquad = g(-m'(r^m, s^m) \pi \dd' r^m) - s^m \\
& \qquad = -m'(r^m,s^m)g(\pi\dd 'r^m)-s^m \\
&\qquad = -m'(r^m,s^m) \left( i \pi \dd 'r^m - \sum_{n \in \Z}m'(r^n,s^n) m'(i \pi \dd 'r^m,
s^n)r^n\right) -s^m \\
&\qquad = -m'(r^m,s^m)  i \pi \dd 'r^m -s^m \\
& \qquad = -m'(r^m,s^m)(\dd 'r^m - \sum_{n \in \Z} m'(r^m,s^n)s^n) -s^m \\
&\qquad \stackrel{\ref{m kron}}{=} -m'(r^m,s^m) \dd 'r^m + m'(r^m, s^m)m'(r^m, s^m)s^m
-s^m \\
& \qquad = -m'(r^m,s^m) \dd 'r^m
\end{align*}
and for $p \neq r^m$, $s^m$ for $m \in \Z$
\begin{align*}
(g \circ f - \Id_{{\mathfrak C}_*'})(p) 
& = g\left(\pi(p) - \sum_{n \in \Z}
m'(r^n,s^n) \lkl p,s^n\rkl  \pi\dd 'r^n\right)- p \\
& = g(\pi(p))-p  \\
& = i\pi (p) -\sum_{n \in \Z} m'(r^n,s^n) m'(i \pi(p), s^n)r^n - p
\\
& =  - \sum_{n \in \Z} m'(r^n,s^n) m'(p, s^n)r^n.
\end{align*}
Comparing the results yields $g \circ f - \Id_{{\mathfrak C}_*'} = h \circ \dd '
+ \dd ' \circ h$ which proves the claim.
\end{proof}

Moreover, note that also the chain homotopy $h$ commutes with the $\Z$-action on
the chain complexes.
Now we divide by the $\Z$-action. Define $C_*$ and $C_*'$ analogously to
$\mathfrak C_*$ and $\mathfrak C_*'$. Since $f$, $g$ and $h$ commute with the
$\Z$-action on the chain complexes they pass to $C_*$ and $C_*'$ and we obtain

\begin{Theorem}
\label{prim move invar}
The homologies of $(C_*, \del)$ and $(C'_*, \del')$ are isomorphic, i.e. primary
moves leave the primary Floer homology invariant.
\end{Theorem}

\subsection{Invariance under mixed moves}

The invariance under mixed moves will be reduced to the invariance under primary
and secondary moves.
If not stated otherwise, we will work with the lifted tangles on the universal
cover.

\vsp

Now we want to investigate how mixed moves look like. If a $(r,s)$-move flips a
primary points $p$ secondary the segments $]p,x[_u$ and $]p,x[_s$ have to
intersect after the move. In particular, $r$, $s$ and $p$ have to lie in the
same pair of branches.
Since the (un)stable \mf s are free of self-intersections a mixed move always
takes place within a fixed frame, i.e. the mixed move cannot `overlap' into
another iterate of the frame.
With \refprimarymax\ in mind, mixed moves look as sketched in Figure \ref{mixed}
(where $p_{2n+1}$ still lies in the frame induced by $p$). Mixed moves come
along with $2n +1$ flips and are called {\em simple} if $n=0$.

\begin{figure}[h]
\begin{center}

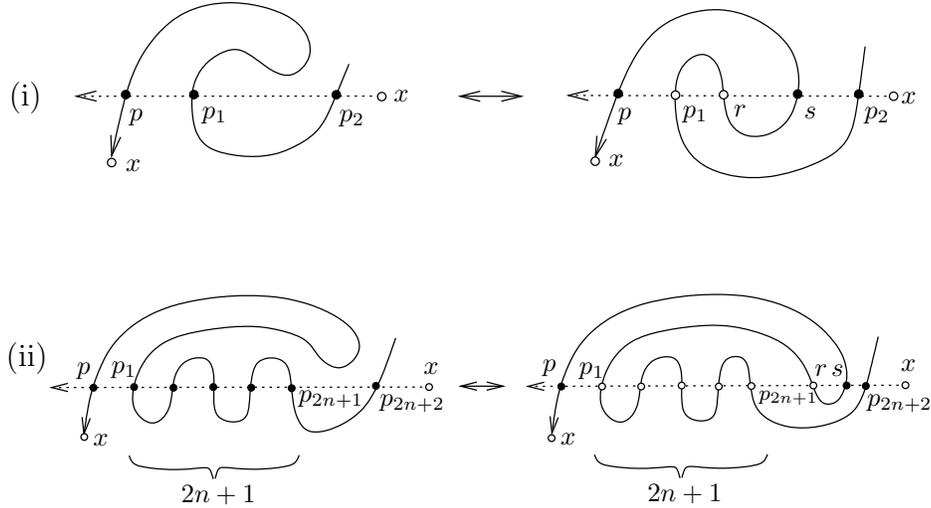

\caption{Mixed $(r,s)$-moves with one flip in (i) and 2n+1 flips in (ii).}
\label{mixed}

\end{center}
\end{figure}

\vsp

%\label{no inv}
%\newcommand{\refnoinv}{Proposition \ref{no inv}}

Before we consider the invariance under mixed moves we note the following.
Without the condition `$\dots \cap \mcH_{[x]}$' in the definition of `primary',
primary Floer homology would not be invariant: Consider Figure \ref{bad flip}
where a move circles around genus. Assume for sake of simplicity that only the
branches containing $p$ intersect.
In our convention, the move is secondary and thus leaves the homology invariant.
Dropping `$\dots \cap \mcH_{[x]}$' is equivalent to using the {\em contractible}
semi-primary points as generators of the chain complex. 
Before the move, $p$ and $q$ are contractible and semi-primary, but after the
move $q$ is no longer semi-primary. The generated $r$ is secondary and $s$
semi-primary, but not contractible. Thus it is excluded as generator.
Before the move we obtain $H_{-1}=\Z \langle p \rangle $ and $H_{-2}=\Z \langle
q \rangle$ and $H_*=0$ for $n \neq -1, -2$. But after the move there is only $p$
left as generator. Thus $H_{-1}=\Z \langle p \rangle$ and $H_*=0$ otherwise.
This phenomenon inspires the definition of semi-primary Floer homology in
\refsemiprimhom.

\begin{figure}[h]
\begin{center}

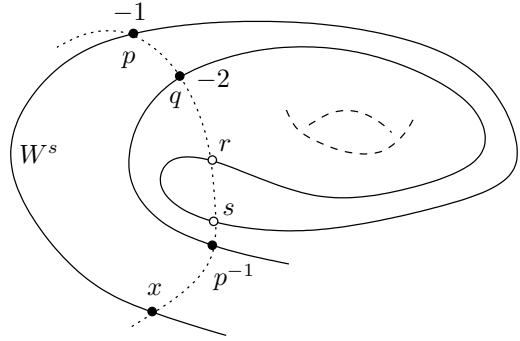

\caption{Arising of nontrivial homotopy classes.}
\label{bad flip}

\end{center}
\end{figure}

\begin{Proposition}
\label{simple mixed invar}
Primary Floer homology stays invariant under simple mixed moves.
\end{Proposition}

\begin{proof}
Consider Figure \ref{mixed} (i): the simple mixed $(r,s)$-move can be recognized
as an identification followed by the secondary $(p_1,r)$-move. Since both leave
the homology invariant so does the simple mixed move.
An explicit chain complex isomorphism is given by $f: (C_*, \del)\ 
\longrightarrow\  (C_*', \del')$, $a \mapsto a$ for $ a \in
\mcHprti\backslash\{\langle p_1 \rangle\}$ and $ \langle p_1 \rangle \mapsto
\langle s \rangle$.
\end{proof}

Now we consider the invariance under arbitrary mixed moves.

\begin{Theorem}
\label{mixed invar}
Primary Floer homology is invariant under mixed moves.
\end{Theorem}

\begin{proof}
For simple mixed moves, the claim was already proven in \refsimplemixedinvar.
Now consider Figure \ref{mixed} (ii). The mixed move can be recognized as a
sequence of primary moves $(p_2, p_3)$, \dots, $(p_{2n}, p_{2n+1})$, followed by
a simple mixed $(r,s)$-move and a sequence of secondary moves $(p_2, p_3)$,
\dots, $(p_{2n}, p_{2n+1})$. This yields the claim.
\end{proof}

\subsection{The proof of \refstartend}

\label{proof start end}

Let $\phi\in \Diff_\om(M)$ with \hyp\ $x \in \Fix(\phi)$. Let $(x,\phi)$ be csi
and let all primary points be transverse. 
First we generalize \refflipviamixed.

\begin{Lemma}
\label{pert sec unimport}
\label{prim stab}
\begin{enumerate}[(1)]
\item
\label{first}
Let $\phihat \in \Diff_\om(M)$ be a small perturbation of $\phi$ and $\xhat$ the
continuation of $x$.
Let $p_\phi\in \mcHpr(\phi)$ be primary and let $p_\phi$ persist as transverse
homoclinic points $p_\phihat$, but nonprimary. Then there is $q \in
\mcHpr(\phihat)$ which is no continuation of any primary point of $\phi$.
\item
\label{second}
Let $(x,\phi)$ be csi and let all primary points be transverse. Then for
sufficiently small perturbations $\phihat \in \Diff_\om(M)$ of $\phi$, all
primary points remain transverse and no primary points arise or vanish.
\end{enumerate}
\end{Lemma}

\begin{proof}
{\em First item:}
We work with the lifted tangles of $\phi$ and $\phihat$, but we drop the tilde
for sake of readability.
The segments $[x,p_\phi]_i$ and $[\xhat, p_\phihat]_i$ are close. Since $p_\phi$
is primary $]x,p_\phi[_u \ \cap\ ]x, p_\phi[_s\ =\emptyset$. But $p_\phihat$ is
nonprimary, thus $]\xhat,p_\phihat[_u \ \cap\ ]\xhat, p_\phihat[_s\ \neq
\emptyset$. $\xhat$ and $p_\phihat$ remain transverse. Figure
\ref{prim_second_pert} (ii) -- (iv) lists the three types which prevent
$p_\phihat$ to be primary. In all three cases, there is a primary $q \in \
]\xhat, p_\phihat[_u\ \cap\ ]\xhat,p_\phihat[_s$ which has no corresponding
point in $]x, p_\phi[_u \ \cap \ ]x, p_\phi[_s$ and thus in $\mcHpr(\phi)$.

{\em Second item:}
Since all primary points of $\phi$ are transverse they persist at least as
transverse intersection points for small perturbations. Any primary-secondary
flip would require the rise of a new primary point.
But primary points only arise in frames and the compactness of the frame
prevents this for sufficiently small perturbations.
\end{proof}

\begin{figure}[h]
\begin{center}

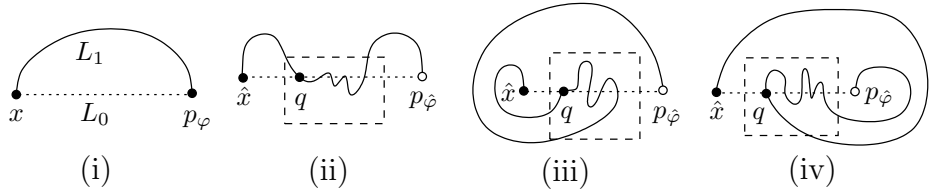

\caption{Causes for primary-secondary flips.}
\label{prim_second_pert}

\end{center}
\end{figure}

Thus also in this generalized situation, a primary-secondary flip is coupled
with the rise of a new primary point.
Now we generalize \refsecunimport.

\begin{Lemma}
\label{sign equal}
Let $\phi \in \Diff_\om(M)$ be csi with $x \in \Fix(\phi)$ \hyp\ and all primary
points transverse. Let $\phihat \in \Diff_\om(M)$ be small perturbation of
$\phi$ such that all primary points persist transverse. Consider primary
$p_\phi$ and $q_\phi$ with $\mu(p_\phi, q_\phi)=1$ and denote their continuation
by $p_\phihat$ and $q_\phihat$. Then $m(p_\phi, q_\phi)=m(p_\phihat,
q_\phihat)$.
\end{Lemma}

\begin{proof}
For simplicity, abbreviate $p:=p_\phihat$ and $q:=q_\phihat $
Clearly $\mu(p_\phi, q_\phi)=\mu(p, q)$ and if $m(p_\phi,q_\phi) \neq 0 \neq
m(p,q)$ then their signs coincide. Thus it is enough to show $\mcM(p_\phi,
q_\phi) \neq \emptyset$ if and only if $\mcM(p,q) \neq \emptyset$. We will work
on the universal cover with the lifted tangles. 

We have to check if the proof of \refsecunimport\ carries over to our more
general situation.
Let $\mcM(p_\phi, q_\phi) \neq \emptyset$ and assume $\mcM(p,q)=\emptyset$, i.e.
$]p,q[_u\ \cap\ ]p,q[_s\ \neq \emptyset$.

For that, we have to admit perturbations as in the dashed boxes in Figure
\ref{prim_second_pert} and check a generalization of Figure \ref{rs perturb
fish} etc. which yields the claim.
\end{proof}

\begin{proof}[Proof of \refstartend]
For sufficiently small perturbations, the primary points persist by
\refprimstab. Thus the generator set of primary homoclinic chain complex stays
unchanged. Moreover, the boundary operator persists due to \refsignequal. Thus
the homology remains unchanged.
\end{proof}

\section{Dynamics and homoclinic Floer theory}

\subsection{Conjugacy} 

Since conjugacy does not affect the intersection behaviour of homoclinic tangles
one expects primary Floer homology to be invariant under conjugacy.
Nevertheless, one has to be a little careful. If one is only interested in the
topological information, then it is enough to have a homeomorphisms $h$
conjugating $\phi$, $\psi \in \Diff_\om(M)$, i.e. $\phi \circ h = h \circ \psi$:
If $x \in \Fix(\psi)$ then $h(x) \in \Fix(\phi)$ and $H_*(x, \psi)=H_*(h(x),
\phi)$ since $h(W^i(x,\psi))=W^i(h(x),\phi)$ for $i \in \{s,u\}$. 

But if symplectic properties should be preserved (as for example the symplectic
volume $\int v^* \om$ of an immersion $v$) we have to require $h$ to be
symplectic.

\subsection{{\boldmath $\rk H_*(x, \phi) \leq \rk H_*(x, \phi^n)$}}

Now denote by $\langle p_1 \rangle , \dots, \langle p_k\rangle $ the generators
of $C_*(x,\phi)$ and set $p_i^j:=\phi^j(p_i)$.  
For $n \in \N_0$, we have $W_i^{\phi}=W_i^{\phi^n}$ for $i \in \{s,u\}$ and
$W_i^{\phi}=W_j^{\phi^{-n}}$ for $i \neq j \in \{s,u\}$. Note that the number of
equivalence classes multiplies: $C_*(x, \phi^n)$ is generated by $\langle p_1^0
\rangle, \dots, \langle p_k^0 \rangle, \langle p_1^1 \rangle, \dots, \langle
p_k^{n-1} \rangle$.
Abbreviate $\Z_n:=\Z \slash n \Z = \{\bar{0}, \bar{1}, \dots, \overline{n-1}\}$
and set $\phi_*^l=\bar{l}$. There is a $\Z_n$-action on the generators via 
\beqs
\Z_n \x C_*(x, \phi^n) \to C_*(x, \phi^n), \qquad \phi_*^l.\langle p_i^j \rangle
:= \langle p_i^{j+l \ mod\ n} \rangle = \langle \phi^l(p_i^j) \rangle
\eeqs
and extend it by linearity to the complex. We notice 
$\phi^l_*.(\del \langle p_i^j \rangle) = \del \langle \phi^l(p_i^j) \rangle$
such that the $\Z_n$-action descends to homology. 
If we use the $m(p,q)$-signs we assume $\Q$-coefficients and so we do for
$n(p,q)$-signs if $\phi$ is $W$-orientation preserving $\phi$ and $ n \in \N_0$.
In the orientation reversing case, assume $\Z_2$ as coefficient ring for the
$n(p,q)$-signs if $n=2m+1 \in \N$ odd. Then \refhomcohom\ allows us to treat
simultanously also negative exponents and we define
\begin{align*}
f: C_*(x,\phi^n;\K)\simeq C^{-*}(x, \phi^{-n}, \K) \to C_*(x, \phi;\K), &&&
f(\langle p_i^j \rangle):= \langle p_i \rangle, \\
g: C_*(x, \phi; \K) \to C_*(x, \phi^n;\K)\simeq C^{-*}(x, \phi^{-n};\K) , &&&
g(\langle p_i \rangle):= \frac{1}{n}\sum_{j=0}^{n-1} \langle p_i^j \rangle
\end{align*}
where $\K$ stands for the suitable coefficient ring. $f$ and $g$ are chain maps
and we compute
$f \circ g = \Id _{C_*(x, \phi;\K)}$.
Denote by $g_*$ and $f_*$ the induced maps on the (co)homology.

\begin{Proposition}
\label{dimleq}
$g_*$ is injective, $f_*$ surjective and
\beqs
\rk H_*(x, \phi; \K) \leq \rk H_*(x, \phi^n; \K)= \rk H^{-*}(x, \phi^{-n}; \K).
\eeqs
The difference is measured by the long exact sequence
\beqs
\cdots \to H_l (\ker f;\K) \to H_l(x, \phi^n;\K) \to H_l(x, \phi; \K) \to
H_{l-1}(\ker f, \K)
\to \cdots
\eeqs
\end{Proposition}

\begin{proof}
We drop the coefficient ring $\K$ in the notation in favour of better
readability. $f \circ g = \Id _{C_*(x, \phi)}$ implies the injectivity of $g_*$
and surjectivity of $f_*$ which yield the dimension estimates. The range of $g$
are the invariants under the $\Z_n$-action and the kernel of $f$ the
coinvariants which are both subcomplexes of $C_*(x, \phi^n)$. We obtain the
short exact sequence of chain complexes
\begin{equation}
\label{kes}
((\ker f)_*, \del) \hookrightarrow (C_*(x, \phi^n), \del) \twoheadrightarrow
\left( \frac{C_*(x, \phi^n)}{(\ker f)_*}, \delbar \right)
\end{equation}
where $\delbar $ is induced by the projection. Moreover 
\beqs
h:  \left(\frac{C_*(x, \phi^n)}{(\ker f)_*}, \delbar\right) \to (\Img(g)_*,
\del), \qquad [c] \mapsto \sum_{l=0}^{n-1}\phi_*^l(c)
\eeqs
is an isomorphism and satisfies $h \circ \delbar = \del \circ h$,  thus an
isomorphism of chain complexes. Since also $g: C_*(x, \phi) \to \Img(g)_*$ is an
isomorphism of chain complexes we obtain by means of the long exact sequence of
\refkes\
\beqs
\cdots \to H_l (\ker f) \to H_l(x, \phi^n) \to H_l(x, \phi) \to H_{l-1}(\ker
f)\to \cdots
\eeqs
\end{proof}

Now let us discuss under which circumstances we might have equality in
\refdimleq. We call a smooth Hamiltonian function $H: \R \x M \to \R$ (with
compact support) {\em normalized} if $\int_M H_t \ d\Vol=0 $ for all $t$ where
$H_t:=H(t, \cdot)$. Let $X$ be its nonautonomous vector field. Denote by
$\phi_{(t, t_0)}$ the nonautonomous flow of $\zdot(t)=X(t, z(t))$ starting at
time $t_0$, i.e. $\phi:=\phi_{(1,0)}$ is the usual time-1 map. If we assume in
addition $H(t, \cdot)=H(t+1, \cdot)$ then $\phi_{(n, 0)}=\phi^n_{(1,0)}=\phi^n$
and $\phi$ and $\phi^n$ are joint by the isotopy $\tau \mapsto \phi_{(1+
(n-1)\tau, 0)}$.
Changing the parametrization of a Hamiltonian path is easy: Given $\tau \mapsto
\psi_{(\tau, 0)}$ with Hamiltonian $F(t,z)$, we obtain $\tau \mapsto
\psi_{(b(\tau), 0)}$ using $b'(t)F(b(t), z)$ as Hamiltonian.

Conversely, Banyaga \cite{banyaga} proved that for every path of Hamiltonian
diffeomorphisms $\tau \mapsto \psi_\tau$, there is a normalized Hamiltonian
having $\psi_\tau=\psi_{(\tau, 0)}$ as nonautonomous flow. Therefore we
conclude

\begin{Corollary}
Equality for Hamiltonian diffeomorphisms in \refdimleq\ is tied to the question
of invariance of primary Floer homology under Hamiltonian isotopies, i.e. the
question if Hamiltonian isotopies fulfil the requirements of \refinvthcpt. In
\refcomment, we conjecture the answer to be affirmative.
\end{Corollary}

If there is (conjecturally) equality for Hamiltonian diffeomorphisms in
\refdimleq\ we need to know how large the groups of Hamiltonian diffeomorphisms
$\Ham(M, \om) \subset \Diff_\om(M)$ actually is. Assume $M$ to be closed and
denote by $\Diff_\om^0(M)$ the component of the identity in $\Diff_\om(M)$. The
difference between $\Diff_\om^0(M)$ and $\Ham(M, \om)$ is measured via
\beqs
\Diff_\om^0(M) \slash \Ham(M, \om) = H^1(M, \R) \slash \Ga
\eeqs
where $\Ga \subset H^1(M, \R) $ is the so-called {\em flux group} (cf.\
Polterovich \cite{polterovich2}). Thus for manifolds with vanishing first
cohomology class, we have $ \Diff_\om^0(M) = \Ham(M, \om)$.
Examples with $H(x, \phi) < H(x, \phi^n)$ might arise for non-Hamiltonian
symplectomorphisms, especially symplectomorphisms not isotopic to the identity.

\vsp

Another interesting aspect is the relationship between $H(x, \phi^n)$ and
$\mathfrak H_m(x, \phi)$. One might ask if actually $H(x, \phi^n)$ might somehow
converge to $\mathfrak H_m(x, \phi)$. This turns out to be not true at least for
Hamiltonian diffeomorphisms.

\begin{Proposition}
\label{infinitehom}
There is $\phi \in \Ham(M, \om)$ with $H(x, \phi^n) \neq \mathfrak H(x, \phi)\
\forall\ n$.
\end{Proposition}

\begin{proof}
Consider the homoclinic tangle in Figure \ref{quadr cubic} (a). There are
exactly two distinct equivalence classes of primary points. Let us denote them
by $\langle p \rangle$ and $\langle q \rangle$ and assume w.l.o.g. $\mu(\langle
p \rangle)=-1$ and $\mu(\langle q \rangle)=-2$. Then $\del \langle p \rangle =
\langle q \rangle - \langle q \rangle =0$ and $\del \langle q \rangle =0$ and we
obtain $H_{-1}(x, \phi) \simeq \Z \langle p \rangle$ and $H_{-2}(x, \phi) \simeq
\Z \langle q \rangle$. Moreover, we calculate explicitly $H_{-1}(x, \phi^n)
\simeq \Z \simeq H_{-2}(x, \phi^n)$ for $n \in \N$.

On the other hand, we compute $\mathfrak C_{-1}= \Span_\Z\{p^l \mid l \in \Z\}$
and $\mathfrak C_{-2}= \Span_\Z \{q^l \mid l \in \Z \}$ and $\mathfrak d p^l=q^l
- q^{l-1}$ and $ \mathfrak d q^l=0$ for all $l \in \Z$.
Thus $\mathfrak H_{-1}=0$ and $\mathfrak H_{-2}= {\Span_\Z \{q^n \mid n \in
\Z\}}\slash {\Span_\Z\{q^n + q^{n-1}\}} \simeq \Z$ and therefore $\mathfrak
H_{-1} \neq H_{-1}$.
\end{proof}

\subsection{{\boldmath $\rk \Hti _*(x, \phi) < \rk \Hti_*(x, \phi^n)$}}

In this section, we define a version of homoclinic Floer homology based on
contractible semi-primary points, called {\em semi-primary Floer homology}. The
construction is analogous to primary Floer homology except for the invariance
property in \refinvthcpt. The weaker invariance property of semi-primary Floer
homology allows a better sensitivity for the underlying symplectomorphism. 
For example, certain interactions of the tangle and the topology of the manifold
are noticed to which primary Floer homology is oblivious. Moreover, semi-primary
Floer homology distinguishes between $\phi$ and $\phi^n$ for certain
symplectomorphisms $\phi$.

\vsp

Denote by $\mcH_s:=\{p \in \mcH_{[x]} \mid \ ]x,p[_u \ \cap\ ]x, p[_s \ =
\emptyset\}$ the set of {\em contractible} semi-primary points and by
$\mcHti_s:=\mcH_s \slash_\sim$ the set of contractible semi-primary equivalence classes where
$p \sim q$ if and only if $p=q^n$ for some $n \in \Z$. 
As before, the equivalence classes are denoted by $\langle p \rangle$. On
$\R^2$, the notion of primary and semi-primary coincide since $\R^2$ is
contractible. Thus assume from now on that $(M, \om)$ is a surface with genus $g
\geq 1$.

\vsp

We define the semi-primary Floer chain complex via 
\beqs
\Cti_k:=\Cti(x, \phi):=\bigoplus_{\stackrel{\langle p \rangle \in
\mcHti_s}{\mu(\langle p \rangle)=k}} \Z \langle p \rangle, \qquad \delti \langle
p \rangle := \sum_{\stackrel{\langle q \rangle \in \mcHti_s}{\mu(\langle q
\rangle ) = \mu(\langle p \rangle) -1}} m(\langle p \rangle , \langle q \rangle)
\langle q \rangle
\eeqs
and extend the boundary operator linearly to $\delti : \Cti_* \to \Cti_{*-1}$. 

\begin{Theorem}
\label{semiprimhom}
It holds $\delti \circ \delti=0$ and $\Hti_*(x, \phi):=\ker \delti_* \slash \Img
\delti_{*+1}$ is called {\em semi-primary Floer homology}.
\end{Theorem}

\begin{proof}
Since the set of contractible semiprimary points is a subset of the set of primary points most of the proofs for primary Floer homology carry over. But we have to check the cutting procedure: since we are restricting the boundary operator to the subset $\mcH_s \subset \mcHpr$ of the primary points we have to make sure that the cutting points are also in $\mcH_s$.

\vsp

More precisely, we have to proof an analogon of \refpropclassifprimary\ and \refprimcutth\ for contractible semiprimary points. This can be deduced from the already existing primary classification \refpropclassifprimary\ and Figure \ref{primary cutting} as follows: 

Let $p$, $r \in \mcH_s$ with $\mu(p,r)=2$. Since $p$ and $r$ are contractible and semiprimary, they satisfy $]p,x[_s \ \cap \ ]p,x[_u\ = \emptyset\ = \ ]r,x[_s \ \cap \ ]r,x[_u $ and both `loops' $[p,x]_s \ \cup \ [p,x]_u$ and $[r,x]_s \ \cup \ [r,x]_u $ span a 2-gon. This yields us the same pictures and positioning for $p$ and $r$ as in Figure \ref{primary cutting}, but this time on the manifold and not on the universal cover. Then we notice that we have 2-gons with vertices $p$ and $r$ of relative Maslov index $\pm 2$ in the middle column of Figure \ref{primary cutting} and 2-gons with vertices $x$ and $r$ resp.\ $p$ of relative index $\pm 3$ in the right and left column of Figure \ref{primary cutting}. We observe that, in the left and right column, the cutting procedure takes place {\em within} this 2-gon of relative Maslov index $\pm 3$. Thus it does not matter if we are on the universal cover or on the manifold itself --- the cutting works in the very same way as for primary points and provides us with two contractible semiprimary cutting points $q_s$ and $q_u$.

Now let us consider the middle column. Remember, the 2-gon with vertices $r$ and $p$ lies on the manifold. Consider the small `overlapping nose'. $q_s$ and $q_u$ are not primary whether or not the `overlapping nose' wraps around some genus like in Figure \ref{bad flip}. And neither are they contractible semiprimary.

\vsp

Thus we reproved the classification in Figure \ref{primary cutting} for contractible semiprimary points which implies an analogon of the cutting procedure \refprimcutth\ (and of \refsignskewsym). Altogether, we deduce $\delti \circ \delti=0$ and the existence of semiprimary Floer homology.
\end{proof}

The difference of $\mcHpr$ and $\mcH_s$ is only noticable on manifolds with genus and their
different properties come apparent in the following example. A brief look at Figure \ref{bad flip} tells us that, on the one hand, a semi-primary point is lost as generator since the move turns $q$ from
(semi-)primary to primary and that, on the other hand, the new arising semi-primary point $s$ is not
contractible. This observation leads to

\begin{Example}
Consider the example of Figure \ref{bad flip} and assume for simplicity, that
only the branches containing $p$ intersect (which is entirely possible if $M$ is
not closed). Then the semi-primary Floer homology is given by
\beqs
\Hti_{-1}(x, \phi) \simeq \Z \quad \mbox{and} \quad \Hti_{*}(x, \phi)=0 \mbox{
otherwise}.
\eeqs
Let $\si(n)=-1$ for $n \in \Z^{>0}$ and $\si(n)=1$ for $n \in \Z^{<0}$. Then we
obtain
\beqs
\Hti_{\si(n) }(x, \phi^n) \simeq \Z^n \quad \mbox{and} \quad \Hti_{*}(x,
\phi^n)=0 \mbox{ otherwise}.
\eeqs
\end{Example}

\begin{proof}
In Figure \ref{bad flip}, $p$ is (semi-)primary and contractible with
$\mu(p)=-1$. $q$ is primary (and contractible) with $\mu(q)=-2$, but not
semi-primary. $s$ is semi-primary and not contractible. $r$ is secondary and not
contractible. Thus we obtain $\Cti_{-1}(x, \phi) = \Z \langle p \rangle$ and
$\Cti_*(x, \phi)=0$ otherwise. The boundary operator is given by $\delti \langle
p \rangle =0$ and thus $\Hti_*(x, \phi)=\Cti_*(x,\phi)$.

Now we consider iterates of the symplectomorphism. For $n \in \Z^{>0}$, we
obtain the complex $\Cti_{-1}(x, \phi^n)=\Span_\Z\{\langle p^0 \rangle, \dots,
\langle p^{n-1} \rangle\} $ and $\Cti_*(x, \phi^n)=0$ otherwise. The boundary
operator is given by $\delti \langle p^l \rangle =0$ for $0 \leq l \leq n-1$ and
thus $\Hti_*(x, \phi^n)= \Cti_*(x, \phi^n) \simeq \Z^n$. For $n \in \Z^{<0}$,
$W^s$ and $W^u$ are exchanged which leads to the change of the Maslov index.
\end{proof}

The computation of primary Floer homology for $\phi$ and $\phi^n$ for the
example in Figure \ref{bad flip} was partially done before \refmixedinvar\ and
in the proof of \refinfinitehom\ and we recall
\beqs
H_{-1}(x, \phi)=C_{-1}(x, \phi)\simeq \Z \quad \mbox{and} \quad H_{-2}(x,
\phi)=C_{-2}(x, \phi)\simeq \Z.
\eeqs
For higher iterates with $n \in \Z^{>0}$, we found $C_{-1}(x, \phi^n)=
\Span_\Z\{\langle p^0 \rangle, \dots, \langle p^{n-1} \rangle\} $ and $C_{-2}(x,
\phi^n)=\Span_\Z\{\langle q^0 \rangle \dots \langle q^{n-1} \rangle\}$ and
$H_{-1}(x, \phi^n)= H_{-1}(x, \phi) \simeq \Z$ and $H_{-2}(x, \phi^n)= H_{-2}(x,
\phi) \simeq \Z$. For negative $n$, we have $H_{2}(x, \phi^n) \simeq \Z$ and
$H_{1}(x, \phi^n) \simeq \Z$.

\vsp

As long as all primary points are also semi-primary, $H_*(x, \phi)$ and
$\Hti_*(x, \phi)$ coincide. The difference becomes apparent as soon as a move
circles around some genus and turns a semi-primary point primary. For primary
Floer homology, this kind of move is in fact secondary. The arising of $r$ and
$s$ turns $q$ from semi-primary to primary which is not noticed by primary Floer
homology.
Semi-primary Floer homology is sensitive to this move since it means the loss of
a generator. 

\vsp

The distinction between homoclinic points $p$ with contractible or
noncontractible loop $[x, p]_u  \cup [x, p]_s$ arise naturally in systems on the
torus or cylinder resp.\ annulus.
Hockett $\&$ Holmes \cite{hockett-holmes} study the existence and impact of such
(semi-primary) homoclinic points on the annulus. If $[x, p]_u  \cup [x, p]_s$ is
contractible they call $p$ {\em non-rotary}. If $[x, p]_u \cup [x, p]_s$ winds
$k$ times around the hole of the annulus, they call $p$ {\em $k$-rotary}.
Noncontractible, semi-primary points therefore fit as 1-rotary orbits in their
framework.

\subsection{Chaotic Floer homology}

The difference of primary Floer homology and semi-primary Floer homology is due
to their different generator sets. In this subsection, we define a version of
homoclinic Floer homology which is based on primary points as generators, but
whose boundary operator is different from the one in primary Floer homology.

\vsp

We want to include some of the nearby chaos in the definition of homoclinic
Floer homology. Before we start, recall some classical results about the
existence of periodic points near a homoclinic tangle. Birkhoff \cite{birkhoff}
proved in 1935 that there is an intricate amount of (mostly high)periodic points
near a homoclinic one which was formalized by Smale's horseshoe.
For periodic points, there is Conley's conjecture which claims the existence of
infinitely many periodic points on certain symplectic manifolds. By now, it has
been established for certain manifolds, cf.\ Ginzburg \cite{ginzburg}, Hingston
\cite{hingston}. 

\vsp

Now we will define a homoclinic Floer homology which takes also periodic points
of the underlying symplectomorphism into account.
Assume $\phi \in \Diff_\om(M)$ and $x \in \Fix(x)$ hyperbolic. Depending on the
iteration number $n \in \Z$, we assign new signs to primary points $p$, $q \in
\mcH(\phi^n, x)$ via
\beqs
\nu_n(p,q):=
\left\{
\begin{aligned}
&m(p,q) && \mbox{if }  \emptyset \neq \mcM(p,q) \ni u , \ \Fix(\phi^n) \cap
\Img(u) = \emptyset \\
& 0 && \mbox{otherwise.}
\end{aligned}
\right.
\eeqs
Set $\nu_n(\langle p \rangle , \langle q \rangle):= \sum_{l \in \Z}\nu_n(p,q^l)$
and define the chain complexes as $\mathscr C_*^{(n)}:= C_*(x,\phi^n;\Z)$. The
boundary operators are
\beqs
\mathscr D^{(n)} : \mathscr C_*^{(n)} \to \mathscr C^{(n)}_{*-1}, \qquad
\mathscr D^{(n)}(\langle p \rangle) := \sum_{\stackrel{\langle q \rangle \in
\mcHprti(\phi^n)}{\nu_n(\langle p \rangle , \langle q \rangle)=1}} \nu_n(\langle
p \rangle , \langle q \rangle) \langle q \rangle
\eeqs
on a generator and are extended to $\mathscr D^{(n)}$ by linearity. 

\begin{Theorem}
\label{thchaoticfh}
It holds $\mathscr D^{(n)} \circ \mathscr D^{(n)}=0$ and $\Hhat_*(x,\phi^n):=
\Hhat_*(x,\phi^n; \Z):={\frac{\ker \mathscr D^{(n)}_*}{\Img \mathscr
D^{(n)}_{*+1}}}$ is called {\em chaotic Floer homology}.
\end{Theorem}

\begin{proof}
Compared to primary Floer homology, chaotic Floer homology also uses primary points as generators, but employs a modified boundary operator which counts `less' digons than the one of primary Floer homology. Thus the gluing procedure  and also the finiteness of the sum in the definition of the boundary operator carry directly over from primary Floer homology. 

And also the cutting procedure is still valid:
Let $n \in \N$ and consider the primary points of $\phi^n$. For $p$, $r \in \mcHpr(\phi^n)$ with $\mu(p,r)=2$, \refprimcutth\ yields the two cutting points $q_s$, $q_u \in \mcHpr(\phi^n)$ since {\em counting or not counting} a di-gon with the new sign $\nu_n(\cdot, \cdot)$ is independent from the {\em existence} of the cutting points.

Thus we only have to prove an analogon of \refsignskewsym\ for the new signs $\nu_n(\cdot, \cdot)$. Consider the possible cutting situations in Figure \ref{signchaoticfh}. If there are no fixed points of $\phi^n$ in the ranges of the involved di-gons as in Figure \ref{signchaoticfh} (a), the signs $\nu_n(\cdot, \cdot)$ coincide with the signs $m(\cdot, \cdot)$ and \refsignskewsym\ holds true. 
\begin{figure}[h]
\begin{center}

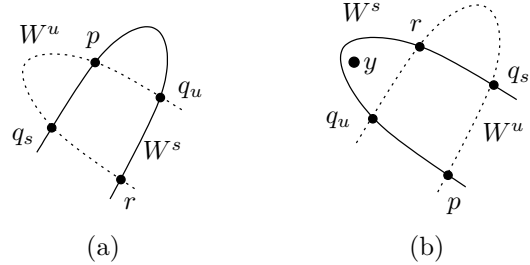

\caption{Signs in chaotic Floer homology.}
\label{signchaoticfh}

\end{center}
\end{figure}
Now assume that there are fixed point(s) in the range(s). For instance, as in Figure \ref{signchaoticfh} (b), let $y \in \Fix(\phi^n)$ lie in the range of all di-gons in $\mcM(p, q_s)$, but not in the range of the di-gons in $\mcM(q_s,r)$. 
We compute
\begin{align*}
\nu_n(p, q_s) \cdot \nu_n(q_s, r) =  0 \cdot \nu_n(q_s,r) & = 0,  \\
   - \nu_n(p,q_u) \cdot \nu_n(q_u,r)    = - \nu_n(p,q_u) \cdot 0 & = 0.
\end{align*}
Other placements of (possibly several) fixed points yield similar calculations. Thus we proved an analogon of \refsignskewsym\ which yields the claim.
\end{proof}

\begin{Remark}
\begin{enumerate}[(1)]
 \item 
Chaotic primary Floer homology is invariant under conjugation. 
\item
The additional condition on the signs renders an invariance discussion for
arbitrary high $n$ futile since the fixed point condition prevents moves.
Invariance makes only sense for fixed $n$ and then one would have to require the
existence of continuations of all involved periodic points.
\end{enumerate}
\end{Remark} 

The main importance of chaotic Floer homology lies in its change under
iteration. Therefore let us consider the dynamics of $n \mapsto \Hhat_*(x,
\phi^n)$ in an explicit example. 

\vsp

\begin{figure}[h]
\begin{center}

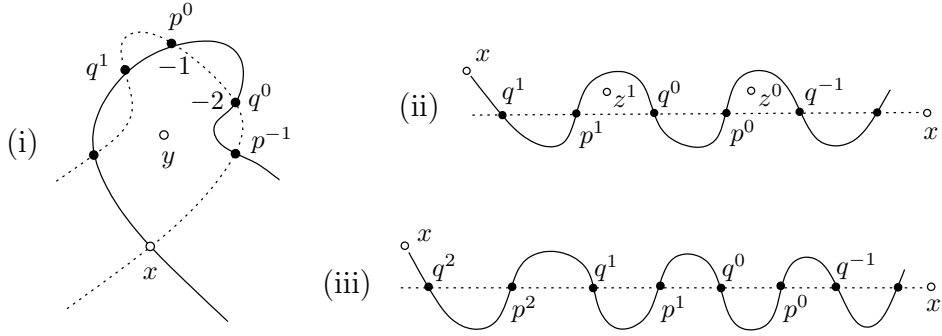

\caption{Chaos near the homoclinic tangle.}

\label{chaos}

\end{center}
\end{figure}

Recall the notation $\phi^n(p)=:p^n$ with $p=p^0$.
Let $\phi \in \Diff_\om(\R^2)$ have the homoclinic tangle sketched in Figure
\ref{chaos} (i) and assume the following additional data: Let $\Fix(\phi)=\{x,
y\}$ and set $\Per_k(\phi)$ to be the set of periodic points whose smallest
period is $k$. Assume $\Per_2(\phi)=\{z^0, z^1\}$ and $\Per_3(\phi)= \emptyset$
and that only the branches containing $p$ intersect.
The homoclinic tangles of $\phi^2$ and $\phi^3$ are drawn in Figure \ref{chaos}
(ii) and (iii) where we have splitted $x$ into two copies. Assume the positions
of $x$, $y$, $z^0$ and $z^1$ as in Figure \ref{chaos}.
Now we compute the chaotic Floer homology for $n \in \{1, 2, 3\}$.

\begin{Example}
Under the above assumption, we obtain
\begin{align*}
\mbox{For } n=1: \qquad &\Hhat_{-1}(x, \phi) \simeq \Z & \mbox{and} \quad\qquad
& \Hhat_{-2}(x, \phi) \simeq \Z ,&&&\\
\mbox{For } n=2: \qquad &\Hhat_{-1}(x, \phi^2) =0 & \mbox{and} \quad\qquad &
\Hhat_{-2}(x, \phi^2)=0,&&& \\
\mbox{For } n=3: \qquad & \Hhat_{-1}(x, \phi^3) \simeq \Z & \mbox{and}
\quad\qquad & \Hhat_{-2}(x, \phi^3)\simeq \Z.&&&	
\end{align*}
\end{Example}

\begin{proof}
For $n=1$ we have $\msC^{(1)}_{-1}= \Z \langle p \rangle$ and $\mathscr
C^{(1)}_{-2}= \Z \langle q \rangle$. The boundary operator is $\msD^{(1)}
\langle p \rangle = -\langle q \rangle +\langle q \rangle =0$ and $\mathscr
D^{(1)} \langle q \rangle =0$. Thus $\Hhat_*(x, \phi) \simeq \msC_*^{(1)}$.

For $n=2$ we obtain $\msC^{(2)}_{-1} = \Z \langle p^0 \rangle \oplus \Z \langle
p^1 \rangle$ and $\msC^{(2)}_{-2} = \Z \langle q^0 \rangle \oplus \Z \langle q^1
\rangle$. The boundary operator is given by $\mcD^{(2)}\langle p^0 \rangle =
\langle q^0 \rangle$ and $\mcD^{(2)} \langle p^1 \rangle = \langle q^1 \rangle$.
This yields $\Hhat_{-1}(x, \phi^2) =0$ and $\Hhat_{-2}(x, \phi^2)=0$.

For $n=3$, there is 
\begin{align*}
& \mathscr C^{(3)}_{-1}= \Z \langle p^0 \rangle \oplus \Z \langle p^1 \rangle 
\oplus \Z \langle p^2 \rangle,
&& \mathscr D^{(3)} \langle p^0 \rangle  = \langle q^0 \rangle - \langle q^2
\rangle , && \mathscr D^{(3)} \langle q^0 \rangle =0, \\
&\mathscr C^{(3)}_{-2}= \Z \langle q^0 \rangle \oplus \Z \langle q^1 \rangle
\oplus \Z \langle q^2 \rangle,
&& \mathscr D^{(3)} \langle p^1 \rangle  = \langle q^1 \rangle - \langle q^0
\rangle , && \mathscr D^{(3)} \langle q^1 \rangle =0 ,\\
&
&& \mathscr D^{(3)} \langle p^2 \rangle  = \langle q^2 \rangle - \langle q^1
\rangle , && \mathscr D^{(3)} \langle q^2 \rangle =0
\end{align*}
and we compute $\Hhat_{-1}(x, \phi^3) \simeq \Z \simeq \Hhat_{-2}(x, \phi^3)$.
\end{proof}

Recall that we calculated the primary Floer homology for the iterates of $\phi$
in \refinfinitehom\ and obtained $H_{-1}(x, \phi^n) \simeq \Z $ and $ H_{-2}(x,
\phi^n) \simeq \Z$ for all $n \in \Z^{>0}$ which was oblivious to the iteration.

\vsp

This simple example demonstrates the properties of chaotic Floer homology very
well. For the higher iterates we know that $ z \in \Fix(\phi)$ implies $ z \in
\Fix(\phi^n)$ and $ z \in \Fix (\phi^l) \cap \Fix(\phi^k) $ implies $ z \in
\Fix(\phi^{k \cdot l})$. But apart from those `old ones', new fixed points might
or will arise according to our discussion above.

\vsp

The dynamical behaviour of $n \mapsto H_*^{\Fix}(x, \phi^n)$ leads to a
symplectic zeta function
\beqs
\ze_{x,\phi}(z):= \exp\left(\sum_{n=1}^\infty \frac{\chi(\Hhat_*(x,\phi^n))}{n}
z^n\right )
\eeqs
where $\chi(H^{\Fix}_*(x,\phi^n))$ denotes the Euler characteristic of
$\Hhat_*(x,\phi^n)$. Zeta function have been studies a lot in number theory,
algebraic geometry and dynamics. For an overview see for instance Fel'shtyn
\cite{felshtyn, felshtyn2}.

\begin{Question}
\begin{enumerate}[a)]
\item
Is there $\phi$ such that $\ze (x, \phi)$ is rational? If yes, which $\phi$?
\item
Is there a relation to the classical (symplectic) zeta function?
\item
Are there applications to Nielsen theory and Reidemeister torsion whose relation
to dynamical zeta \fct s is described in Fel'shtyn \cite{felshtyn}?
\end{enumerate}
\end{Question}

\end{document}